\documentclass[11pt]{amsart} 
\usepackage{amscd,amssymb,verbatim}
\usepackage{hyperref}

\newtheorem{thm}{Theorem}[section]
\newtheorem{cor}[thm]{Corollary} 
\newtheorem{lem}[thm]{Lemma}
\newtheorem{prop}[thm]{Proposition} 
\newtheorem{Def}[thm]{Definition} 
\newtheorem{ex}[thm]{Example}

\numberwithin{equation}{section}

\newcommand{\ep}{\varepsilon}
\newcommand{\al}{\alpha}

\newcommand{\vp}{\varphi}

\newcommand{\bs}{\backslash}

\newcommand{\ol}{\overline}
\newcommand{\ti}{\widetilde}
\newcommand{\wh}{\widehat}

\newcommand{\diam}{\operatorname{diam}}
\newcommand{\inte}{\operatorname{int}}
\newcommand{\Lip}{\operatorname{Lip}}

\newcommand{\N}{{\mathbb N}}
\newcommand{\R}{{\mathbb R}}

\newcommand{\Q}{{\mathbb Q}}

\newcommand{\cD}{{\mathcal D}}
\newcommand{\cC}{{\mathcal C}}
\newcommand{\cF}{{\mathcal F}}

\begin{document}

\title{Nonlinear biseparating maps}

\begin{abstract}
An additive map $T$  acting between spaces of vector-valued functions is said to be biseparating if $T$ is  a bijection so that $f$ and $g$ are disjoint if and only if $Tf$ and $Tg$ are disjoint.
Note that an additive bijection retains $\Q$-linearity.
For a general nonlinear map $T$, the definition of biseparating given above turns out to be too weak to determine the structure of $T$.
In this paper, we propose a revised definition of biseparating maps for general nonlinear operators acting between spaces of vector-valued functions, which coincides with the previous definition for additive maps.  Under some mild assumptions on the function spaces involved, it turns out that a map is biseparating if and only if it is  locally determined.   We then delve deeply into some specific function spaces -- spaces of continuous functions, uniformly continuous functions and Lipschitz functions -- and characterize the  biseparating maps acting on them.
As a by-product, certain forms of automatic continuity are obtained.  We also prove some finer properties of biseparating maps in the cases of uniformly continuous and Lipschitz functions.
\end{abstract}

\author{Xianzhe Feng}
\author{Denny H.\ Leung}
\address{Department of Mathematics, National University of Singapore, Singapore 119076}

\email{a0120755@nus.edu.sg, dennyhl@u.nus.edu}


\maketitle

\tableofcontents

\section{Introduction}\label{S0}

Many function spaces of practical interest are also algebras and/or lattices.  An important example is $C(X)$, the space of continuous real-valued functions on a compact Hausdorff space. 
Hence, a natural problem in the course of understanding the structure of the space $C(X)$ is to  characterize the algebraic and/or lattice isomorphisms on it.
The classical solutions  were given by Gelfand and Kolmogorov \cite{GK} and Kaplansky \cite{K} respectively.
An in depth exposition of investigations into the algebraic structure of $C(X)$ and much more can be found in the classic monograph \cite{GJ}.
Subsequent research has tied these two strands together in the form of the disjointness structure of the function space $C(X)$.
Specifically, algebraic or lattice homomorphisms are disjointness preserving; they map disjoint functions to disjoint functions.  An algebraic or lattice isomorphism is biseparating; that is, it is a bijection $T$ so that both $T$ and $T^{-1}$ are disjointness preserving.  Moreover, generalization to disjointness preserving or biseparating maps allows for extension to function spaces that are neither algebras nor lattices, and even to vector-valued functions.
Copius research has been devoted to the study of disjointness preserving and biseparating maps on various function spaces; see, e.g., \cite{A}-\cite{AJ2}, \cite{BBH}, \cite{GJW}, \cite{HBN}-\cite{J-V W}, \cite{L}.
As far as the authors are aware of, the study  of biseparating maps thus far has been confined to linear or at least additive maps.  Since additive bijective maps are $\Q$-linear, such maps are not far removed from the linear world.
In this paper, we initiate the study of general nonlinear biseparating maps on spaces of vector-valued functions.
The following example shows that the definition of ``biseparating'' needs to be adjusted in order to obtain meaningful results.

\bigskip

\noindent{\bf Example}.  Let $A$ be the set of all functions $f\in C[0,1]$ so that the set $\{t\in [0,1]: f(t) \neq 0\}$ is dense in $[0,1]$.
Let  $T:C[0,1]\to C[0,1]$ be a map such that $T$ maps $A$ bijectively onto itself and that $Tf = f$ if $f\neq A$.
Then $T$ is a bijection so that $f\cdot g = 0 \iff Tf\cdot Tg = 0$.

\bigskip

The example shows that the definition of ``biseparating'' used for linear or additive maps is too weak when applied to general nonlinear maps.  In the next section, we propose a revised definition of ``biseparating'' for nonlinear maps.  The definition reduces to the usual one for additive maps.  Moreover, with the revised definition, a satisfactory theory of nonlinear biseparating maps arise, subject to some mild assumptions.  See the paragraph preceding Lemma \ref{l1.21}.  The theory of nonlinear biseparating maps is somewhat related to the theory or order isomorphisms developed in \cite{LT}.  It also partly generalizes the notion of ``nonlinear superposition operators''.  We refer to \cite{AZ} for a comprehensive study of the latter types of operators.

Let us give an overview of the content of the paper.  As mentioned, the definition of ``nonlinear biseparatimg maps'' is given in \S \ref{s1}. Under the mild assumptions of ``basic'' and ``compatible'', the fundamental characterization theorem of nonlinear biseparating maps (Theorem \ref{t5}) is obtained.  The theorem shows that a nonlinear bijective operator is biseparating if and only if it is  ``locally determined''.  For an exposition of some applications of locally determined operators to operator functional equations, particularly on spaces of differentiable functions, refer to \cite{KM}.
The characterization theorem applies in particular to a number of  familiar (vector-valued) function spaces such as spaces of continuous, uniformly continuous, Lipschitz and differentiable functions.
We would like to point out a general resemblance of Theorem \ref{t5} with the fundamental characterization theorem for ``nonlinear order isomorphisms'' \cite[Theorem 2.11]{LT}.
Indeed, our study of nonlinear biseparating maps is motivated and informed by the study of nonlinear order isomorphisms at various points.  However, the lack of an order makes many of the arguments more difficult in the present case, especially for uniformly continuous and Lipschitz functions.
For further information on nonlinear order isomorphisms on function spaces, we refer to \cite{LT} and the references therein.

\S \ref{s2} studies nonlinear biseparating maps between spaces of vector-valued continuous or bounded continuous functions.
One of the main results is Theorem \ref{t2.8}, which shows that if $X,Y$ are realcompact spaces and $E,F$ are Hausdorff topological vector spaces, and there is a biseparating map $T:C(X,E)\to C(Y,F)$, then $X$ and $Y$ are homeomorphic.  With reference to the classical Gelfand-Kolmogorov and Kaplansky theorems, one sees that one needs rather much less than the full algebraic or lattice structure of $C(X)$ to determine the topology of $X$.

From \S\ref{s3} onwards, we focus on metric spaces $X$ and $Y$.
In the course of \S\ref{s3} and \S\ref{s4}, full representations of biseparating maps between spaces of continuous, uniformly continuous and Lipschitz functions defined on metric spaces are obtained.  See Propositions \ref{p4.2}, \ref{p4.3} and \ref{p4.4}. \S\ref{s5} revisits spaces of continuous functions, this time defined on metric spaces.
Complete characterizations of nonlinear biseparating maps are obtained; see Theorems \ref{t5.4} and \ref{t5.5}.
We also prove an automatic continuity result Theorem \ref{t5.6}.

\S7 is concerned with nonlinear biseparating maps between spaces of uniformly continuous functions.
Characterization of such maps is carried out in two stages.  First it is shown that a biseparating map induces a uniform homeomorphism of the underlying metric spaces.  The second part involves solving the ``section problem'': determining the maps $\Xi$ so that $Sf(x) = \Xi(x,f(x))$ is uniformly continuous whenever the input function $f$ is uniformly continuous.  Refer to Theorems \ref{t6.7.1} and \ref{t6.7.2}.
From these characterization theorems, one can also obtain an automatic continuity result (Theorem \ref{t6.9}).
A classical result of Atsuji \cite{At} and Hejcman \cite{H}, rediscovered in \cite{O'F}, states that all uniformly continuous functions on a metric space $X$ are bounded if and only if $X$ is Bourbaki bounded (see definition in \S7.2).  Theorem \ref{t6.10} generalizes this result.  It shows that there is a biseparating map from $U(X,E)$ onto a space  $U_*(Y,F)$ (the space of bounded uniformly continuous functions) if and only if $X$ is Bourbaki bounded.

\S \ref{s8} focuses on spaces of Lipschitz functions. First it is shown that we may reduce to considering spaces $\Lip(X,E)$, where $X$ is bounded metric (Proposition \ref{p6.2}).
Making use of the Baire Catergory Theorem and some intricate combinatorial arguments, it is then shown that a biseparating map between vector-valued Lipschitz spaces defined on bounded complete metric spaces induces a Lipschitz homeomorphism between the underlying metric spaces (Theorem \ref{t7.5}).
Next, the section problem for Lipschitz functions is solved (Theorem \ref{t7.7}), which enables the characterization of nonlinear biseparating maps between spaces of Lipschitz functions (Theorem \ref{t7.8}).
Suppose that $\Xi$ is a ``Lipschitz section'', i.e., the function $\Xi(x,f(x))$ is a Lipschitz function of $x$ whenever $f$ is Lipschitz.  It is known that even if $x_0$ is an accumulation point, the function $\Xi(x_0,\cdot)$ need not be continuous with respect to the second variable.  Nevertheless, exploiting the Baire Category Theorem, we show that $\Xi(x_0,\cdot)$ is continuous on an open dense set if $x_0$ is an accumulation point (Theorem \ref{t7.11}).

The final section \S \ref{s9} determines the biseparating maps that act between a space of uniformly continuous functions on the one hand and a space of Lipschitz functions on the other.  The main results (Theorems \ref{t6.6} and \ref{t6.7}) show that there is a certain rigidity, so that the existence of such maps imply very strong conditions on the underlying metric spaces.

To end this introduction, we note that linear or nonlinear biseparating maps acting between spaces of differentiable functions seem to be rather more difficult to deal with.  A notable achievement in this regard is the paper by Araujo \cite{A3}.  We intend to address some of the problems raised therein in a future paper.

\section{Generalities}\label{s1}

Let $X, Y$ be sets and let $E, F$ be (real or complex) vector spaces.  
Suppose that $A(X,E)$ is a vector subspace of $E^X$ and $A(Y,F)$ is a vector subspace of $F^Y$.
If $f\in A(X,E)$, let the {\em carrier} of $f$ be the set 
\[C(f) = \{x\in X: f(x) \neq 0\}.\]   
Set $\cC(X) =  \{C(f):f\in A(X,E)\}$.
For functions $f,g,h\in A(X,E)$, say that $f$ and $g$ are {\em disjoint with respect to} $h$, $f\perp_h g$, if $C(f-h) \cap C(g-h) = \emptyset$. We abbreviate $\perp_0$ as $\perp$.
The {\em support} of a function $f\in A(X,E)$ is the set
\[ \wh{C}(f) = X\bs \bigcup\{C(g): g\in A(X,E), f\perp g\} = X\bs \bigcup\{C\in \cC(X): C \cap C(f) = \emptyset\}.\]
% Similarly for functions in $A(Y,F)$. 
Obviously $C(f) \subseteq \wh{C}(f)$. Furthermore, if $f_1,f_2\in A(X,E)$ and $f_1=f_2$ on $C(f)$, then $f_1 = f_2$ on $\wh{C}(f)$.
 Set $\cD(X) = \{\wh{C}(f): f\in A(X,E)\}$.  Similar definitions apply to $A(Y,F)$.
 A map $T:A(X,E) \to A(Y,F)$  is {\em biseparating} if it is a bijection and for any $f,g,h \in A(X,E)$,
\[ f\perp_hg \iff Tf\perp_{Th}Tg.\]
%Say that $A(X,E)$ is {\em completely regular} if for any $f,g\in A(X,E)$ and any $x\in C(g) \bs \wh{C}(f)$, there exists $u \in A(X,E)$ so that $u\perp f$ and $x\in C(u) \subseteq C(g)$.
For the remainder of the section, let $T:A(X,E)\to A(Y,F)$ be a  given biseparating map.
The following proposition, although simple to state and easy to prove, turns out to be  key to understanding biseparating maps.

\begin{prop}\label{p1}(Araujo's Lemma, cf \cite[Lemma 4.2]{A})
 %, where $A(X,E)$ and $A(Y,F)$ are completely regular. 
For any $f,g,h\in A(X,E)$, 
\[ \wh{C}(f-h) \subseteq \wh{C}(g-h) \iff \wh{C}(Tf-Th) \subseteq \wh{C}(Tg-Th).\]
\end{prop}

\begin{proof}
Suppose that $\wh{C}(f-h) \subseteq \wh{C}(g-h)$.
Assume that there exists $z\in \wh{C}(Tf-Th) \bs \wh{C}(Tg-Th)$.
There exists $v\in A(Y,F)$ so that $v\perp Tg-Th$ and $z\in C(v)$.
Since $z\in \wh{C}(Tf-Th)$, $v\not\perp Tf-Th$.  %Thus there exists $y\in C(v) \cap C(Tf-Th)$.
%Clearly, $y \in C(Tf-Th)\bs \wh{C}(Tg-Th)$.
%As $A(Y,F)$ is completely regular, there exists $w\in A(Y,F)$ so that $w \perp Tg -Th$ and $y \in C(w)\subseteq C(Tf-Th)$.
Set $u = T^{-1}(v+Th) \in A(X,E)$.  Then $v = Tu - Th$.  Hence
\[ Tu -Th = v \perp Tg-Th \implies Tu \perp_{Th} Tg \implies u\perp_h g\implies u-h \perp g-h.\]
Therefore,
\[ C(u-h) \subseteq (\wh{C}(g-h))^c \subseteq (\wh{C}(f-h))^c \implies u-h \perp f-h \implies u\perp_h f.\]
It follows that 
\[ Tu \perp_{Th}Tf \implies v = Tu-Th \perp Tf-Th.\]
This contradicts that fact that $v\not\perp Tf-Th$.  This completes the proof for the forward implication ``$\implies$".  The reverse implication follows by symmetry.
\end{proof}

%Let $\cC(X)$, respectively, $\cD(X)$, consist all subsets of $X$ of the form $C(f)$, respectively, $\wh{C}(f)$, for some $f\in A(X,E)$.  Define $\cC(Y)$ and $\cD(Y)$ similarly.

\begin{prop}\label{p2}
Let $f\in A(X,E)$ be given.  The map \[\theta_f: \wh{C}(h) \mapsto \wh{C}(T(f+h) -Tf)\] is a well-defined bijection from $\cD(X)$ onto $\cD(Y)$ that preserves set inclusion.  For any $f,g\in A(X,E)$ and any $U\in\cD(X)$, $f= g$ on $U$ if and only if $Tf = Tg$ on $\theta_f(U)$.
\end{prop}

\begin{proof}
By Proposition \ref{p1}, $\wh{C}(h_1) = \wh{C}(h_2)$ if and only if 
\[ \wh{C}((h_1+f)-f) = \wh{C}((h_2+f)-f) \iff \wh{C}(T(h_1+f) - Tf) = \wh{C}(T(h_2+f)-Tf).\]
This shows that the map $\theta_f$ is well-defined and injective.
Since any $g\in A(Y,F)$ can be written in the form $T(f+h) - Tf$ with $h = T^{-1}(g+Tf) -f$, $\theta_f$ is surjective.  It follows from Proposition \ref{p1} that  $\theta_f$ preserves set inclusion.

Finally, suppose that $U = \wh{C}(h) \in \cD(X)$.
Then $f= g$ on $U$ if and only if $g-f \perp h = (f+h) -f$, which in turn is equivalent to the fact that $Tg-Tf \perp T(f+h) - Tf$.  The last statement is easily seen to be equivalent to the fact that $Tg-Tf = 0$ on $\theta_f(\wh{C}(h))$.
\end{proof}

The idea behind Proposition \ref{p2} is that a biseparating map gives rise to a collection of ``set movers'' $\theta_f$.  
In order to make  the set mover $\theta_f$  independent of the function $f$, we impose two conditions on the function space $A(X,E)$.
Say that $A(X,E)$ is 
\begin{enumerate}
\item {\em basic} if whenever $x\in C_1\cap C_2$ for some $C_1, C_2\in \cC(X)$, then there exists $C\in \cC(X)$ so that $x\in C \subseteq C_1\cap C_2$; 
\item {\em compatible} if for any $f\in A(X,E)$, any $D\in \cD(X)$ and any point $x\notin D$, there exist $g\in A(X,E)$ and $C\in \cC(X)$ so that $x\in C$ and 
\[ g = \begin{cases} f &\text{on $C$},\\
                     0 &\text{on $D$}.\end{cases}\]
 \end{enumerate}

\begin{lem}\label{l1.21}
Suppose that $A(Y,F)$ is basic.  If $f, g\in A(Y,F)$ and $V\in \cD(Y)$ are such that $f = g$ on $V$ and $V\subseteq \wh{C}(g)$, then $V \subseteq \wh{C}(f)$.
\end{lem}

\begin{proof}
Assume otherwise.  There is a point $y_1\in V$ so that $y_1\notin \wh{C}(f)$.
There exists $u\in A(Y,F)$ so that $y_1\in C(u)$ and $u\perp f$. Say $V = \wh{C}(v)$.
Since $y_1\in \wh{C}(v)$, $u\not\perp v$.  As $A(Y,F)$ is basic, there exists $C\in \cC(Y)$ so that $\emptyset\neq C\subseteq C(u)\cap C(v)$.
If $y\in C$, then $y \in C(u)$ and hence $f(y) = 0$.  Moreover, $y\in C(v) \subseteq V$ and hence $g(y) = f(y) = 0$.  This proves that $C\cap C(g) = \emptyset$.  Since $C\in \cC(Y)$, it follows that $C\cap \wh{C}(g) = \emptyset$.
This is impossible since $C$ is a nonempty subset of $C(v)\subseteq \wh{C}(v) = V\subseteq \wh{C}(g)$.
\end{proof}

\begin{prop}\label{p3}
Assume that $A(Y,F)$ is basic. Suppose that $f, g\in A(X,E)$ and $U \in \cD(X)$.  If $f = g$ on $U$, then $\theta_f(U) = \theta_g(U)$.
\end{prop}

\begin{proof}
Let $U = \wh{C}(h)$.  Since $f= g$ on $U\supseteq C(h)$, $C(h)\subseteq C(f+h-g)$.
Hence $U \subseteq \wh{C}(f+h- g)$.
Thus 
\[ \theta_g(U) \subseteq \theta_g(\wh{C}(f+h- g)) = \wh{C}(T(f+h) - Tg)).\]
By Proposition \ref{p2}, $Tf = Tg$ on $\theta_g(U)$.  In other words,
\[ T(f+h) - Tf = T(f+h)-Tg \text{ on } \theta_g(U) \subseteq \wh{C}(T(f+h) - Tg)).\]
Therefore, by Lemma \ref{l1.21}, \[\theta_g(U) \subseteq \wh{C}(T(f+h) - Tf) = \theta_f(U).\]
The reverse inclusion follows by symmetry.
\end{proof}

\begin{prop}\label{p4}
Assume that $A(Y,F)$ is both basic and compatible. Let $f,g\in A(X,E)$ and let $U$ be  a set  in $\cD(X)$.  Then $\theta_g(U) = \theta_f(U)$.
\end{prop}

\begin{proof} 
Suppose that $U = \wh{C}(h)$ and that there exists $y \in \theta_f(U) \bs \theta_g(U)$.
Then there exists $C\in \cC(Y)$ so that $y \in C$ and that $C \cap \theta_g(U) = \emptyset$.
Since $y\in \theta_f(U)$,  $C \cap C(T(f+h)-Tf) \neq \emptyset$.
Using the fact that $A(Y,F)$ is basic, there exist  $C'\in \cC(Y)$ and $z\in Y$ so that 
\[z\in  C' \subseteq C \cap C(T(f+h)-Tf).\]
In particular, $z\notin  \theta_g(U)$ and $\theta_g(U) \in \cD(Y)$.
Use the compatibility of $A(Y,F)$ to choose $v\in A(Y,F)$ and $C''\in \cC(Y)$ so that $z\in C''$ and that 
\[ v = \begin{cases}  Tf-Tg  &\text{on $C''$,}\\
0 &\text{on $\theta_g(U)$}.\end{cases}\]
Since $A(Y,F)$ is basic, we may also assume that $C'' \subseteq C'$.
Set  $k = v+Tg$.  We have  $k= Tg$ on $\theta_g(U)$.  By Proposition \ref{p2}, $T^{-1}k = g$ on $U$.
Say $C'' = C(w)$.  Then $k = Tf$ on $\wh{C}(w)\subseteq \theta_f(U)$.  Thus Proposition \ref{p2} implies that 
\[ T^{-1}k = f  \text{ on } (\theta_f)^{-1}(\wh{C}(w)) \subseteq U.
\]
It follows that $f = g$ on the  set $(\theta_f)^{-1}(\wh{C}(w))\in \cD(X)$.
By Proposition \ref{p3}, $\wh{C}(w) = \theta_g((\theta_f)^{-1}(\wh{C}(w)) \subseteq \theta_g(U)$.
Hence $z\in C(w)\subseteq \theta_g(U)  =\emptyset$, contrary to the choice of $z$.
This proves that $\theta_f(U) \subseteq \theta_g(U)$.  The reverse inclusion follows by symmetry.
\end{proof}

%\noindent{\bf Remark}.  For any $g\in A(Y,F)$, we may apply Proposition \ref{p2} to $T^{-1}$ to see that the map
%%\[ \sigma_g(V) = \wh{C}(T^{-1}(g+h) - T^{-1}g), V = \wh{C}(h)\]
%is well defined.
%An easy calculation shows that 
%$(\theta_f)^{-1} = \sigma_{Tf}$.

%\bigskip

%Suppose that, instead of $A(Y,F)$, it is assumed that $A(X,E)$ is both basic and compatible.  Let $f,g\in A(X,E)$ and $U\in\cD(X)$.
%Then $V = \theta_f(U) \in \cD(Y)$.
%Apply Proposition \ref{p4} to see that $\sigma_{Tg}(V) = \sigma_{Tf}(V)$.
%Thus
 %$(\theta_g)^{-1}(V) = (\theta_f)^{-1}(V) = U$, which implies that $\theta_f(U) = V = \theta_g(U)$.
 
We now obtain the fundamental description of biseparating maps  from  the foregoing propositions.

\begin{Def}\label{d2.1}
Retain the notation above.  A bijection $T: A(X,E)\to A(Y,F)$ is {\em locally determined} if there is a bijection $\theta:\cD(X)\to\cD(Y)$, preserving set inclusions, so that 
for any $f,g\in A(X,E)$ and any $U \in \cD(X)$, $f = g$ on $U$ if and only if $Tf = Tg$ on $\theta(U)$
\end{Def}

\begin{lem}\label{l2.0}
Assume that $A(X,E)$ is basic.  Let $T:A(X,E)\to A(Y,F)$ be locally detemined, with a map $\theta$ as given in Definition \ref{d2.1}.  If $g,h \in A(X,E)$, then $C(Tg-Th) \subseteq \theta(\wh{C}(g-h))$.
\end{lem}

\begin{proof}
Suppose that $y \notin \theta(\wh{C}(g-h))$. Choose  $v_1\in A(Y,F)$ so that 
$\wh{C}(v_1) = \theta(\wh{C}(g-h))$.
There exists $v_2\in A(Y,F)$ such that $y\in C(v_2)$ and $v_1\perp v_2$.
Let $u_1 = g-h$ and $u_2\in A(X,E)$ be such that $\wh{C}(u_2) = \theta^{-1}(\wh{C}(v_2))$.

We claim that $u_1\perp u_2$.
Otherwise, there exists nonempty $C\in \cC(X)$ so that $C\subseteq C(u_1) \cap C(u_2)$.
Hence 
\[ \theta(\wh{C}) \subseteq \theta(\wh{C}(u_1))\cap \theta(\wh{C}(u_2)) = \wh{C}(v_1)\cap \wh{C}(v_2).\]
Let $\theta(\wh{C}) = \wh{C}(v)$ for some nonzero $v\in A(Y,F)$.
Since $v_1\perp v_2$, $C(v_1) \cap \wh{C}(v_2) = \emptyset$.
Therefore, 
\[C(v_1) \cap  C(v) \subseteq C(v_1) \cap \wh{C}(v_2) =\emptyset.
\]
Hence $\wh{C}(v_1)\cap {C}(v) = \emptyset$.  But $\emptyset \neq C(v) \subseteq \wh{C}(v) = \theta(\wh{C}) \subseteq \wh{C}(v_1)
$.  Thus we have a contradiction. This proves the claim.

From the claim, $g=h$ on $\wh{C}(u_2)$.  Hence $Tg = Th$ on $\theta(\wh{C}(u_2)) = \wh{C}(v_2)$.  In particular, $y\notin C(Tg-Th)$.  This completes the proof of the lemma.
\end{proof}

\begin{thm}
\label{t5}
Suppose that  $A(X,E)$ and  $A(Y,F)$ are both basic and compatible.  A bijection  $T:A(X,E)\to A(Y,F)$ is a  biseparating map if and only if it is locally determined.
\end{thm}

\begin{proof}
Assume that $T$ is biseparating.  Take any $f\in A(X,E)$ and let $\theta = \theta_f$.  By Proposition \ref{p4}, $\theta$ is independent of the choice of $f$.
The properties enunciated for $\theta$ now follow from the same ones for $\theta_f$ by Proposition \ref{p2}.  Therefore, $T$ is locally determined.

Conversely, suppose that $T$ is locally determined.  Let $f,g,h\in A(X,E)$ be such that $f\perp_h g$.
Then $f =h$ on $\wh{C}(g-h)$.
Therefore, $Tf = Th$ on $\theta(\wh{C}(g-h))$.
By Lemma \ref{l2.0}, $Tf = Th$ on $C(Tg-Th)$.  Thus $Tf \perp_{Th} Tg$.
Since the same argument applies to $T^{-1}$, we see that $T$ is biseparating.
\end{proof}

Let us give some examples of function spaces that are both basic and compatible.  The verifications are simple and will be omitted.  
%Recall that a functional $\|\cdot\|: E \to [0,\infty)$ on a vector space $E$ is a {\em quasi-norm} if 
%\begin{enumerate}
%\item $\|x\| > 0$ if $x\neq 0$.
%\item $\|\al x\| = |\al| \|x\|$ for any scalar $\al$ and any $x\in E$.
%\item There exists $1\leq C<\infty$ so that $\|x+y\| \leq C\max\{\|x\|,\|y\|\}$ for all $x,y\in E$.
%\end{enumerate}
%A quasinorm $\|\cdot\|$ induces a translation invariant  metric $d(x,y) = \|x-y\|$ under which $E$ is a metric linear space .
If $G$ is a Banach space, a {\em bump function} on $G$ is a nonzero real-valued function on $G$ with bounded support.

\begin{ex}\label{e1.7}
Let $A(X,E)$ be any of the spaces described below.  Then $A(X,E)$ is both basic and compatible.
Furthermore, $\wh{C}(f) = \ol{C(f)}$ for any $f \in A(X,E)$. 
\begin{enumerate}
\item Let $X$ be a Hausdorff completely regular topological space and let $E$ be a nonzero Hausdorff topological vector space.
$A(X,E) = C(X,E)$ or $C_*(X,E)$, the subspace consisting of all bounded functions in $C(X,E)$. (By a bounded function, we mean a function whose image is bounded in $E$, i.e, can be absorbed by any neighborhood of $0$.)
\item Let $X$ be a metric space and let $E$ be  a normed space.  Take $A(X,E)$ to be one of the following spaces.  $U(X,E)$, the space of all $E$-valued uniformly continuous functions on $X$; or $\Lip(X,E)$, the space of all $E$-valued Lipschitz functions on $X$; or  $U_*(X,E)$, respectively, $\Lip_*(X,E)$, the bounded functions in $U(X,E)$ and $\Lip(X,E)$ respectively.
\item Let $X$ be an open set in a Banach space $G$, $p\in \N\cup\{\infty\}$, and let $E$ be a normed space. Assume that $G$ supports a $C^p$ bump function and take  $A(X,E) = C^p(X,E)$, the space of all $p$-times continuous differentiable $E$-valued functions on $X$.  Alternatively, let  $A(X,E) =  C^p_*(X,E)$, the subspace of $C^p(X,E)$  so that 
$D^kf$ is bounded on $X$ for $k\in \{0\}\cup\N$, $k\leq p$.  In the latter case, assume that $G$ supports a $C^p_*$ bump function.
%\item Let $X$ be an open set in a Banach space $G$ and let $E$ be a normed space.  For $p\in \N\cup\{\infty\}$,
%let $C^p_*(X,E)$ be the space of all $p$-times continuously differentiable $E$-valued functions $f$ on $X$ so that 
%$D^kf$ is bounded on $X$ for $k\in \{0\}\cup\N$, $k\leq p$.  %Abbreviate $C^p_*(X,\R)$ as $C^p_*(X)$.
%Assume that $G$ supports a real-valued $C^p$ bump function.  Take $A(X,E) = C^p_*(X,E)$.
\end{enumerate}
\end{ex}

\section{Spaces of continuous functions}\label{s2}

In this section, let $X$ and $Y$ be Hausdorff completely regular topological spaces and $E$ and $F$ be nontrivial Hausdorff topological vector spaces.
Take $A(X,E) = C(X,E)$ or $C_*(X,E)$ and $A(Y,F) = C(Y,F)$ or $C_*(Y,F)$.
Let $T:A(X,E)\to A(Y,F)$ be a biseparating map.
Without loss of generality, we may assume that $T0 =0$.
We retain the notation in \S \ref{s1}.
The main aim is to derive topological relationship between $X$ and $Y$ based on the map $T$.
Recall that a Hausdorff completely regular topological space $X$ has a ``largest" compactification, namely the Stone-\v{C}ech compactification $\beta X$.    
If $V$ is a set in $Y$, denote its closures in $Y$ and $\beta Y$ by $\ol{V}$ and $\ol{V}^{\beta Y}$ respectively.
By  Example \ref{e1.7},  $\wh{C}(f) = \ol{C(f)}$ for any $f\in A(X,E)$.

\begin{lem}\label{l2.0.1}
Let $U_i, i=1,2$, be open sets in $\beta X$ so that $U_i\cap X\in \cC(X)$ and that $\ol{U_1}^{\beta X}\cap \ol{U_2}^{\beta X}= \emptyset$.  Then $\ol{\theta(\ol{U_1\cap X})}^{\beta X} \cap \ol{\theta(\ol{U_2\cap X})}^{\beta X} = \emptyset$.
\end{lem}

\begin{proof}
Let $v$ be a nonzero vector in $F$.  There exists $h\in C_*(X)$ so that $h(x) =1$ for all $x\in U_1\cap X$ and $h(x) = 0$ for all $x\in U_2\cap X$. The function $f = h\cdot T^{-1}(1\otimes v)$ belongs to $A(X,E)$.  By Lemma \ref{t5}, $Tf = 1\otimes v$ on $\theta(\ol{U_1\cap X})$ and $Tf = 0$ on $\theta(\ol{U_2\cap X})$.
Since $F$ is a Hausdorff topological vector space, it is completely regular. 
So there exists a continuous function $g: F\to \R$ so that $g(v) \neq g(0)$.
Set $k =g\circ (Tf):Y\to \R$.  Then $k$ is continuous on $Y$ has hence has a continuous extension $\ti{k}:\beta Y\to \R\cup \{\infty\}$.
Now $k(y) = g(v)$ for all $y\in \theta(\ol{U_1\cap X})$ and $k(y) = g(0)$ for all $y\in \theta(\ol{U_2\cap X})$.
By continuity of $\ti{k}$, the sets $\ol{\theta(\ol{U_i\cap X})}^{\beta X}$, $ i=1,2$, must be disjoint.
\end{proof}

For any $x\in \beta X$, let $\cF_x$ be the family of all open neighborhoods $U$ of $x$ in $\beta X$ so that $U\cap X \in \cC(X)$.
Define $\cF_y$ similarly for $y\in \beta Y$.
We will use the following fact which is easily deduced from the Urysohn Lemma.
Let $U$ be an open neighborhood of a point $x$ in $\beta X$. There exists an open neighborhood $V\in \cF_x$  so that $V\subseteq U$.

\begin{lem}\label{l2.0.2}
Let $x\in \beta X$, $y\in \beta Y$. Then $y\in\ol{\theta(\ol{U\cap X})}^{\beta Y}$ for all $U \in \cF_x$ if and only if 
$x\in\ol{\theta^{-1}(\ol{V\cap Y})}^{\beta X}$ for all $V \in \cF_y$.
\end{lem}

\begin{proof}
Assume that $y\in\ol{\theta(\ol{U\cap X})}^{\beta Y}$ for all $U \in \cF_x$.
Suppose that there exists $V\in \cF_y$ so that $x\notin\ol{\theta^{-1}(\ol{V\cap Y})}^{\beta X}$.
Choose $U \in \cF_x$ so that $\ol{U}^{\beta X} \cap \ol{\theta^{-1}(\ol{V\cap Y})}^{\beta X} = \emptyset$.
Note that by definition of $\theta^{-1}$, $\theta^{-1}(\ol{V\cap Y}) = \ol{W_0}$ for some $W_0\in \cC(X)$.
Express $W_0 = W\cap X$ for some open set $W\in \beta X$.
Then $\ol{U}^{\beta X} \cap \ol{W}^{\beta X} = \emptyset$.
By Lemma \ref{l2.0.1}, $\ol{\theta(\ol{U\cap X})}^{\beta X} \cap \ol{\theta(\ol{W\cap X})}^{\beta X} = \emptyset$. By choice, $y \in \ol{\theta(\ol{U\cap X})}^{\beta Y}$.
Also, since $V\in \cF_y$, 
\[ y \in \ol{{V\cap Y}}^{\beta Y}= \ol{\theta(\ol{W\cap X})}^{\beta Y}. \]
This contradicts the disjointness of $\ol{\theta(\ol{U\cap X})}^{\beta X}$ and $\ol{\theta(\ol{W\cap X})}^{\beta X}$ and  completes the proof of the ``only if'' part.  The reverse implication follows by symmetry.
\end{proof}

\begin{lem}\label{l2.0.3}
For any $x\in \beta X$, the set $\bigcap\{\ol{\theta(\ol{U\cap X})}^{\beta Y}: U \in \cF_x\}$ has exactly one point in $\beta Y$.
\end{lem}

\begin{proof}
Let $U_i$, $i=1,\dots, n$, be sets in $\cF_x$.
Then $\bigcap^n_{i=1}U_i$ is an open neighborhood of $x$ in $\beta X$. By the remark after Lemma \ref{l2.0.1}, 
there exists $U\in \cF_x$ so that $U \subseteq \bigcap^n_{i=1}U_i$.
Then $\ol{\theta(\ol{U\cap X})}^{\beta Y} \subseteq \bigcap^n_{i=1}\ol{\theta(\ol{U_i\cap X})}^{\beta Y}$.
Since the set on the left is nonempty, this shows that the family $\{\ol{\theta(\ol{U\cap X})}^{\beta Y}: U \in \cF_x\}$, which consists of closed sets in $\beta Y$,  has the finite intersection property.  By compactness of $\beta  Y$, we conclude that the intersection in the statement of the lemma is nonempty.

Suppose that $y_1,y_2$ are distinct points in 
 $\bigcap\{\ol{\theta(\ol{U\cap X})}^{\beta Y}: U \in \cF_x\}$.
Choose sets $V_i \in \cF_{y_i}$, $i=1,2$, so that $\ol{V_1}^{\beta _Y} \cap \ol{V_2}^{\beta Y} = \emptyset$.
 By Lemma \ref{l2.0.2}, $x\in \ol{\theta^{-1}(\ol{V_i\cap Y})}^{\beta X}$, $i=1,2$.
This contradicts Lemma \ref{l2.0.1} applied to the map $\theta^{-1}$.
\end{proof}

Define the map $\vp: \beta X\to \beta Y$ by taking $\vp(x)$ to be the unique point in $\bigcap\{\ol{\theta(\ol{U\cap X})}^{\beta Y}: U \in \cF_x\}$.
By symmetry, we also have an analogous map $\psi :  \beta Y\to \beta X$.
%Note that with this notation, Lemma \ref{l2.0.2} yields
Now we arrive at the first structural result on biseparating maps on vector-valued $C/C_*$ spaces.

\begin{thm}\label{t2.0.4}
Let $X, Y$ be Hausdorff completely regular topological spaces and let $E,F$ be nonzero Hausdorff topological vector spaces.
Suppose that $T:A(X,E)\to A(Y,F)$ is a biseparating map, where $A(X,E) = C(X,E)$ or $C_*(X,E)$ and $A(Y,F) = C(Y,F)$ or $C_*(Y,F)$.
Then there is a homeomorphism $\vp:\beta X\to \beta Y$ so that for any $f,g \in A(X,E)$ and any open set $U$ in $\beta {X}$, $f = g$ on $U\cap X$ if and only if $Tf = Tg$ on ${\vp}(U) \cap Y$.
\end{thm}

\begin{proof}
Consider the maps $\vp:\beta X \to \beta Y$ and $\psi: \beta Y \to \beta X$ given above.
By Lemma \ref{l2.0.2}, $\vp$ and $\psi$ are mutual inverses.
Let us show that $\vp$ is continuous. 
If $\vp$ is not continuous at some $x_0 \in \beta X$, then there is a net $(x_\al)$ converging to $x_0$ so that $(\vp(x_\al))$ converges to $y_1 \neq \vp(x_0)= y_0$.
Choose $V_i \in \cF_{y_i}$, $i=0,1$, so that $\ol{V_0}^{\beta Y} \cap \ol{V_1}^{\beta Y} = \emptyset$.
For a cofinal set of $\al$, $V_1 \in \cF_{\vp(x_\al)}$, hence $x_\al = \psi(\vp(x_\al)) \in \ol{\theta^{-1}(\ol{V_1\cap Y})}^{\beta X}$. Therefore, $x_0 \in \ol{\theta^{-1}(\ol{V_1\cap Y})}^{\beta X}$. Also, $V_0\in \cF_{y_0}$ implies that $x_0\in  \ol{\theta^{-1}(\ol{V_0\cap Y})}^{\beta X}$.
This is impossible since  $\ol{\theta^{-1}(\ol{V_i\cap Y})}^{\beta X}$, $i =0,1$, are disjoint by Lemma \ref{l2.0.1}.
This completes the proof of continuity of $\vp$.  It follows that $\vp$ is a homeomorphism by symmetry.

Let $U$ be an open set in $\beta X$ and suppose that $f= g$ on $U\cap X$ for some $f,g\in A(X,E)$. 
Let $y_0\in \vp(U)\cap Y$ and set $x_0 = \psi(y_0)\in U$.  We wish to show that $Tf(y_0) =Tg(y_0)$.
By the remark preceding Lemma \ref{l2.0.2}, we may assume that $U\cap X \in \cC(X)$.  
Then $U\in \cF_{x_0}$.  By definition of $\vp$, $y_0 = \vp(x_0) \in \ol{\theta(\ol{U\cap X})}^{\beta Y}$.
Since $y_0\in Y$ and $\theta(\ol{U\cap X})$ is closed in $Y$, $y_0 \in \theta(\ol{U\cap X})$.
By Theorem \ref{t5}, $Tf = Tg$ on $\theta(\ol{U\cap X})$.
Hence $Tf(y_0) = Tg(y_0)$.
This proves that $Tf = Tg$ on $\vp(U) \cap Y$.
The reverse implication follows by symmetry.
\end{proof}

Recall that a  Hausdorff completely regular  topological space
$X$ is {\em realcompact} if for any $x_0\in \beta X\bs X$, there is a continuous function $f:\beta X\to [0,1]$ so that $f(x_0) =1 > f(x)$ for all $x\in X$.
For more on realcompact spaces, refer to the classic \cite{GJ}.
In particular, let $\upsilon X$ be the Hewitt realcompactification of $X$ \cite{GJ}.  Then every $f\in C(X)$ has a unique continuous extension to a (real valued) function $\stackrel{\smile}{f} \in C(\upsilon X)$. 
The map $f\mapsto\stackrel{\smile}{f}$ is an algebraic isomorphism and hence biseparating.  Hence it is rather natural to consider realcompact spaces in the present context.
When one or more of the spaces $X$ or $Y$ is realcompact, Theorem \ref{t2.0.4} can be improved.

\begin{lem}\label{l2.4}\cite[Lemma 3.2]{LT}
Let $Y$ be a realcompact space and let $y_0 \in \beta Y \bs Y$.  There exist  open sets $U_n$ and $V_n$ in $\beta Y$, $n \in \N$, such that
\begin{enumerate}
\item $\ol{U_n}^{\beta Y} \subseteq V_n$ for all $n$;
\item $y_0 \in \ol{\bigcup^\infty_{n=m}U_n}^{\beta Y}$ for all $m$;
\item $Y \cap \bigcap^\infty_{m=1}\ol{\bigcup^\infty_{n=m}V_n}^{\beta Y} = \emptyset$;
\item $\ol{V_n}^{\beta Y} \cap \ol{V_m}^{\beta Y} = \emptyset$ if  $n\neq m$.
\end{enumerate}
\end{lem}

\begin{lem}\label{l2.5.1}
Let $E$ be a Hausdorff topological vector space and let $0 \neq u\in E$.
There is a continuous function $h:E\to \R$ so that $h(nu)=  u$ for all $n\in \N$.
\end{lem}

\begin{proof}
Let $U$ be a circled open neighborhood of $0$ in $E$ so that $u \notin U + U +U$.
Then let $V$ be an open neighborhood of $0$ so that $\ol{V}\subseteq U$.
Set $V_n = nu + V$, $n\in\N$.  Suppose that $m\neq n$ and $\ol{V_m} \cap \ol{V_n} \neq \emptyset$.
Then there are $v_1,v_2\in \ol{V}$ so that $mu+v_1 = nu + v_2$.
Thus
\[ u = \frac{v_2}{m-n} + \frac{v_1}{n-m} \in U+U \subseteq U+U+U,
\]
contrary to the choice of $U$.  This shows that $\ol{V_m}\cap \ol{V_n} = \emptyset$ if $m\neq n$.

Next, we claim that 
$\ol{\bigcup V_n} = \bigcup \ol{V_n}$.
Suppose that $x\in \ol{\bigcup V_n}$.  
Choose $n_0\in \N$ so that $\frac{x}{n_0} \in U$.
For any $n\geq n_0$, if $V_n \cap (x+U) \neq \emptyset$, then there are $v\in V$ and $w\in U$ so that 
$nu+ v = x+w$.  Thus 
\[ u = \frac{x}{n} + \frac{w}{n} -\frac{v}{n} \in U + U + U,\]
a contradiction. Hence $x\notin \ol{\bigcup_{n\geq n_0}V_n}$.
Therefore, $x\in \ol{\bigcup^{n_0-1}_{n=1}V_n} = \bigcup^{n_0-1}_{n=1}\ol{V_n}$.
This proves the claim.

Since $E$ is completely regular, for each $n\in \N$, there exists a continuous function $h_n:E\to \R$ so that $h_n(nu) = n$ and that $h_n(x) = 0$ if $x \notin V_n$.
Define $h:E\to \R$ by $h(x) = h_n(x)$ if $x\in V_n$ for some $n$ and $h(x) = 0$ otherwise.
From the properties of the sets $V_n$ shown above, for each $x\in E$, there are an open neighborhood  $O$ 
of $x$ and some $n\in \N$ so that $h = h_n$ on $O$ or $h = 0$ on $O$.
It follows easily that $h$ is continuous.
Obviously, $h(nu)= n$ for all $n\in \N$.
\end{proof}

\begin{lem}\label{l2.6}
If $A(Y,F) = C(Y,F)$ and $Y$ is realcompact, then $\vp(X)\subseteq Y$.
\end{lem}

\begin{proof}
Retain the notation of Theorem \ref{t2.0.4}.  Suppose that there exists $x_0\in X$ so that $y_0 = \vp(x_0) \in \beta Y \bs Y$.
Choose open sets $U_n, V_n$ in $\beta Y$ using Lemma \ref{l2.4}.
From property (1) of said lemma, there exists a continuous function $f_n:\beta Y \to [0,1]$ so that
$f_n =1$ on $U_n$ and $f_n = 0$ outside $V_n$.
Fix a nonzero vector $u \in E$ and let $g_n$ be defined on $Y$ by $g_n(y) = f_n(y)T(n\otimes u)(y)$.
Then set $g(y) = g_n(y)$ if $y\in V_n\cap Y$ for some $n$ and $g(y) = 0$ if $y \in Y \bs (\bigcup^\infty_{n=1}V_n)$.
Fix $y \in Y$.  By property (3) of Lemma \ref{l2.4}, there exists $m$ so that $y \notin \ol{\bigcup^\infty_{n=m}V_n}^{\beta Y}$.
By property (4) of Lemma \ref{l2.4}, there exists at most one $n_0$, $1\leq n_0< m$, so that $y\in \ol{V_{n_0}}^{\beta Y}$.
Therefore, there exists an open neighborhood $U$ of $y$ in $Y$ so that $g = g_{n_0}$ or $g =0$ on the set $U$.
Thus $g$ is continuous at $y$.  Since $y\in Y$ is arbitrary, $g\in C(Y,F)$.
As $g = T(n\otimes u)$  on $U_n\cap Y$, by Theorem  \ref{t2.0.4}, $T^{-1}g = nu$ on ${\vp}^{-1}(U_n)\cap  X$.
By Lemma \ref{l2.5.1}, there is a continuous function $h:E\to \R$ so that $h(nu) = n$ for all $n$.
Set  $k = h\circ T^{-1}g \in C(X)$. Let $m\in \N$.  From the above, $k \geq m$ on ${\vp}^{-1}(\bigcup^\infty_{n=m}U_n)\cap X$.
By (2) of Lemma \ref{l2.4}, $y_0 \in \ol{\bigcup^\infty_{n=m}U_n}^{\beta Y}$.
Since each $U_n$ is open in $\beta Y$ and $\vp(X)$ is dense in $\beta Y$, % and $\ti{\vp}(X)$ is dense in $\beta Y$ by Corollary \ref{c2.51}, 
\[y_0 \in \ol{(\bigcup^\infty_{n=m}U_n) \cap \vp(X)}^{\beta Y}.
\] As $\vp:\beta X\to \beta Y$ is a homeomorphism,
\[x_0  =\vp^{-1}(y_0) \in \ol{\vp^{-1}(\bigcup^\infty_{n=m}U_n)\cap X}^{\beta X}.
\]
Recall that $x_0\in X$.  By continuity of $k$, $k(x_0) \geq m$. This is a contradiction since $k$ is real-valued and $m$ is arbitrary.
\end{proof}

\begin{thm}\label{t2.8}
Let $X$, $Y$ be realcompact spaces and let $E$ and $F$ be Hausdorff topological vector spaces.
If $T:C(X,E)\to C(Y,F)$ is a (nonlinear) biseparating map, then $X$ and $Y$ are homeomorphic.
\end{thm}

\begin{proof}
By Lemma \ref{l2.6}, $\vp$ maps $X$ into $Y$.  By symmetry, $\vp$ maps $X$ onto $Y$.
Hence it is a homeomorphism from $X$ onto $Y$.
\end{proof}

Theorem \ref{t2.8} generalizes the same result obtained in \cite{A, BBH} for {\em linear} biseparating maps.

\begin{thm}\label{t2.9}
Suppose that $Y$ is realcompact. Let $T:C_*(X,E)\to C(Y,F)$ be a biseparating map.  Then $Y$ is compact.
\end{thm}

\begin{proof}
The proof is the same as the proof of Lemma \ref{l2.6}.
If $y_0 \in \beta Y \bs Y$, then following the proof of Lemma \ref{l2.6}, one can choose $0\neq u\in E$ and construct a function $g\in C(Y,F)$ 
and  a sequence of nonempty sets $(W_n)$ in $X$ so that $T^{-1}g = nu$ on $W_n$ for each $n\in \N$.
($W_n$ is the set $\vp^{-1}(U_n)\cap X$ in the proof of Lemma \ref{l2.6}.)
Since the set $(nu)^\infty_{n=1}$ cannot be a bounded set in $E$, $T^{-1}g\notin C_*(X,E)$, contrary to the assumption.  Therefore, $Y = \beta Y$  is compact.
\end{proof}

\noindent 
{\bf Remark}.  It is well known that $C_*(X)$ is algebraically isomorphic to $C(\beta X)$.  Hence one cannot expect $X$ to be compact in Theorem \ref{t2.9} in general.

\section{Metric cases -- general results}\label{s3}
%In this section, we have in mind to treat  the spaces given in parts (2) to (4) of Example \ref{e1.7}.
Throughout this section, let $X$ and $Y$ be  metric spaces,  $E$ and $F$  be nontrivial normed spaces and $A(X,E)$, $A(Y,F)$  be  vector subspaces of $C(X,E)$ and $C(Y,F)$ respectively. Recall that a {\em bump function} on a Banach space $G$ is a nonzero real-valued function on $G$ with bounded support. 
When we speak of the spaces $C^p(X,E)$ or $C^p_*(X,E)$, it will  be assumed additionally that $X$ is an open set in a Banach space that supports a $C^p$, respectively, $C^p_*$, bump function.  
%Note that if $f$ is  a $C^p$ bump function defined on a Banach space $G$  and $h$ is a bounded function in $C^p(\R)$ so that $h(t) \neq 0$ if $t\neq 0$, then $h\circ f$ is a bounded $C^p$ bump function on $G$.
A sequence $(x_n)$ in $X$ is {\em separated} if $\inf_{n\neq m}d(x_n,x_m) > 0$.
The main aim of the section is an analog of the structural result Theorem \ref{t2.0.4} when $X$ and $Y$ are metric spaces.  In this instance, we make use of completion instead of compactification.

\begin{prop}\label{p3.0.1}
Let $A(X,E)$ be one of the spaces $C(X,E)$, $C_*(X,E)$, $U(X,E)$, $U_*(X,E)$, $\Lip(X,E)$, $\Lip_*(X,E)$, $C^p(X,E)$ or $C^p_*(X,E)$.
Then $A(X,E)$ has
 the following properties.
\begin{enumerate}
\item[(S1)] $A(X,E)$ is compatible.
\item[(S2)] For any $x\in X$ and any $
\ep >0$, there exists  $C\in \cC(X)$ so that $x\in C$ and  $\diam C <\ep$. In particular, $A(X,E)$ is basic and $\wh{C}(f) = \ol{C(f)}$ for all $f\in A(X,E)$. 
\item[(S3)] If $f\in A(X,E)$ and $(x_n)$ is a separated sequence in $X$, then there are a  sequence $(C_n)$ in $\cC(X)$ and a function $g\in A(X,E)$ so that $x_n \in C_n$ for all $n$, $g = f$ on $C_n$ for infinitely many $n$ and $g = 0$ on $C_n$ for infinitely many $n$. 
\item[(S4)] Let $(x_n)$ and $(x'_n)$ be Cauchy sequences so that $\inf_{m,n} d(x_m,x'_n) > 0$.
For any $f\in A(X,E)$, there are sets $U,V\in \cC(X)$ and a function $g\in A(X,E)$ so that $x_n\in U$ and $x'_n\in V$ for infinitely many $n$,  $g = f$ on $U$ and $g = 0$ on $V$.
\end{enumerate}
\end{prop}

\begin{proof}
Except for property (S3) for the spaces $U(X,E)$ and $\Lip(X,E)$, all other verifications are straightforward and are left to the reader.
To verify (S3) for $A(X,E) = U(X,E)$ or $\Lip(X,E)$, 
let $(x_n)$ be a sequence in $X$ so that $\inf_{n\neq m}d(x_n,x_m) = 3r > 0$ and let $f\in A(X,E)$.

In the first instance, assume that $(f(x_n))$ is a bounded sequence in $E$.  
Since $f\in U(X,E)$, there exists $0< r' < r$ so that \[\sup\{\|f(x)\|: x\in \bigcup_n B(x_n,r')\} = M < \infty.\]
For each $n\in \N$, let $h_n, k_n:E\to [0,1]$ be defined by 
\[ h_n(x) = (2 - \frac{2}{r'}d(x,x_n))^+\wedge 1 \text{ and } k_n(x) = (1 - \frac{d(x,x_n)}{r'})^+.\]
Then $(h_n)$ is a sequence of disjoint functions.  Let $h$ be the pointwise sum $\sum h_{2n-1}$. It is easily verified that $h: E\to [0,1]$ is a Lipchitz function with Lipschitz constant $\frac{2}{r'}$.
Take $g = h\cdot f$.  For each $n$, let $C_n = B(x_n,\frac{r'}{2})$.  Fix a nonzero vector $a\in E$.  Then $k_n\otimes a \in \Lip(X,E) \subseteq A(X,E)$.  Hence
$C_n = C(k_n\otimes a) \in \cC(X)$.  Clearly $g= f$ on $C_n$ if $n$ is odd and $g = 0$ on $C_n$ if $n$ is even.
Let us verify that $g\in A(X,E)$.
Indeed, suppose that $s, t\in X$. If $s,t\notin \bigcup_n B(x_n,r')$, then $h(s) = h(t) =0$.  Hence $\|g(s) - g(t)\| = 0$.  Otherwise, assume without loss of generality that $t\in \bigcup_n B(x_n,r')$.  We have
\begin{align*}
\|g(s) -g(t)\| & \leq |h(s)|\, \|f(s) - f(t)\| + |h(s)-h(t)|\,\|f(t)\| \\
& \leq \|f(s)-f(t)\| + \frac{2}{r'}\,d(s,t)\cdot M.
\end{align*}
Since $f\in A(X,E)$, it follows that $g \in A(X,E)$.

In the second case, assume that $(f(x_n))$ is unbounded in $E$.
Let $t_n = \|f(x_n)\|$ for all $n$.  By replacing $(x_n)$ by a subsequence if necessary, we may assume that $t_1> 0$ and $6t_n < t_{n+1}$ for all $n$.  Define $\gamma: [0,\infty) \to [0,1]$ by
\[ \gamma(t) = \begin{cases} 
                        (2 - \frac{3}{t_n}|t-t_n|)\wedge 1 &\text{if $|t-t_n| < \frac{2t_n}{3}$ for some odd $n$}\\
                        0 & \text{otherwise}. \end{cases}
                        \]
Direct verification shows that if  $0\leq a\leq  b\neq 0$, then $|\gamma(a)-\gamma(b)| \leq \frac{6|a-b|}{b}$.
 %Divide $[0,\infty)$ into segments $[t_n/3, t_{n+1}/3)$
%If $a, b$ belong to the same segment, use Lipschitzness of $\gamma$.
%If $a \in [t_n/3,t_{n+1}/3), b\in [t_m/3, t_{m+1}/3)$, $n < m$, we may assume $a < 5t_n/3$ (otherwise, $\gamma(a) =0$ and we can increase $a$ to put it in the $m$th segment).
%Now $|\gamma(a) - \gamma(b)| \leq 2$.  Use the fact that $a < 5t_n/3 < 5b/6$.
Let $g:X\to E$ be given by $g(x) = \gamma(\|f(x)\|)f(x)$.
Suppose that $y,z\in X$ with $\|f(y)\| \leq \|f(z)\|$ and $f(z) \neq 0$.
Then
\begin{align*}
\|g(y)- g(z)\| & \leq |\gamma(\|f(y)\|) - \gamma(\|f(z)\|)|\,\|f(y)\| + |\gamma(\|f(z)\|)|\,\|f(y) - f(z)\|\\
& \leq \frac{6(\|f(z)\| - \|f(y)\|)}{\|f(z)\|}\,\|f(y)\| + \|f(y) - f(z)\|\\
& \leq 7\|f(y)-f(z)\|.
\end{align*}
The same inequality obviously holds if $f(y) = f(z) = 0$.
Since $f$ belongs to $A(X,E)$, so does $g$. 
As $\frac{t_n}{3} < \|f(x_n)\| < \frac{5t_n}{3}$, there exists $C_n\in \cC(X)$ so that 
\[ x_n \in C_n \subseteq \{x: \frac{t_n}{3} < \|f(x)\| < \frac{5t_n}{3}\}.\]
 %then $U_n$ is an open neighborhood of $x_n$ in $X$.  
Finally, $\gamma(\|f(x)\|) =1$ if  $x \in C_{2n-1}$ and $\gamma(\|f(x)\|) =0$ if  $x \in C_{2n}$.
Hence $g = f$ on $C_{2n-1}$ and $g=0$ on $C_{2n}$.
%Clearly, one can find $u_n\in A(X,E)$ so that $x_n \in C(u_n) \subseteq U_n$ for all $n$.
This completes the verification of property (S3) for $A(X,E) = U(X,E)$ or $\Lip(X,E)$.
\end{proof}

For the sake of brevity, let us say that $A(X,E)$ is {\em standard} if it satisfies properties (S1) -- (S4).
%We will see later that all spaces in parts (2) and (3) of Example \ref{e1.7} are standard.  The space $C^p_*(X,E)$ of Example \ref{e1.7}(4) is standard under additional assumption on $X$.
For the rest of the section, assume that $A(X,E)$ and $A(Y,F)$ are standard spaces and that $T:A(X,E)\to A(Y,F)$ is a biseparating map.  Without loss of generality, normalize $T$ by taking $T0=0$.
Let $\theta:\cD(X) \to \cD(Y)$ be the map obtained from Theorem \ref{t5}.
As in Section \ref{s2}, we will show that $\theta$ induces a point mapping $\vp$.

Denote by $\ti{X}$ and $\ti{Y}$ the respective completions of the spaces $X$ and $Y$.
For any subset $U$ of $\ti{X}$, denote the closure of $U$ in $\ti{X}$ by $\ti{U}$.  Similarly for sets in $\ti {Y}$.
If $x_0\in \ti{X}$ and $(U_n)$ is a sequence of nonempty sets in  $\cC(X)$ so that $\diam U_n \to 0$ and $d(x_0,U_n)\to 0$, we write $(U_n) \sim x_0$.   By condition (S2), for any $x_0\in \ti{X}$, there is always a sequence $(U_n)$ so that $(U_n)\sim x_0$.

Suppose that $y \in \theta(\ol{U})\cap V$, where $U = C(u) \in \cC(X)$ and  $V= C(v)\in \cC(Y)$.
Since $\theta(\ol{U}) = \ol{C(Tu)}$, $C(Tu) \cap C(v)\neq \emptyset$.
Thus $C(u)\cap C(T^{-1}v)\neq \emptyset$.  
Hence $U \cap \theta^{-1}(\ol{V})\neq \emptyset$.

\begin{lem}\label{l3.1}
Let $x_0\in \ti{X}$ and assume that $(U_n)\sim x_0$. %Let $(x_n)$ be a sequence of  points in $X$.
Take $y_n \in \theta(\ol{U_n})$ for each $n$.
\begin{enumerate}
\item If $x_0\in X$,  then $(y_n)$ is a Cauchy sequence in $Y$.
\item If, in additon, $A(X,E)\subseteq U(X,E)$ and contains a nonzero constant function, then $(y_n)$ is a Cauchy sequence in $Y$ for any $x_0\in \ti{X}$.
\end{enumerate}
\end{lem}

\begin{proof}
First we show that every subsequence of $(y_n)$ has a further Cauchy subsequence. Otherwise, there is a subsequence $(y'_n)$ of $(y_n)$ that is separated.  
Under assumption (1),  $x_0 \in X$.  It follows from condition (S2) that  there is a function $f\in A(X,E)$ so that $f(x_0) \neq 0$.  Under assumption (2), take $f = 1\otimes a\in A(X,E)$, where $a\neq 0$.  
Since $A(Y,F)$ has property (S3), there are a subsequence of $(y_n')$, still denoted as $(y_n')$, a sequence $(V_n)$ in $\cC(Y)$ and a function $g\in A(Y,E)$ so that $y'_n\in V_n$ for all $n$, $g= Tf$ on $V_n$ for infinitely many $n$ and $g = 0$ on $V_n$ for infinitely many $n$.
Then $T^{-1}g = f$ on $\theta^{-1}(\ol{V_n})$ for infinitely many $n$ and $T^{-1}g = 0$ on $\theta^{-1}(\ol{V_n})$ for infinitely many $n$.
%Say $V_n = C(v_n)$, $v_n\in A(Y,F)$. 
Since $y'_n \in \theta(\ol{U_n}) \cap V_n$, $U_n \cap \theta^{-1}(\ol{V_n}) \neq \emptyset$ by the discussion just before the lemma.  Choose a point $x'_n$ from the intersection.  Then $(T^{-1}g)(x'_n)=f(x_n')$ for infintely many $n$  and $0$ infinitely often.
Moreover, $(x'_n)$ converges to $x_0$ in $\ti{X}$.
Under assumption (1),
$x_0 \in X$, and we have a contradiction to the continuity of $T^{-1}g$ at $x_0$. 
Under assumption (2), $T^{-1}g \in U(X,E)$ and $(x'_n)$ is Cauchy in $X$.  Hence $((T^{-1}g)(x'_n))$ is Cauchy in $E$.  This is impossible since $(T^{-1}g)(x'_n) = f(x'_n) = a$  and $(T^{-1}g)(x'_n) = 0$ both occur infinitely many times.

If the whole sequence $(y_n)$ is not Cauchy, then in view of the previous paragraph, there are subsequences $(y_{i_n})$ and $(y_{j_n})$ and $\ep >0$ so that both subsequences are Cauchy and that $d(y_{i_m},y_{j_n}) > \ep $ for all $m,n$.
%By property ($\star$), and replacing $(y_n')$ and $(y_n'')$ by subsequences if necessary, we may assume that there are functions $v_1,v_2\in A(Y,F)$ so that $y_n'\in C(v_1)$, $y_n''\in C(v_2)$ for all $n$ and that $\diam(C(v_i)) < \frac{\ep}{3}$, $i=1,2$. 
Choose the function $f$ as in the last paragraph.  By property (S4), there are  $U,V \in  \cC(Y)$ and $g\in A(Y,F)$ so that $y_{i_n}\in U$, $y_{j_n}\in V$ for infinitely many $n$, $g= Tf$ on $U$ and $g = 0$ on $V$.
Thus  $T^{-1}g = f$ on $\theta^{-1}(\ol{U})$ and $T^{-1}g = 0$ on $\theta^{-1}(\ol{V})$.
%Choose $i_n,j_n$ so that $y_{i_n} \in \theta(\ol{U_{i_n}})$ and $y_n'' \in \theta(\ol{U_{j_n}})$ for all $n$.
Then $y_{i_n} \in \theta(\ol{U_{i_n}}) \cap U$ for infinitely many $n$ and hence $U_{i_n} \cap \theta^{-1}(\ol{U}) \neq \emptyset$ for infinitely many $n$.
Let  $x_{i_n} \in  U_{i_n} \cap \theta^{-1}(\ol{U})$.  Then $(x_{i_n})$ converges to $x_0$ in $\ti{X}$ and  $T^{-1}g(x_{i_n}) = f(x_{i_n})$ for all $n$.
Similar consideration using the sequence $(y_{j_n})$ shows that there is a sequence $(x_{j_n})$ converging to $x_0$ in $\ti{X}$ so that $T^{-1}g(x_{j_n}) = 0$ for all $n$.
%Of course, both $(x_n')$ and $(x''_n)$ converge to $x_0$.
Under assumption (1), 
\[ \lim T^{-1}g(x_{i_n}) = \lim f(x_{i_n}) = f(x_0) \neq 0 = \lim T^{-1}g(x_{j_n}),\]
contradicting the continuity of $T^{-1}g$ at $x_0$.
Under assumption (2), $f(x_{i_n}) =a\neq 0$ by  choice of $f$.  Thus $T^{-1}g(x_{i_n}) = a$  and $T^{-1}g(x_{j_n}) = 0$ for all $n$, contradicting the uniform continuity of $T^{-1}g$.
\end{proof}

%\medskip

%Assume that $A(X,E)$ is standard.
%If $x_0\in X$ and $(U_n)$ is a sequence of nonempty sets in  $\cC(X)$ so that $\diam U_n \to 0$ and $d(x_0,U_n)\to 0$, we write $(U_n) \sim x_0$.   
%For any $x\in \ti{X}$, let $(U_n)\sim x$ and $y_n\in \theta(\ol{U_n})$.  

Suppose that $x_0\in \ti{X}$.  Let $(U_n)\sim x_0, (V_n)\sim x_0$ and $y_n\in U_n, z_n\in V_n$ for all $n$.
Then $(U_1,V_1,U_2,V_2,\dots)\sim x_0$.
By Lemma \ref{l3.1}, if $x_0\in X$, then the sequence $(y_1,z_1,y_2,z_2,\dots)$ is Cauchy.
Define $\vp: X\to \ti{Y}$ by setting $\vp(x) = \lim y_n$, where $(U_n)\sim x$ and $y_n \in \theta(\ol{U_n})$ for all $n$.  From the above, $\vp(x)$ is independent of the choices of $(U_n)$ and $(y_n)$.
Similarly, if $A(X,E)\subseteq U(X,E)$  and contains a constant function $1\otimes a$ for some $a\in E\bs \{0\}$, then Lemma \ref{l3.1}(2) shows that there is  a well defined  map $\ti{\vp}:\ti{X}\to \ti{Y}$ given by 
$\ti{\vp}(x) = \lim y_n$, where $(U_n)\sim x$ and $y_n \in \theta(\ol{U_n})$ for all $n$.
Clearly, in this case, $\ti{\vp}$ extends $\vp$.
 By symmetry, there is also a similar map $\psi:Y\to \ti{X}$ and a map $\ti{\psi}:\ti{Y}\to \ti{X}$ under corresponding assumptions on $A(Y,F)$.

%\bigskip

%\noindent{\bf Remark}.  
%If $(U_n)\sim x$, Lemma \ref{l3.1} implies that  $\diam \theta(\ol{U_n}) \to 0$.
%\item Suppose that $x\in U \in \cC(X)$.  We can choose $(U_n)$ so that $(U_n)\sim x$ and $U_n \subseteq U$ for all $n$.  By definition of $\vp$, $\vp(x) = \lim y_n$ where $y_n\in \theta(\ol{U_n}) \subseteq \theta(\ol{U})$.  Hence $\vp(x) \in \ti{\theta(\ol{U})}$.
%\end{enumerate}

\begin{lem}\label{l3.2}
The map $\vp$ is continuous from $X$ into $\ti{Y}$.   If, in addition, $A(X,E)\subseteq U(X,E)$  and contains a nonzero constant function, then  $\ti{\vp}:\ti{X}\to \ti{Y}$ is continuous on $\ti{X}$.
\end{lem}

\begin{proof}
We will prove the second assertion.  The first statement can be shown in the same way.
Under the second assumption, $\ti{\vp}$ is well defined.  Let $(x_n)$ be a sequence in $\ti{X}$ that converges to a point $x_0\in \ti{X}$.
By definition of $\ti{\vp}$, for each $n$, $\ti{\vp}(x_n) = \lim_k y_{nk}$, where $y_{nk} \in \theta(\ol{U_{nk}})$ and $(U_{nk})_k \sim x_n$.
For each $n$, choose $k_n$ so that $d(y_{nk_n}, \ti{\vp}(x_n)),  \diam U_{nk_n}, d(U_{nk_n}, x_{n})    < \frac{1}{n}$.  Then $y_{n_{k_n}} \in U_{nk_n}$ and $(U_{nk_n})\sim x_0$.
Thus $\ti{\vp}(x_0) = \lim y_{n{k_n}} = \lim \ti{\vp}(x_n)$.
\end{proof}

%If $A(X,E)\subseteq U(X,E)$, then every $f\in A(X,E)$ has a unique continuous extension $\ti{f}: \ti{X} \to \ti{E}$, where $\ti{E}$ is the completion of $E$.

%Observe that if $A(X,E)\subseteq U(X,E)$, then any function $f\in A(X,E)$ has a unique continuous extension $\ti{f}:\ti{X} \to \ti{E}$, where $\ti{E}$ is the completion of $E$.

Suppose that $x\in U \in \cC(X)$.  We can choose $(U_n)$ so that $(U_n)\sim x$ and $U_n \subseteq U$ for all $n$.  By definition of $\vp$, $\vp(x) = \lim y_n$ where $y_n\in \theta(\ol{U_n}) \subseteq \theta(\ol{U})$.  Hence $\vp(x) \in \ti{\theta(\ol{U})}$.

\begin{lem}\label{l3.2.1}
Assume that  $A(X,E)\subseteq U(X,E)$  and contains a nonzero constant function.
Let $f,g\in A(X,E)$ and $U$ be an open set in $\ti{X}$.  If $f = g$ on $U\cap X$,  then $Tf = Tg$ on the set $\ti{\vp}(U) \cap Y$.
\end{lem}

\begin{proof}
Assume that $x_0\in U$ and $y_0 = \ti{\vp}(x_0) \in Y$.
Choose $(U_n)\sim x_0$ and let $x_n \in U_n$.
By the foregoing remark, $\vp(x_n) \in \ti{\theta(\ol{U_n})}$.  Pick $y_n \in \theta(\ol{U_n})$.
By definition of $\ti{\vp}$, $y_0 = \ti{\vp}(x_0) = \lim y_n$.
For all sufficiently large $n$, $\ol{U_n} \subseteq U\cap X$.
Hence $f=g$ on $\ol{U_n}$.  By Theorem \ref{t5}, $Tf = Tg$ on $\theta(\ol{U_n})$.
In particular, $Tf(y_n) = Tg(y_n)$.
By continuity of $Tf$ and $Tg$ at $y_0$, $Tf(y_0) = Tg(y_0)$.
\end{proof}

The following structure theorem applies to spaces of uniformly continuous functions and spaces of Lipschitz functions.

\begin{thm}\label{t3.5}
Suppose that both $A(X,E)$ and $A(Y,F)$ are standard subspaces of 
$U(X,E)$ and $U(Y,F)$ respectively so that both contain nonzero constant functions.
There is a homeomorphism $\ti{\vp}:\ti{X}\to \ti{Y}$ so that if
 $f,g\in A(X,E)$ and $U$ is an open set in $\ti{X}$, then $f=g$ on $U\cap X$ if and only if $Tf = Tg$ on $\ti{\vp}(U)\cap Y$.
 \end{thm}
 
\begin{proof}
Under the given assumptions, we have well defined continuous maps $\ti{\vp}:\ti{X}\to \ti{Y}$ and $\ti{\psi}:\ti{Y}\to \ti{X}$ by Lemma \ref{l3.2}.
In the next paragraph, we will show that $\ti{\psi}\circ \ti{\vp}$ is the identity map on $\ti{X}$.  With symmetry, this allows us to conclude that $\ti{\vp}$ is a homeomorphism.
The final property in the statement of the theorem follows from Lemma \ref{l3.2.1} and symmetry.

Let $x_0\in \ti{X}$ and let $(U_n)\sim x_0$.
It follows from (2) of Lemma \ref{l3.1} and the definition of $\ti{\vp}$ that $\diam \theta(\ol{U_n})\to 0$ and $d(\theta(\ol{U_n}),\ti{\vp}(x_0)) \to 0$.  By definition of $\theta$, there exists $V_n \in \cC(Y)$ so that $\theta(\ol{U_n}) = \ol{V_n}$.
Then $(V_n)\sim \ti{\vp}(x_0)$.
Hence $\ti{\psi}( \ti{\vp}(x_0))= \lim x_n$, where $x_n \in \theta^{-1}(\ol{V_n}) = \ol{U_n}$ for all $n$.
Therefore,  $\ti{\psi}( \ti{\vp}(x_0))= x_0$, as claimed.
\end{proof}

Next, we consider the cases where one or both of $A(X,E)$ and $A(Y,F)$ is either $C, C_*$ or $C^p$.

\begin{lem}\label{l3.6}
Suppose that $A(X,E)$ is standard and $A(Y,F) = C(Y,F)$, $C_*(Y,F)$ or $C^p(Y,F)$.  If $T:A(X,E)\to A(Y,F)$ is a biseparating map, then $\vp(X)\subseteq Y$.
\end{lem}

\begin{proof}
%First we show that $\vp(X) \subseteq Y$.  
Suppose on the contrary that there exists $x_0\in X$ so that $y_0 = \vp(x_0) \in \ti{Y}\bs Y$.
Let $(U_n)\sim x_0$.  Then $(\theta(\ol{U_n}))$ is a sequence of sets in $Y$, each with nonempty interior, so that $\diam\theta(\ol{U_n}) \to 0$ and $d(\theta(\ol{U_n}),y_0) \to 0$.
Hence one can find a sequence $(g_n)$ in $C_*(Y)$, respectively $C^p(Y)$, and a sequence $(V_n)$ of nonempty sets in $\cC(Y)$ so that $g_n =1$ on $V_n$, $\ol{C(g_n)}\subseteq \theta(\ol{U_n})$ for all $n$, and $\ol{C(g_m)}\cap \ol{C(g_n)} = \emptyset$ if $m\neq n$.
As observed in the proof of Theorem \ref{t3.5}, $\diam \theta(\ol{U_n})\to 0$.
So $(\ol{C(g_n}))$ is a pairwise disjoint sequence so that $\diam \ol{C(g_n)}\to 0$ and $d(\ol{C(g_n)},y_0) \to 0$, where $y_0 \notin Y$.  Therefore the pointwise sum $g = \sum g_{2n}$ belongs to $C_*(Y)$, respectively, $C^p(Y)$.
By condition (S2), there exists $f\in A(X,E)$ so that $f(x_0) \neq 0$.
Then $h = g\cdot Tf$  lies in $A(Y,F)$.
Since $h = Tf$ on $\ol{V_n}$ if $n$ is even and $h = 0$ on $\ol{V_n}$ if $n$ is odd, and $\ol{V_n} \in \cD(Y)$,
by Theorem \ref{t5}, $T^{-1}h = f$ on $\theta^{-1}(\ol{V_n})$ if $n$ is even and 
$T^{-1}h = 0$ on $\theta^{-1}(\ol{V_n})$ if $n$ is odd.
Since $\ol{V_n} \subseteq \theta(\ol{U_n})$, $\theta^{-1}(\ol{V_n}) \subseteq \ol{U_n}$.
Choose $x_n \in \theta^{-1}(\ol{V_n})$ for each $n$.  Then $(x_n)$ converges to $x_0$.
However, $T^{-1}h(x_n) = f(x_n)$ if $n$ is odd and $T^{-1}h(x_n) = 0$ if $n$ is even.
As $(f(x_n))$ converges to $f(x_0) \neq 0$, 
this contradicts the continuity of $T^{-1}h$ at $x_0$. This proves that  $\vp(X) \subseteq Y$.
\end{proof}

The next two results can be obtained utilizing the proof of Theorem \ref{t3.5} and taking into account Lemma \ref{l3.6}.  

\begin{thm}\label{t3.7}
Let $A(X,E) = C(X,E),$ $C_*(X,E)$ or $C^p(X,E)$ and let $A(Y,F) = C(Y,F),$ $C_*(Y,F)$ or $C^q(Y,F)$.
There exists a homeomorphism $\vp: X\to Y$ so that 
 for any $f,g\in A(X,E)$, and any open set $U$ in $X$, $f =g$ on $U$ $\iff$ $Tf = Tg$ on $\vp(U)$. 
\end{thm}

\begin{thm}\label{t3.8}
Let  $A(X,E)$ be a standard vector subspace of $U(X,E)$ that contains a nonzero constant fucntion.   Suppose that $A(Y,F) = C(Y,F)$, $C_*(Y,F)$ or $C^p(Y,F)$.
There exists a homeomorphism $\vp: X\to \vp(X)$, where $\vp(X)$ is a dense  subset of  $Y$, and
 for any $f,g\in A(X,E)$ and any open set $U$ in $X$, $f =g$ on $U$ $\iff$ $Tf = Tg$ on $\vp(U)$. 
\end{thm}

\begin{proof}
We will only prove the density of $\vp(X)$ in $Y$.  The other parts follow from the proof of Theorem \ref{t3.5}, using Lemma \ref{l3.6}.
By Lemma \ref{l3.2} and Lemma \ref{l3.6}, $\vp$ is a continuous map from $X$ into $Y$ with a continuous extension $\ti{\vp}:\ti{X}\to \ti{Y}$.
Also, we have an analogous continuous map $\psi:Y\to \ti{X}$.  
From the second paragraph of the proof of Theorem \ref{t3.5}, we see that $\ti{\vp}\circ\psi$ is the identity map on $Y$.  Given $y\in Y$, $\psi(y)\in \ti{X}$.  Hence there is a sequence $(x_n)$ in $X$ that converges to $\psi(y)$.
By continuity of $\ti{\vp}$ at $\psi(y)$, $(\vp(x_n))=(\ti{\vp}(x_n))$ converges to $\ti{\vp}(\psi(y)) =y$.
This proves that $y\in \ol{\vp(X)}$.
\end{proof}

We conclude this section with a remark concerning the space $C^p_*(X,E)$, where $X$ is an open set in a Banach space $G$ that supports a $C^p_*$ bump function.  In general, it may not be true that all functions in $C^p_*(X,E)$ are uniformly continuous (with respect to the norm on $G$).
On the other hand, an easy application of the mean value inequality shows that if $X$ is open and {\em convex} in $G$, then $C^p_*(X,E) \subseteq U(X,E)$.  %More generally, the same holds if $X$ is a finite union of open convex subsets of $G$.
In particular, Theorems \ref{t3.5} and \ref{t3.8} apply to $C^p_*$  spaces whose domains are convex open sets.
%(More generally, they apply whenever $C^p_*(X,E) \subseteq U(X,E)$ (and  $C^p_*(Y,F) \subseteq U(Y,F)$ for Theorem \ref{t3.5}.)

%%%%%%%%%%%%%%%%%%%%%%%%%%%

\section{Pointwise representation}\label{s4}

Retain the notation of Section \ref{s3}.  That is, let $X$ and $Y$ be metric spaces, $E$ and $F$ be nontrivial normed spaces, and assume that $A(X,E)$ and $A(Y,F)$ are standard vector subspaces of $C(X,E)$ and $C(Y,F)$ respectively. Say that  $A(X,E)$ has property (P) if
\begin{enumerate}
\item[(P)] For any accumulation point $x$ of $\ti{X}$ and any function $f\in A(X,E)$ so that $\lim_{\stackrel{z\to x}{z\in X}}f(z) =0$, there are  open sets $U$ and $V$ in $X$ and a function $g\in A(X,E)$ so that $x\in \ti{U} \cap \ti{V}$ and that $g =f$ on $U$ and $g = 0$ on $V$.
\end{enumerate}

\medskip

\noindent {\bf Remark}. 
If $x$ is an isolated point of $\ti{X}$, then $x\in X$.  In this case, given $f\in A(X,E)$ so that $f(x) = 0$, take $U = V= \{x\}$ and $g =f$.  It is clear that the conditions above are fulfilled.

\begin{prop}\label{p3.10}
Let $A(X,E)$ be one of the spaces $C(X,E)$, $C_*(X,E)$, $U(X,E)$, $U_*(X,E)$, $\Lip(X,E)$ or $\Lip_*(X,E)$.  Then $A(X,E)$  has property (P).
\end{prop}

\begin{proof}
Let $x_0$ be an accumulation point of $\ti{X}$ and let $f$ be a function in  $A(X,E)$ so that $\lim_{\stackrel{x\to x_0}{x\in X}}f(x) =0$.
There is a  sequence $(x_n)$ in $X$ converging to $x_0$ so that $0 < d(x_{n+1},x_0) < \frac{d(x_n,x_0)}{3}$ for all $n$.  Set $r_n = d(x_n,x_0)$ and let $\gamma_n:[0,\infty)\to \R$ be the function
\[\gamma_n(r) = (2 - \frac{4|r-r_n|}{r_n})^+\wedge 1.\]
$(\gamma_n)$ is a disjoint sequence of functions.  Furthermore,
\[ |\gamma_n(a) - \gamma_n(b)| \leq \frac{4}{r_n}|a-b| \wedge 1 \text{ for all $a, b\geq 0$}.\]
We may assume that $\|f(x)\| \leq1$ if $d(x,x_0) < \frac{3r_1}{2}$.
Let 
\[ g_n(x) = \gamma_n(d(x,x_0))f(x)\quad \text{and} \quad g = \sum g_{2n} \ \text{(pointwise sum).}\]
Since $\gamma_n(d(\cdot,x_0))$ is bounded Lipschitz and $f$ is bounded on the support of $g_n$, it is easy to check that $g_n\in A(X,E)$.
Note that
\[ \|g_n\|_\infty \leq \sup\{\|f(x)\| :\frac{r_n}{2} < d(x,x_0) < \frac{3r_n}{2}\} \to 0.\]
Therefore, if $A(X,E)$ is any of the spaces except $\Lip(X,E)$ or $\Lip_*(X,E)$, $g$ is the uniform limit of its partial sums and hence belongs to $A(X,E)$.

Now consider the cases $A(X,E) = \Lip(X,E)$ or $\Lip_*(X,E)$.  First of all, the function $g$ is bounded.
Let's check that it is Lipschitz. Since $f$ is Lipschitz and $\lim_{\stackrel{x\to x_0}{x\in X}}f(x) =0$, $\|f(x)\| \leq L(f)d(x,x_0)$, where $L(f)$ is the Lipschitz constant of $f$.
For any $n\in \N$, we claim that $g_n$ is Lipschitz with $L(g_n) \leq 7L(f)$.  Let   $x,z\in X$, $a = d(x,x_0)$, $b = d(z,x_0)$.  
%First of all, suppose there exists $n$ so that $a,b\in [\frac{r_{2n}}{2},\frac{3r_{2n}}{2}]$.
If $\gamma_n(a) = \gamma_n(b) =0$, then $g_n(x) - g_n(z) =0$.  Otherwise, we may assume that $\gamma_n(a) \neq 0$, so that $\frac{r_n}{2} < a < \frac{3r_n}{2}$.
Then
\begin{align*}
\|g_n(x) - g_n(z)\| & \leq |\gamma_{n}(a)-\gamma_{n}(b)|\,\|f(x)\| + |\gamma_{n}(b)|\,\|f(x) - f(z)\|\\
&\leq \frac{4}{r_{n}}|a-b|\,L(f)a + \|f(x) -f(z)\|\\
&\leq 6L(f)|a-b| + L(f)d(x,z) \leq 7L(f)d(x,z).
\end{align*}
Thus  $L(g_n)\leq 7L(f)$, as claimed.
For any $x,z\in X$, either there exists $n$ so that $g(x) = g_{2n}(x)$, $g(z) = g_{2n}(z)$, or there are distinct $m,n$ so that $g(x) = g_{2n}(x) + g_{2m}(x)$, $g(z) = g_{2n}(z) + g_{2m}(z)$.  In either case, it follows that $\|g(x)-g(z)\| \leq 14L(f)d(x,z)$. This completes the proof that  $g\in \Lip_*(X,E)\subseteq A(X,E)$.
Clearly, $g= f$ on the  open set 
\[ U = \bigcup_n \{x\in X: \frac{3r_{2n}}{4}< d(x,x_0) < \frac{5r_{2n}}{4}\}\]
 and $g = 0$ on the  open set 
\[V =  \bigcup_n \{x\in X: \frac{3r_{2n-1}}{4}< d(x,x_0) < \frac{5r_{2n-1}}{4}\}.\]
Since $x_n\in U$ for all even $n$, and $x_n\in V$ for all odd $n$, $x_0 \in \ti{U} \cap \ti{V}$.
\end{proof}

With the help of property (P), we can improve Theorems \ref{t3.5}, \ref{t3.7} and \ref{t3.8}.
First we consider the case where  $A(X,E)$ and $A(Y,F)$ are standard subspaces of 
$U(X,E)$ and $U(Y,F)$ respectively so that both contain nonzero constant functions.
%In this case, we have continuous maps $\ti{\vp}:\ti{X}\to \ti{Y}$ and $\psi: Y\to \ti{X}$ (Lemma \ref{l3.2}). By the proof of Theorem \ref{t3.5}, $\psi(\ti{\vp}(x)) = x$ if $\ti{\vp}(x) \in Y$ and $\ti{\vp}(\psi(y)) = y$ for all $y\in Y$.
Denote by $\ti{E}$ the completion of $E$.  Since $A(X,E) \subseteq U(X,E)$, every function $f\in A(X,E)$ has a unique continuous extension $\ti{f}:\ti{X}\to \ti{E}$. 
For each $x\in \ti{X}$, let 
\[ \ti{E}_x = \{\ti{f}(x): f\in A(X,E)\}.\]
Similarly for $\ti{F}_y$ if $y\in \ti{Y}$.
Fix a biseparating map $T:A(X,E)\to A(Y,F)$, which we may normalize by taking $T0 = 0$.
Let  $\ti{\vp}: \ti{X}\to \ti{Y}$ be the homeomorphism given by Theorem \ref{t3.5}, with inverse $\ti{\psi}$.

\begin{prop}\label{p4.2}
Suppose that  $A(X,E)$ and $A(Y,F)$ are standard subspaces of 
$U(X,E)$ and $U(Y,F)$ respectively so that both contain nonzero constant functions.  Assume that $A(X,E)$ has property (P).
Given any $y\in \ti{Y}$, there is a bijective function $\Phi(y,\cdot): \ti{E}_{\ti{\psi}(y)}\to \ti{F}_y$ so that 
\[ \ti{Tf}(y) = \Phi(y,\ti{f}(\ti{\psi}(y))) \text{ for all $f\in A(X,E)$}.\]
\end{prop}

\begin{proof}
%By Theorem \ref{t3.5}, $\ti{\vp}:\ti{X}\to \ti{Y}$ is a homeomorphism.
Let $y_0 \in \ti{Y}$ and $\ti{\psi}(y_0) =x_0\in \ti{X}$.  For any $a\in \ti{E}_{x_0}$, fix a function $g_a\in A(X,E)$ so that $\ti{g_a}(x_0) =a$ and define $\Phi(y_0,\cdot): \ti{E}_{x_0} \to F$ by $\Phi(y_0,a) = Tg_a(y_0)$.
%First consider the case where $x_0$ is an isolated point in $\ti{X}$.  In particular, $x_0\in X$.
%Consider the open set  $U = \{x_0\}$ in $\ti{X}$.  We have  $\ti{\vp}(U) = \{y_0\}$.
%If $f\in A(X,E)$, then $f = g_a$ on $U = U\cap X$, where $a = f(x_0)$.
%By Lemma \ref{l3.2.1}, $Tf = Tg_a$ on $\ti{\vp}(U) \cap Y  =  \{y_0\}$.
%Thus $Tf(y_0) = \Phi(y_0, f(x_0))$, as claimed.
%
%Next, suppose that $x_0$ is an accumulation point of $\ti{X}$. 
%Consider a point $x_0\in \ti{X}$.
If $f\in A(X,E)$, let $a = \ti{f}(x_0)$.  Clearly $\ti{f- g_a}(x_0) =0$.  By property (P) and the remark following its definition, there are  open sets $U,V$ in $X$ and a function $h\in A(X,E)$ so that $x_0 \in \ti{U} \cap \ti{V}$, $h = f-g_a$ on $U$ and $h =0$ on $V$.
Let $W$ be an open set in $\ti{X}$ so that $W\cap X = U$.
Since $\ti{\vp}$ is a homeomorphism and $y_0 = \ti{\vp}(x_0)$,   
$y_0 \in \ti{\vp}(\ti{U}) \subseteq \ti{\ti{\vp}(W)}$.
But $\ti{\vp}(W)$ is open in $\ti{Y}$.  So $y_0 \in ({\vp(W)}\cap Y)^{\ti{\ }}$.
As $f= h+g_a$ on $W\cap X$, $Tf = T(h+g_a)$ on $\ti{\vp}(W)\cap Y$ by Lemma \ref{l3.2.1}.
By continuity, $\ti{Tf}(y_0) = [T(h+g_a)]^{\ti{ \ }}(y_0)$.  %As $y_0\in Y$, this means that $Tf(y_0) = T(h+g_a)(y_0)$.
Similarly, looking at the set $V$ instead of $U$, one can show that 
\[ [T(h+g_a)]^{\ti{ \ }}(y_0) = {Tg_a}(y_0) = \Phi(y_0,a). \]
 Thus $\ti{Tf}(y_0) = \Phi(y_0,a)$, as required.
 
 By symmetry, there is a function $\Psi(x, \cdot): \ti{F}_{\ti{\vp}(x)}\to \ti{E}_x$ so that $(T^{-1}g)^{\ti{\ }}(x) = \Phi(x, \ti{g}(\ti{\vp}(x)))$ for all $g\in A(Y,F)$.
 The fact that $\Phi(y,\cdot)$ is a bijection follows from expressing the equations $T(T^{-1}g) = g$ and $T^{-1}(Tf) = f$ in terms of the mappings $\Phi(y,\cdot)$ and $\Psi(x,\cdot)$.
\end{proof}

The next two propositions can be obtained in a similar vein.  The details are omitted.

\begin{prop}\label{p4.3}
Suppose that  $A(X,E) = C(X,E)$ or $C_*(X,E)$ and $A(Y,F)$ $= C(Y,F)$ or $C_*(Y,F)$.  
There is a function $\Phi:Y\times E\to F$ so that 
\[ Tf(y) = \Phi(y,f(\psi(y))) \text{ for all $f\in A(X,E)$ and all $y\in Y$}.\]
\end{prop}

\begin{prop}\label{p4.4}
Suppose that  $A(X,E)$ is standard subspace of 
$U(X,E)$ that contains a nonzero constant function.  Assume that $A(X,E)$ has property (P).
Let $A(Y,F) = C(Y,F)$ or $C_*(Y,F)$.
\begin{enumerate}
\item For any $y\in Y$, there is a function  $\Phi(y,\cdot): \ti{E}_{\psi(y)}\to F$ so that 
\[Tf(y) = \Phi(y,\ti{f}(\psi(y))) \text{ for all $f\in A(X,E)$}.\]
\item There is a function  $\Psi: X\times F\to E$ so that 
\[T^{-1}g(x) = \Psi(x,g(\vp(x))) \text{ for all $g\in A(Y,F)$ and all $x \in X$}.\]
\end{enumerate}
\end{prop}

\section{Spaces of continuous functions -- metric case}\label{s5}

In this section, let $X, Y$ be metric spaces and $E,F$ be normed spaces.
%We will characterize all nonlinear biseparating maps between $C(X,E)$ or $C_*(X,E)$ on the one hand, onto $C(Y,F)$ or $C_*(Y,F)$ on the other hand.
%
Let $A(X,E) = C(X,E)$ or $C_*(X,E)$ and $A(Y,F) = C(Y,F)$ or $C_*(Y,F)$.
Fix a biseparating map $T:A(X,E)\to A(Y,F)$.  
By Theorem \ref{t3.7}, Proposition \ref{p4.3} and symmetry,   there are a homeomorphism $\vp:X\to Y$ and functions $\Phi:Y\times E\to F$, $\Psi:X\times F\to E$ so that 
\[ (Tf)(y)= \Phi(y,f(\vp^{-1}(y))),\quad (T^{-1}g)(x) = \Psi(x,g(\vp(x)))\]
for any $f\in A(X,E)$, $g\in A(Y,F)$, $x\in X$ and $y\in Y$.
From the equations 
\[ 1\otimes a = T^{-1}(T(1\otimes a)),\quad 1\otimes b = T(T^{-1}(1\otimes b))\]
for all $a\in E$, $b\in F$, we find that $\Phi(y,\cdot)$ and $\Psi(x,\cdot)$ are mutual inverses provided $y = \vp(x)$.  
The aim of the present section is to characterize the functions $\Phi$ that lead to biseparating maps and prove a result on automatic continuity.
Observe that if we define $S:C(Y,F)\to C(X,F)$ by $Sg(x) = g(\vp^{-1}(x))$, then $S$ is a biseparating map that also acts as a biseparating map from $C_*(Y,F)$ onto $C_*(X,F)$.
Thus characterization of biseparating maps reduces to the ``section problem'' addressed in Proposition \ref{p5.1}.
The result is well known, at least in the case for for $C(X)$.  See, e.g. \cite[Chapter 9]{AZ}.
We omit the easy proof of the next lemma.

\begin{lem}\label{l5.0}
Let $(x_n)$ be a sequence of distinct points in $X$ and let $(a_n)$ be a sequence in $E$.
\begin{enumerate}
\item If $(x_n)$ has no convergent subsequence, then there is a function $f\in C(X,E)$ so that $f(x_n) = a_n$ for all $n$.  Moreover, $f$ can be chosen to be bounded if $(a_n)$ is bounded.
\item If $(x_n)$ converges to a point $x_0\in X$, $x_0\neq x_n$ for all $n$, and $(a_n)$ converges to a point $a_0\in E$, then there exists $f\in C_*(X,E)$ so that $f(x_n) = a_n$ for all $n$.
\end{enumerate}
\end{lem}

Denote the set of accumulation points of $X$ by $X'$ and the unit balls  of $E$ and $F$ by $B_E$ and $B_F$ respectively.

\begin{prop}\label{p5.1}
Let $X$ be a metric space and let $E$ and $F$ be normed spaces.  Consider a function $\Phi:X\times E\to F$.
\begin{enumerate}
\item The function $x\mapsto \Phi(x,f(x))$ belongs to $C(X,F)$ for every $f\in C(X,E)$ if and only if $\Phi$ is continuous at every point in $X'\times E$.
\item The function $x\mapsto \Phi(x,f(x))$ belongs to $C_*(X,F)$ for every $f\in C_*(X,E)$ if and only if both of the following conditions hold.
\begin{enumerate}
\item $\Phi$ is continuous at every point in $X'\times E$. 
\item For any  bounded set $B$ in $E$, every  $(x_n) \in \prod_n X_n(B)$ has a subsequence that converges in $X$, where 
\[X_n(B) = \{x\in X: \Phi(x,B) \not\subseteq nB_F\}.\]
\end{enumerate}
\end{enumerate}
\end{prop}

\begin{proof}
Suppose that $x\mapsto \Phi(x,f(x))$ belongs to $C(X,F)$ for every $f\in C_*(X,E)$.
Let $(x_0,a_0)$ be a point in $X'\times E$ and let $((x_n,a_n))$ be a sequence in $X\times E$ that converges to $(x_0,a_0)$.
Since $\Phi(x,a_n)$ is a continuous function of $x$, by making small perturbations if necessary, we may assume that $x_n \neq x_0$ for all $n$.
Changing to a subsequence, we may further assume that $(x_n)$ is a sequence of distinct points.  By Lemma \ref{l5.0}(2), there exists $f\in C_*(X,E)$ so that $f(x_n) = a_n$ for all $n$.  By continuity of $f$, $f(x_0) = a_0$.
Then 
\[ \Phi(x_n,a_n) = \Phi(x_n,f(x_n)) \to \Phi(x_0,f(x_0)) = \Phi(x_0,a_0).\]
This proves the continuity of $\Phi$ at $(x_0,a_0)$.
Hence the ``only if'' parts in statement (1) and statement (2)(a) are verified.
On the other hand,
if  $\Phi$ is continuous at any point in $X'\times E$ and $f\in C(X,E)$, then it is clear that $\Phi(x,f(x))$ is a continuous function of $x\in E$.  This completes the proof of  statement (1).

Let us proceed to prove the necessity of condition (b) in statement (2).
Assume that condition (b) in (2) fails. 
Let $B$ be a bounded set in $E$ and let $(x_n)$ be an element in $\prod X_n(B)$ so that $(x_n)$ has no convergent subsequence in $X$.
Since $X_n(B) \subseteq X_m(B)$ if $m \leq n$, we may replace $(x_n)$ by a subsequence to assume that all $x_n$'s are distinct.
Choose $a_n\in B$ so that $\|\Phi(x_n,a_n)\| > n$ for all $n$.
By Lemma \ref{l5.0}(1), there exists $f\in C_*(X,E)$ so that $f(x_n) = a_n$ for all $n$.
Then $\|\Phi(x_n,f(x_n))\| = \|\Phi(x_n,a_n)\| \to \infty$.
This contradicts the assumption that the function $x\mapsto \Phi(x,f(x))$ is bounded.

Finally, we prove the sufficiency in statement (2).  Let $f\in C_*(X,E)$.  As observed above, by (2) condition (a), $x\mapsto \Phi(x,f(x))$ is continuous on $X$ since  $f\in C(X,E)$.
Let $B = f(X)$.  Then $B$ is a bounded set in $E$. If $(\Phi(x,f(x)))_{x\in X}$ is unbounded, there is a sequence of distinct points $(x_n)$ so that $\|\Phi(x_n,f(x_n))\| > n$ for all $n$.
In particular, $(x_n) \in \prod X_n(B)$. By condition 2(b), we may replace it by a subsequence to assume that $(x_n)$ converges to a point $x_0$ in $X$. In particular, $x_0\in X'$.  By assumption 2(a),
$\Phi(x_n, f(x_n))\to \Phi(x_0,f(x_0))$, contradicting the unboundedness of the sequence.
\end{proof}

The next two results follow immediately from the preceding discussion.

\begin{thm}\label{t5.4}
Let $X$ and $Y$ be metric spaces and let $E$ and $F$ be normed spaces.
A map  $T:C(X,E)\to C(Y,F)$ is a biseparating map if and only if there are a homeomorphism $\vp:X\to Y$ and functions $\Phi:Y\times E\to F$, $\Psi:X\times F\to E$ so that
\begin{enumerate}
\item $\Phi(y,\cdot)$ and $\Psi(x,\cdot)$ are mutual inverses if $\vp(x) = y$. 
\item $\Phi$ is continuous at any point in $Y'\times E$; $\Psi$ is continuous at any point in $X'\times F$.
\end{enumerate}
\end{thm}

\begin{thm}\label{t5.5}
Let $X$ and $Y$ be metric spaces and let $E$ and $F$ be normed spaces.
A map  $T:C_*(X,E)\to C_*(Y,F)$ is a biseparating map if and only if there are a homeomorphism $\vp:X\to Y$ and functions $\Phi:Y\times E\to F$, $\Psi:X\times F\to E$ so that
\begin{enumerate}
\item $\Phi(y,\cdot)$ and $\Psi(x,\cdot)$ are mutual inverses if $\vp(x) = y$. 
\item $\Phi$ is continuous at any point in $Y'\times E$; $\Psi$ is continuous at any point in $X'\times F$.
\item If $B_1$ and $B_2$ are bounded sets in $E$ and $F$ respectively, and 
\[Z_n(B_1,B_2) = \{x\in X: \Phi(\vp(x),B_1) \not\subseteq nB_F \text{ or } B_2 \not\subseteq \Phi(\vp(x),nB_E)\},\]
then  every $(x_n) \in \prod_n Z_n(B_1,B_2)$ has a subsequence that converges in $X$.
\end{enumerate}
\end{thm}

Observe that by condition (1) of Theorem \ref{t5.5}, $B_2 \not\subseteq \Phi(\vp(x),nB_E)$ if and only if $\Psi(x,B_2)\not\subseteq nB_E$.  Hence condition (3) in Theorem \ref{t5.5} is a combination of condition 2(b) in Proposition \ref{p5.1} for the maps $\Phi$ and $\Psi$.

\begin{prop}\label{p5.2}
If there is a biseparating map $T:C(X,E)\to C_*(Y,F)$, then $X$ are $Y$ are compact.
\end{prop}

\begin{proof}
Assume otherwise. Since $X$ and $Y$ are homeomorphic, %(Theorem \ref{t3.7}),
 there is a sequence of distinct points $(x_n)$ in $X$ that has no convergent subsequence.  Fix a nonzero element $b\in F$ and let $a_n = (T^{-1}(1\otimes nb))(x_n)$ for each $n$.
By Lemma \ref{l5.0}(1), there is a function $f\in C(X,E)$ so that $f(x_n) =a_n$ for all $n$.
Note that 
\[nb = T((T^{-1}(1\otimes nb))(\vp(x_n)) = \Phi(\vp(x_n),a_n) \text{ for each $n$.}\]
Then 
\[(Tf)(\vp(x_n)) = \Phi(\vp(x_n),a_n) = nb \text{  for all $n$},\]
contradicting the boundedness of $Tf$.
\end{proof}

Note that if $T:C(X,E)\to C_*(Y,F)$ is biseparating, then $X$ is compact by Proposition \ref{p5.2}.  Hence $C(X,E) = C_*(X,E)$.  Therefore, the characterization Theorem \ref{t5.5} applies.
We conclude this section with an automatic continuity result.  If $K$ is a compact subset of $X$, let 
\[\|f\|_K = \sup\{\|f(x)\|:x\in K\} \text{ for any $f\in C(X,E)$}.\]

\begin{thm}\label{t5.6}
Let $X$ and $Y$ be metric spaces and let $E$ and $F$ be normed spaces.
Suppose that $A(X,E) = C(X,E)$ or $C_*(X,E)$, $A(Y,F) = C(Y,F)$ or $C_*(Y,F)$. 
Let $T:A(X,E)\to A(Y,F)$ be a biseparating map. For any compact subset $K$ of $Y'$, any $f\in A(X,E)$, and any $\ep > 0$, there exists $\delta >0$ so that 
\[ g\in A(X,E),\ \|g-f\|_{\vp^{-1}(K)} <\delta \implies \|Tg-Tf\|_K < \ep.\]
\end{thm}

\begin{proof}
Suppose that $(g_n) \subseteq A(X,E)$ and that $\|g_n-f\|_{\vp^{-1}(K)} \to 0$.  It suffices to show that a subsequence of $(\|Tg_n-Tf\|)$ converges to $0$.
Pick $(y_n) \subseteq K$ so that $\|Tg_n-Tf\|_K = \|(Tg_n)(y_n) - (Tf)(y_n)\|$ for all $n$.
By using a subsequence if necessary, we may assume that $(y_n)$ converges to some $y_0\in K$.
Let $x_n = \vp^{-1}(y_n)$, 
$g_n(x_n) = a_n$ , and $f(x_n) = a'_n$.   Then $(x_n)$ converges to $x_0 = \vp^{-1}(y_0)$ and $f(x_0) = a_0$.
Since $\|g_n-f\|_K \to 0$, $(a_n)$ converges to $a_0$.  
Also $(a_n')$ converges to $a_0$ by continuity of $f$.
Since $y'\in K\subseteq Y'$, it follows fromm Proposition \ref{p5.1}(1) that $\Phi$ is continuous at $(y_0,a_0)$.
Therefore,
\[ 
\|(Tg_n)(y_n) - (Tf)(y_n)\|  = \|\Phi(y_n,a_n)-  \Phi(y_n,a'_n)\|\to 0.\]
Thus $\|Tg_n-Tf\|_K\to 0$.
\end{proof}

\section{Spaces of uniformly continuous functions}\label{s6}

In this section, let $X$, $Y$ be complete metric spaces and let $E$, $F$ be Banach spaces.
The aim of this section is to characterize biseparating maps from $U(X,E)$ or $U_*(X,E)$ onto $U(Y,F)$ or $U_*(Y,F)$.  By Propositions \ref{p3.10} and  \ref{p4.2}, a biseparating map $T: U(X,E)/U_*(X,E) \to U(Y,F)/U_*(Y,F)$
can be represented in the form 
\[ Tf(y) = \Phi(y,f(\psi(y))) \text{ for all $f\in U(X,E)/U_*(X,E)$ and all $y\in Y$},\]
where $\psi:Y\to X$ is a homeomorphism with inverse $\vp$ and $\Phi:Y\times E\to F$ is a function so that $\Phi(y,\cdot)$ is a bijection from $E$ onto $F$ for all $y\in Y$.
In fact, characterizations can be obtained without completeness assumptions of $X,Y,E,F$.  However, the case of complete spaces contains all pertinent ideas without the distraction of niggling details.
Characterizations  of lattice isomorphisms and of  {\em linear} biseparating maps on spaces of uniformly continuous functions were  obtained in \cite{GaJ} and \cite{A2} respectively.

\begin{prop}\label{p6.3.0}
Let $A(X,E)=  U(X,E)$ or $U_*(X,E)$, and $A(Y,F)=  U(Y,F), U_*(Y,F)$ or $\Lip_*(Y,F)$.
Let ${\vp}: {X}\to {Y}$ be the homeomorphism associated with $T$ according to Theorem  \ref{t3.5}.
Then ${\vp}$ is uniformly continuous.
\end{prop}

\begin{proof}
%Since $X$ is dense in ${X}$ and ${\vp}$ continuously extends $\vp$, it suffices to show that $\vp:X\to {Y}$ is uniformly continuous.
Suppose that $\vp$ is not uniformly continuous.  There are sequences $(x_n)$, $(x_n')$ in $X$ and $\ep >0$  so that %$(x_n) \cup (x'_n)$ consist of distinct points,  
$d(x_n,x'_n)\to 0$ and that $d(\vp(x_n),\vp(x_n')) > \ep$ for all $n$.
Set $y_n = \vp(x_n)$ and $y_n' = \vp(x_n')$.
In view of the continuity of ${\vp}$, niether $(x_n)$ nor $(x'_n)$ can have a convergent subsequence in ${X}$.
Hence we may assume that $(x_n)$ is a separated sequence.
Since ${\vp}^{-1}$ is continuous, neither $(y_n)$ nor $(y'_n)$ can have a convergent subsequence in ${Y}$.
As we also have $d(y_n,y_n') > \ep$ for all $n$, by using subsequences if necessary, we may  assume  that $(y_n) \cup (y_n')$ is a  separated set.
%For any $r >0$, let 
%\[ {Y}_r = \{y\in {Y}: B(y,r) = \{y\}\}.\]
Without loss of generality, take $T0 = 0$. We will use repeatedly the following formulation of Proposition \ref{p4.2}. If $x\in {X}$ and  $f, g\in A(X,E)$, then  ${f}(x) = {g}(x)$ if and only if ${Tf}(\vp(x)) = {Tg}(\vp(x)).$
%Consider several cases.

\medskip

\noindent\underline{Case 1}.  $A(Y,F) = U_*(Y,F)$ or $\Lip_*(Y,F)$.

Fix a nonzero vector $a\in E$ and let $b_n = {(T(1\otimes a))}(y_n)$.  Then $(b_n)$ is a bounded sequence. %(Recall that in the case of $\Lip(Y,F)$, $Y$  is assumed to be bounded.)
%Note that if $y_n\in Y$, then $b_n\in E$.  
Since $(y_n)\cup (y_n')$ is separated,  one can easily construct  a function $g\in  A(Y,F)$ so that ${g}(y_n) = b_n$ and ${g}(y'_n) = 0$ for all $n$. By Proposition \ref{p4.2}, we see that  $(T^{-1}g)(x_n) = a$ and $(T^{-1}g)(x'_n) = 0$.  Since $T^{-1}g$ is uniformly continuous and $d(x_n,x'_n)\to 0$, we have a contradiction.

\medskip

In the remaining cases,  take $A(Y,F) = U(Y,F)$.

\medskip

\noindent\underline{Case 2}. There exist $r>0$ and an infinite subset $N$ of $\N$ so that $B(y_n,r) = \{y_n\}$ for all $n\in N$.

%Note that in this case, $y_n\in Y$ for all $n$.  
Fix a nonzero vector $a \in E$.  Let $b_n = (T(1\otimes a))(y_n)$ for each $n\in N$.  Define $g:Y \to F$ by $g(y) = b_n$ if $y = y_n, n\in N$, and $g(y) = 0$ otherwise.
Clearly $g\in  U(Y,F)$.  
But by Proposition \ref{p4.2}, for all $n\in N$,
\[ (T^{-1}g)(x_n) = a  \text{ and } (T^{-1}g)(x_n') = 0.\]
Hence $T^{-1}g$ is not uniformly continuous, contrary to the fact that $T^{-1}g \in A(X,E) \subseteq U(X,E)$.
%Fix a nonzero vector $a \in E$. Set $b_n = (T(1\otimes a))(y_n)$ for each $n$.

\medskip

\noindent\underline{Case 3}. For all $r>0$, $B(y_n,r) = \{y_n\}$ occurs for only finitely many $n$.

In this case, by using a subsequence if necessary, we may assume that there is a sequence $(y_n'')$ in $Y$ so that $0 < d(y_n,y_n'') \to 0$.
Set $x_n'' = {\vp}^{-1}(y_n'')$ for all $n$.
Take a nonzero element $b\in F$ and let  
$a_n =(T^{-1}(1\otimes b))(x_n)$. %,\ a_n'' =(T^{-1}(1\otimes b))(x_n'').\]
Since $(y_n)\cup(y_n')$ is separated% and $d(y_n,y_n'')\to 0$, we may also assume that $(y_n'')\cup (y_n')$ is also separated.
, we can find  $g\in U(Y,F)$ so that 
%\begin{enumerate}
 ${g}(y_n) = b$ and ${g}(y_n') = 0$ for all $n$.
%\item $g_2(y_n'')=b$ and ${g_2}(y_n') = 0$ for all $n$.
%\end{enumerate}
By Proposition \ref{p4.2},
 \[ (T^{-1}g)(x_n) = a_n \text{ and } (T^{-1}g)(x'_n) = 0.\] % = (T^{-1}g_2)(x'_n).\]
Since $T^{-1}g$ is uniformly continuous and $d(x_n,x_n') \to 0$,
\[ a_n = (T^{-1}g)(x_n) - (T^{-1}g)(x_n') \to 0.
\]
%At this point, we consider the cases where $A(X,E) = U(X,E)$ or $U_*(X,E)$.  
%Furthermore, if $A(X,E) = \Lip_*(X,E)$, then the same equation shows that $\|a_n\| = O(d(x_n,x_n'))$.
As $(x_n)$ is separated, $x_n'' \neq x_n$ and $(a_n)$ is a null sequence, we may, after replacing $(x_n)$ and $(x_n'')$ with subsequences, construct a function $f\in U_*(X,E)\subseteq A(X,E)$ so that $f(x_n) = a_n$ and $f(x_n'') = 0$ for all $n$.
Then ${Tf}(y_n) = {g}(y_n) = b$ and $Tf(y_n'') = 0$.
This is impossible since ${Tf}$ is uniformly continuous and $d(y_n,y_n'')\to 0$.
\end{proof}

By Proposition \ref{p6.3.0}, if $T: U(X,E)/U_*(X,E) \to U(Y,F)/U_*(Y,F)$ is a biseparating map, then $\vp:X\to Y$ is a uniform homeomorphism.  
In this case, the map $\wh{T}$ given by 
\[\wh{T}f(x) = Tf(\vp(x)) = \Phi(\vp(x),f(x)) = \wh{\Phi}(x,f(x))\]
maps $U(X,E)/U_*(X,E)$ onto $U(X,F)/U_*(X,F)$, with $\wh{\Phi}:X\times E\to F$ being a function such that $\wh{\Phi}(x,\cdot):E\to F$ is a bijection for each $x\in X$.
To complete the characterization of $T$, it suffices to determine the functions $\wh{\Phi}: X\times E\to F$ so that $x\mapsto \wh{\Phi}(x,f(x))$ belongs to $U(X,F)/U_*(X,F)$ for each $f\in U(X,E)/U_*(X,E)$.
We will refer to this as the ``section problem'' for uniformly continuous functions.

For any $\ep > 0$, define $d_\ep:\ti{X}\times \ti{X}\to [0,\infty]$ by
\begin{equation}\label{eq6.0} d_\ep(a,b)  = \inf\{\sum^n_{i=1}d(x_{i-1},x_i): n\in \N, x_0 = a, x_n = b, d(x_{i-1},x_i) \leq \ep \text{ for all $i$}\},\end{equation}
where we take $\inf \emptyset  = \infty$.
The connection of the $d_\ep$ ``metrics'' with uniformly continuous functions is well known; see, e.g. \cite{A,H, O'F}.
In particular, the first part of the next proposition formalizes the well known principle that uniformly continuous functions are ``Lipschitz for large distances''.

\begin{prop}\label{p6.3}
Let $X$ be a complete metric space and let $E$ be a Banach space.
\begin{enumerate}
\item If $f\in U(X,E)$, then there exist $\ep > 0$ and $C<\infty$ such that 
\[ \|{f}(x_1)-{f}(x_2)\| \leq Cd_\ep(x_1,x_2) \text{ whenever $x_1,x_2\in {X}$, $d(x_1,x_2) > \ep$}.\]
\item If $f:X\to E$ and there exist $\ep >0$ and $C<\infty$ so that 
\[ \|f(x_1) - f(x_2)\| \leq Cd_\ep(x_1,x_2) \text{ for all $x_1,x_2\in X$},\]
then $f\in U(X,E)$.
\end{enumerate}
\end{prop}

\begin{proof}
Statement (2) is trivial since $d_\ep(x_1,x_2) = d(x_1,x_2)$ if $d(x_1,x_2) \leq \ep$.  Let us prove statement (1).
Assume that $f\in U(X,E)$.  There exists $\ep >0$ so that $\|f(a)-f(b)\| \leq 1$  if $d(a,b) \leq \ep$.
Let $x_1,x_2\in X$ be points so that $\ep < d(x_1,x_2)$ and $d_\ep(x_1,x_2) <\infty$.
There are $n\in \N$, $(a_i)^n_{i=0} \subseteq X$ so that $a_0 = x_1$, $a_n = x_2$, $d(a_{i-1}, a_i) \leq\ep$, and 
\[ \sum^n_{i=1}d(a_{i-1}, a_i) \leq 2d_\ep(x_1,x_2).\]
Note that since $d(x_1,x_2) > \ep$, $n \geq 2$.
It is clear that we may assume that $d(a_{i-1},a_i) + d(a_i,a_{i+1}) >\ep$ for $0\leq i< n$.
Thus 
\[ 2d_\ep(x_1,x_2) \geq \sum^n_{i=1}d(a_{i-1}, a_i) \geq \frac{n-1}{2}\,\ep \geq \frac{n\ep}{4}.\]
By choice of $\ep$, $\|f(a_{i-1}) - f(a_i)\| \leq 1$ for all $i$.
Hence 
\[ \|f(x_1)-f(x_2)\| \leq n \leq \frac{8}{\ep}\,d_\ep(x_1,x_2).\]
\end{proof}

%In the rest of the section, let $T$ be a biseparating map from  $U(X,E)$ or $U_*(X,E)$ onto $U(Y,F)$ or $U_*(Y,F)$.  $T$ has a representation given by Proposition \ref{p4.2}.
%The map $\Phi$ can be regarded as a function from the metric space $\bigcup_{y\in \ti{Y}}(\{y\}\times \ti{E}_{\ti{\psi}(y)})\subseteq \ti{Y}\times \ti{E}$ into the Banach space $\ti{F}$.
For the rest of the section, let $\Xi:X\times E\to F$ be a given function and associate with it a mapping $Sf(x) = \Xi(x,f(x))$ for any function $f:X\to E$.
Denote the set of accumulation points in $X$ by $X'$.

\begin{prop}\label{p6.4}
If $Sf \in U(X,F)$ for any $f\in U_*(X,E)$, then  $\Xi$ is continuous at any $(x_0,e_0)\in X'\times E$ .
\end{prop}

\begin{proof}
Assume to the contrary that $\Xi$ is discontinuous at some $(x_0,e_0)\in X'\times E$ .
There are a sequence $((x_n,e_n))^\infty_{n=1} \in X\times E$ converging to $(x_0,e_0)$ and $\ep>0$
so that $\|\Xi(x_n,e_n) - \Xi(x_0,e_0)\| > \ep$ for all $n$.
Replacing $(x_n)$ by a subsequence if necessary, we may assume that either $(x_n)$ is a sequence of distinct points in $X\bs \{x_0\}$ or $x_n = x_0$ for all $n$. 
In the former case, there is a function $f\in U_*(X,E)$ so that $f(x_n) = e_n$ for all $n$ and $f(x_0) = e_0$.
Since $Sf$ is continuous at $x_0$,
\[ \Xi(x_n,e_n) = Sf(x_n)  \to Sf(x_0) = \Xi(x_0,e_0),\]
contrary to the choice of $((x_n,e_n))^\infty_{n=1}$.
Finally, suppose that $x_n = x_0$ for all $n$.
For each $n$, let $f_n$ be the constant function with value $e_n$. Then $Sf_n$ is continuous at $x_0$. Since $x_0$ is an accumulation point, there exists $x'_n$ with $0< d(x'_n,x_0) < \frac{1}{n}$ so that 
\[ \|\Xi(x'_n,e_n) - \Xi(x_0,e_n)\| = \|Sf_n(x'_n) - Sf_n(x_0)\| < \frac{1}{n}.\]
But by the previous case, $\Xi(x'_n,e_n) \to \Xi(x_0,e_0)$.  Thus 
$\Xi(x_0,e_n) \to \Xi(x_0,e_0)$.
\end{proof}

Call a sequence $((x_n,e_n))^\infty_{n=1}\in X\times E$ a $u$-sequence if  $(x_n)$ is a separated sequence and there are $\ep > 0$, $C<\infty$ so that 
\[ \|e_n-e_m\| \leq Cd_\ep(x_n,x_m) \text{ for all $m,n\in\N$}.\]
The importance of $u$-sequences is captured in the next lemma.

\begin{lem}\label{l6.4}
Let $((x_n,e_n))^\infty_{n=1}$ be a $u$-sequence in $X\times E$.  Then there is an infinite subset $N$ of $\N$ and a uniformly continuous function $f:X\to E$ so that $f(x_n) =e_n$ for all $n \in N$.
\end{lem}

\begin{proof}
Let $\ep$ and $C$ be as in the definition above.
If there is an infinite set $N$ in $\N$ so that $d_\ep(x_n,x_m) = \infty$ for all distinct $m,n\in N$, then clearly the function defined by $f(x) = e_n$ if $d_{\ep}(x,x_n) <\infty$ for some $n\in N$ and $f(x) = 0$ otherwise is uniformly continuous. Obviously $f(x_n) =e_n$ for all $n\in N$.
Thus, without loss of generality, we may assume that $d_\ep(x_m,x_n)<\infty$ for all $m,n$.
If $(d_\ep(x_n,x_m))_{m,n}$ is bounded, then $(e_n)$ is a bounded sequence.  Since $(x_n)$  is separated, there exists $f\in U(X,E)$ so that $f(x_n) = e_n$ for all $n$.

Finally, assume that $(d_\ep(x_n,x_m))_{m,n}$ is unbounded set in $\R_+$.
By taking a subsequence, we may assume that $4r_n < r_{n+1}$ for all $n$, where $r_n = d_\ep(x_n,x_1)$.
Define $f:X\to E$ by 
\[ f(x) = \begin{cases}
              e_1 + (1- \frac{2}{r_n}d_\ep(x,x_n))^+(e_n-e_1) &\text{if $d_\ep(x,x_n)< \frac{r_n}{2}$, $n > 1$}\\
              e_1 &\text{otherwise}.
              \end{cases}\]
Using the fact that $\|e_n-e_1\|\leq Cr_n$ for all $n$, one can check that 
\[ \|f(a) - f(b)\| \leq 16Cd_\ep(a,b)
\]
for all $a,b\in X$.  By Proposition \ref{p6.3}(2), $f\in U(X,E)$.  Clearly, $f(x_n) = e_n$ for all $n$, as required.
\end{proof}

\begin{prop}\label{p6.5}
Suppose that $Sf\in U(X,F)$ for all $f\in U(X,E)$.
Let $((x_n,e_n))^\infty_{n=1}$ be a $u$-sequence.  %For any $\eta>0$, there exists $\delta >0$ such that 
Assume that  $((x_n',e_n'))^\infty_{n=1} \in X\times E$, $x_n\neq x_n'$ for all $n$, and $\lim (d(x_n,x_n') + \|e_n-e_n'\|) =0$, then
\[ \lim\|\Xi(x_n,e_n) - \Xi(x_n',e'_n)\| = 0.\]
\end{prop}

\begin{proof}
 It  suffices to show that $\liminf\|\Xi(x_n,e_n) - \Xi(x_n',e'_n)\| = 0$. By Lemma \ref{l6.4}, there exist $f\in U(X,E)$ and an infinite set $N$ in $\N$ so that $f(x_n) = e_n$ for all $n\in N$.
Since $d(x_n,x_n') \to 0$, $\lim_{n\in N}\|e_n -f(x_n')\| = 0$ and hence $\lim_{n\in N}\|e_n'-f(x_n')\| = 0$.
As  $(x_n)$ is a separated sequence and $0 < d(x_n,x_n') \to 0$, we can construct a uniformly continuous function $g:X\to E$ 
such that $g(x_n) = 0$ and $g(x_n') = e_n'-f(x_n')$ for all sufficiently large $n\in N$.
Then  $f+g\in U(X,E)$, 
\[ \Xi(x_n,e_n) = S(f+g)(x_n) \text{ and } \Xi(x_n',e_n') = S(f+g)(x_n')\] for all sufficiently large $n\in N$.
%Hence $S(f+g)(x_n) = \Xi(x_n,e_n)$ and $S(f+g)(x_n') = \Xi(x_n',e_n')$ for all sufficiently large  $n\in N$.
As $S(f+g)\in U(X,F)$ and $d(x_n,x_n') \to 0$, we see that $\lim_{n\in N}\|\Xi(x_n,e_n) - \Xi(x_n',e_n')\| =0$.
\end{proof}

We will say that $\Xi$ is {\em $u$-continuous} if it satisfies the conclusion of Proposition \ref{p6.5}.
We can now solve the section problem for uniformly continuous functions.

\begin{thm}\label{t6.5}
Let $X$ be a complete metric space, $E$ and $F$ be Banach spaces.  Given a function $\Xi:X\times E\to F$,  associate with it a mapping $S$ by $Sf(x) = \Xi(x,f(x))$.
Then $S$ maps $U(X,E)$ into $U(X,F)$ if and only if 
 $\Xi$ is continuous at all $(x_0,e_0) \in X'\times E$ and $\Xi$ is 
 $u$-continuous.
\end{thm}

\begin{proof}
The necessity of the two conditions on $\Xi$ follow from Propositions \ref{p6.4} and \ref{p6.5}.
Conversely, suppose that $\Xi$ is continuous at any $(x_0,e_0)\in X'\times E$ and also $u$-continuous.
Let $f\in U(X,E)$.  
If $Sf\notin U(X,F)$, there are sequences $(x_n)$, $(x_n')$ in $X$, $d(x_n,x'_n)\to 0$, and $\eta >0$ so that
\begin{equation}\label{eq6.1} \|\Xi(x_n,f(x_n)) - \Xi(x_n',f(x_n'))\| = \|Sf(x_n) - Sf(x_n')\| > \eta \text{  for all $n$}.\end{equation}
Suppose that $(x_n)$ has a subsequence that converges to some $x_0\in X$.
We may assume that the whole sequence converges to $x_0$. In particular, $(f(x_n))$ and $(f(x_n'))$ converge to $f(x_0)$.
Clearly, $x_n\neq x'_n$ for all $n$.  Hence $x_0\in X'$.
In this case, (\ref{eq6.1}) violates the continuity of $\Xi$ at $(x_0,f(x_0))$.
Finally, assume that $(x_n)$ is a separated sequence.  Choose $\ep' > 0$ so that $d(x_m,x_n) > \ep'$ if $m\neq n$.  Then let $\ep$  and $C$ be as given in condition (1) of Proposition \ref{p6.3} for the function $f$.
Obviously, $d_\ep \leq d_{\ep \wedge \ep'}$.  So we may assume without loss of generality that $\ep \leq\ep'$.
Hence $(x_n,f(x_n))^\infty_{n=1}$ is a $u$-sequence  by Proposition \ref{p6.3}(1).
Again, (\ref{eq6.1}) implies that $x_n\neq x_n'$ for all $n$.
Furthermore, $d(x_n,x'_n)\to 0$ and $\|f(x_n)-f(x_n)\| \to 0$, the latter as a result of the uniform continuity of $f$.  Therefore, 
\[ \lim\|\Xi(x_n,f(x_n)) - \Xi(x_n',f(x'_n))\| =0\]
by $u$-continuity of $\Xi$, contradicting (\ref{eq6.1}).
\end{proof}

Characterization of biseparating maps from $U(X,E)$ onto $U(Y,F)$ can be obtained by using Theorem \ref{t6.5} together with the ``switch'' from $Y$ to $X$ described prior to Proposition \ref{p6.3}.

\begin{thm}\label{t6.7.1}
Let $X, Y$ be complete metric spaces and let $E,F$ be Banach spaces.
Suppose that $T:U(X,E)\to U(Y,F)$ is a biseparating map.  
Then there are a uniform homeomorphism ${\vp}:{X}\to {Y}$ and a function $\Phi:Y\times E\to F$ so that 
\begin{enumerate}
\item For each $y\in Y$, $\Phi(y,\cdot):E\to F$ is a bijection with inverse $\Psi(x,\cdot):F\to E$, where $\vp(x) = y$.
\item $Tf(y) = \Phi(y,f(\vp^{-1}(y)))$ and $T^{-1}g(x) = \Psi(x,g(\vp(x)))$ for all $f\in U(X,E), g\in U(Y,F)$ and $x\in X$, $y\in Y$. 
\item $\Phi$ is continuous on $Y'\times E$ and $\Psi$ is continuous on $X'\times F$.
\item $(x,e)\mapsto \Phi(\vp(x),e))$ and $(y,e')\mapsto \Psi(\vp^{-1}(y),e'))$ are both  $u$-continuous.
\end{enumerate}
Conversely, assume that $\vp,\Phi$ satisfy conditions (1), (3) and (4).  Define $Tf(y)$ as in (2) for any $f\in U(X,E)$ and $y\in Y$.  Then $T$ is a biseparating map from $U(X,E)$ onto $U(Y,F)$.
\end{thm}

%\begin{proof}
%Let us sketch the proof of the converse.
%It is enough to show that the function $Tf(y) = \Phi(y,f(\vp^{-1}(y)))$ is uniformly continuous on $Y$ for every $f\in U(X,E)$.  For then by symmetry, the function $Sg(x) = \Psi(x,g(\vp(x)))$ is uniformly continuous on $X$ for every $g\in U(Y,F)$.
%It is clear that $T$ and $S$ are mutual inverses and that  both maps are disjointness preserving.

%Now suppose that $f\in U(X,E)$ and assume that there are sequences $(y_n)$, $(y_n')$ in $Y$ and $\ep >0$ so that 
%$d(y_n,y_n') \to 0$ but
%\begin{equation}\label{eq6.2} \|\Phi(y_n, f(x_n)) - \Phi(y_n',f(x_n'))\| > \ep \text{ for all $n$},\end{equation}
%where $\vp(x_n) = y_n$.
%Obviously, $y_n\neq y'_n$ for all $n$.  Thus if $(y_n)$ converges to a point $y_0$, then $y_0\in Y'$.
%Also, $(f(x_n))$ and $(f(x_n'))$ both converge to $f(x_0)$, where $\vp(x_0) = y_0$.
%In this case, (\ref{eq6.2}) contradicts the continuity of $\Phi$ at $(y_0,f(x_0))$.
%Therefore, we may assume that $(y_n)$ has no convergent subsequence, whence it has a separated subsequence, which we may relabel as $(y_n)$.
%Since $f\in U(X,E)$, by Proposition \ref{p6.3}(1), $((x_n,f(x_n)))^\infty_{n=1}$ is a $u$-sequence.
%As $\vp$ is a uniform homeomorphism, $d(x_n,x'_n) \to 0$.  As a result, $\|f(x_n) - f(x'_n)\|\to 0$ also.
%By $u$-continuity of $(x,e)\mapsto \Phi(\vp(x),e)$, we see that 
%\[ \lim\|\Phi(\vp(x_n),f(x_n)) - \Phi(\vp(x_n'),f(x_n'))\| \to 0,\]
%once again contradicting (\ref{eq6.2}).
%\end{proof}

\begin{lem}\label{l6.8}
Let $\Xi:X\times E\to F$ be a given function and associate with it a mapping $Sf(x) = \Xi(x,f(x))$ for any function $f:X\to E$.  If $Sf\in U_*(X,F)$ for any $f \in U_*(X,E)$, then for any separated sequence $(x_n)$ in $X$ and any bounded set $B$ in $E$, there is exists $k\in\N$ so that 
$\bigcup_{n=k}^\infty\Xi(x_n,B)$ is bounded in $F$.
\end{lem}

\begin{proof}
Suppose that $Sf \in U_*(X,F)$  for any $f\in U_*(X,E)$.
Let $(x_n)$ be  a separated sequence in $X$ and let $B$ be a bounded set in $E$.
Assume that for any $k\in \N$, $\bigcup_{n=k}^\infty\Xi(x_n,B)$
is unbounded.
Then there exists $(e_n)$ in $B$ so that $(\Xi(x_n,e_n))^\infty_{n=1}$ is unbounded.
Since $(x_n)$ is separated and $(e_n)$ is bounded, there exists $f\in U_*(X,E)$ so that $f(x_n) = e_n$.
By assumption $Sf$ is bounded.  Hence $(\Xi(x_n,e_n))^\infty_{n=1} = (Sf(x_n))^\infty_{n=1}$ is bounded, a contradiction.
\end{proof}

%The proof of the next theorem is similar to that of Theorem \ref{t6.7.1}. We omit the details.

We now obtain the analog of Theorem \ref{t6.7.1} for biseparating maps between spaces of bounded uniformly continuous functions.  The details are similar to Theorem \ref{t6.7.1}, with the extra ingredient Lemma \ref{l6.8} for ``boundedness''.

\begin{thm}\label{t6.7.2}
Let $X, Y$ be complete metric spaces and let $E,F$ be Banach spaces.
Suppose that $T:U_*(X,E)\to U_*(Y,F)$ is a biseparating map.  
Then there are a uniform homeomorphism ${\vp}:{X}\to {Y}$ and a function $\Phi:Y\times E\to F$ so that 
\begin{enumerate}
\item For each $y\in Y$, $\Phi(y,\cdot):E\to F$ is a bijection with inverse $\Psi(x,\cdot):F\to E$, where $\vp(x) = y$.
\item $Tf(y) = \Phi(y,f(\vp^{-1}(y)))$ and $T^{-1}g(x) = \Psi(x,g(\vp(x)))$ for all $f\in U(X,E), g\in U(Y,F)$ and $x\in X$, $y\in Y$. 
\item $\Phi$ is continuous on $Y'\times E$ and $\Psi$ is continuous on $X'\times F$.
\item $(x,e)\mapsto \Phi(\vp(x),e))$ and $(y,e')\mapsto \Psi(\vp^{-1}(y),e'))$ are both  $u$-continuous.
\item Let $(x_n)$ be a separated sequence in $X$ and $y_n = \vp(x_n)$ for all $n$. If $B$ and $B'$ are bounded sets in $E$ and $F$ respectively, then there exists $k\in \N$  so that 
\[ \bigcup_{n=k}^\infty\Phi(y_n,B) \text{ and } \bigcup_{n=k}^\infty\Psi(x_n,B')\]
arer bounded sets in $F$ and $E$ respectively.
\end{enumerate}
Conversely, assume that $\vp,\Phi$ satisfy conditions (1), (3), (4) and (5).  Define $Tf(y)$ as in (2) for any $f\in U_*(X,E)$ and $y\in Y$.  Then $T$ is a biseparating map from $U_*(X,E)$ onto $U_*(Y,F)$.
\end{thm}

\subsection{Automatic continuity}
Automatic continuity results for biseparating maps acting between spaces of uniformly continuous functions can be deduced easily from the characterization theorems \ref{t6.7.1} and \ref{t6.7.2}.
If $S$ is a subset of $X$, respectively, $Y$, and $f:X\to E$, respectively, $f:Y\to F$, let
\[ \|f\|_S = \sup_{s\in S}\|f(s)\|.\]

\begin{thm}\label{t6.9}
Let $X,Y$ be complete metric spaces and $E,F$ be Banach spaces.
Suppose that $T$ is a biseparating map from $U(X,E)$ onto $U(Y,F)$, respectively, from $U_*(X,E)$ onto $U_*(Y,F)$.  Let $T$ be represented as in  theorems \ref{t6.7.1} or \ref{t6.7.2}.
Assume  that $f\in U(X,E)/U_*(X,E)$ and $S\subseteq X'$, the set of accumulation points of $X$.
For any $\ep > 0$, there exists $\delta >0$ so that  if $g\in U(X,E)/U_*(X,E)$ and $\|g-f\|_S < \delta$, then 
$\|Tg- Tf\|_{\vp(S)} < \ep$.
\end{thm}

\begin{proof}
Suppose that the theorem fails.  There exist $S\subseteq X'$, $\ep >0$ and functions $(g_n)$ in $U(X,E)/U_*(X,E)$ so that 
\[ \|g_n-f\|_S\to 0 \text{ and } \|Tg_n-Tf\|_{\vp(S)} > \ep \text{  for all $n$.}\]
Choose $(x_n) \subseteq S$ so that $\|Tg_n(\vp(x_n)) - Tf(\vp(x_n))\| >\ep$ for all $n$.
Thus
\begin{equation}\label{e6.5} \|\Phi(\vp(x_n), v_n) - \Phi(\vp(x_n),u_n)\| >\ep \text{ for all $n$},\end{equation}
where $v_n = g_n(x_n)$ and $u_n = f(x_n)$.
If $(x_n)$ has a subsequence $(x_{n_k})$ that converges to some $x_0$, then $x_0\in X'$.
Note that $(u_{n_k}) = Tf(\vp(x_{n_k}))$ converges to $u_0 = Tf(\vp(x_0))$ and  $\|v_n  - u_n\| \leq \|g_n-f\|_S\to 0$ as well.  Thus $(v_{n_k})$ converges to $u_0$.
This shows that both sequences $((\vp(x_{n_k}),v_{n_k}))$ and  $((\vp(x_{n_k}),u_{n_k}))$ converge to $(\vp(x_0),u_0)$.
By condition (3) of Theorem \ref{t6.7.1} or \ref{t6.7.2}, $\Phi$ is continuous at $(\vp(x_0),u_0)$, contradicting (\ref{e6.5}).

If $(x_n)$ does  not  have a convergent subsequence, then it has a separated subsequence $(x_{n_k})$. Again, let $u_{n_k} = f(x_{n_k})$
 and $v_{n_k} = g_{n_k}(x_{n_k})$.
Since $g_{n_k}$ and $Tg_{n_k}$ are both continuous, one can choose $x_{n_k}'\neq x_{n_k}$
 so that
 \[ d(x_{n_k}',x_{n_k}), \|g_{n_k}(x'_{n_k}) - v_{n_k}\|, \|Tg_{n_k}(x'_{n_k}) - Tg_{n_k}(x_{n_k})\| \to 0.\]
 Note that the last limit can be stated as 
 \begin{equation}\label{e6.6} \Phi(\vp(x'_{n_k}),g_{n_k}(x'_{n_k})) - \Phi(\vp(x_{n_k}), v_{n_k}) \to 0.\end{equation}
By Proposition \ref{p6.3}(1),  $(x_{n_k}, u_{n_k}))$ is a $u$-sequence.
By (4) of Theorem \ref{t6.7.1} or \ref{t6.7.2}, $(x,e) \mapsto \Phi(\vp(x),e)$ is $u$-continuous.
Since  $x_{n_k}' \neq x_{n_k}$ and 
\[ d(x'_{n_k},x_{n_k}) + \|g_{n_k}(x'_{n_k}) - u_{n_k}\| \leq 
 d(x'_{n_k},x_{n_k}) + \|g_{n_k}(x'_{n_k}) - v_{n_k}\| + \|g_{n_k} - f\|_S
 \to 0,\]
  $u$-continuity gives 
 \begin{equation}\label{e6.7} \Phi(\vp(x'_{n_k}),g_{n_k}(x'_{n_k})) - \Phi(\vp(x_{n_k}), u_{n_k}) \to 0.\end{equation}
The limits (\ref{e6.6}) and (\ref{e6.7}) yield 
\[ \Phi(\vp(x_{n_k}), v_{n_k}) - \Phi(\vp(x_{n_k}), u_{n_k}) \to 0,\]
contrary to (\ref{e6.5}).
\end{proof}

\subsection{Bourbaki boundedness} 
Let $X$ be a metric space.  For any $\ep>0$, recall the ``metric'' $d_\ep$ defined by (\ref{eq6.0}).  $d_\ep$ induces an equivalence relation $\sim_\ep$ on $X$ by $a\sim_\ep b$ if and only if $d_\ep(a,b) < \infty$.
The equivalence classes will be called {\em $\ep$-sets}.  $d_\ep$ is a proper metric (i.e., finite valued) on each $\ep$-set.
$X$ is said to be {\em Bourbaki bounded} if for any $\ep>0$, there are only finitely many $\ep$-sets, each of which is bounded in the $d_\ep$ metric.  See \cite{BG, B, GM}.
A classical result of Atsuji \cite{At} and Hejcman \cite{H}, rediscovered in \cite{O'F}, states that $U(X) = U_*(X)$ if and only if $X$ is Bourbaki bounded.
The final theorem in this section generalizes this result.

\begin{thm}\label{t6.10}
Let $X, Y$ be complete metric spaces and let $E, F$ be Banach spaces.
If there  is a biseparating map from $U(X,E)$ onto $U_*(Y,F)$, then $X$ is Bourbaki bounded.
\end{thm}

Before proceeding to the proof of the theorem, observe that if $X$ is Bourbaki bounded, then $U(X,E) = U_*(X,E)$.  This follows easily from Proposition \ref{p6.3}(2).
 
Let $T:U(X,E)\to U_*(Y,F)$  be a biseparating map.  By Propositions \ref{p4.2} and \ref{p6.3.0}, $T$ has a representation
\[ Tf(y) = \Phi(y,f(\vp^{-1}(y))) \text{ for all $y\in Y$ and all $f\in U(X,E)$},\]
where $\vp$ is a uniform homeomorphism and $\Phi(y,\cdot):E\to F$ is  a bijection for all $y\in Y$.
We may and do assume that $T0 = 0$, so that $\Phi(y,0) = 0$ for all $y$.

\begin{lem}\label{l6.11}
Let $X, Y$ be complete metric spaces and let $E, F$ be Banach spaces.
If there is a biseparating map from $U(X,E)$ onto $U_*(Y,F)$, then for any $\ep>0$, $X$ has finitely many $\ep$-sets.
\end{lem}

\begin{proof}
Suppose that there exists some $\ep >0$ so that $X$ contains an infinite sequence $(X_n)^\infty_{n=1}$ of  $\ep$-sets.
Choose $x_n \in X_n$ for each $n$ and let $y_n =\vp(x_n)$.
Since $\Phi(y_n,\cdot):E\to F$ is a bijection, there exists $e_n\in E$ so that $\|\Phi(y_n,e_n)\| > n$.
Define $f:X\to E$ by $f(x) = e_n$ if $x\in X_n$, $n\in \N$ and $0$ otherwise.
Then $f$ is uniformly continuous but $\|Tf(y_n)\|= \|\Phi(y_n,e_n)\| > n$ for all $n$.
This contradicts the assumption that $Tf \in U_*(Y,F)$.
\end{proof}

\begin{lem}\label{l6.12}
Let $X, Y$ be complete metric spaces and let $E, F$ be Banach spaces.
Suppose that  there is a biseparating map from $U(X,E)$ onto $U_*(Y,F)$.  For any $\ep>0$, any $\ep$-set of $X$ is $d_\ep$-bounded.
\end{lem}

\begin{proof}
Define $\Xi:X\times E\to F$ by $\Xi(x,e) = \Phi(\vp(x),e)$.   The formula $Sf(x) = \Xi(x,f(x))$ defines a biseparating map from $U(X,E)$ onto $U_*(X,F)$ so that $S0 = 0$. For each $x$, $\Xi(x,\cdot):E\to F$ is  a bijection.  Denote its inverse by $\Theta(x,\cdot)$. 
Suppose that there exist $\ep>0$ and an $\ep$-set $X_0$ that is not $d_\ep$ bounded.
Fix $x_0\in X_0$ and a sequence $(x_n)$ in $X_0$ so that $d_\ep(x_{n+1},x_0) > 3d_\ep(x_n,x_0)$ for all $n$.
Let $a$ be a nonzero vector in $E$. 
By Proposition \ref{p6.3}(2), the function $f:X\to E$ given by $f(x) = d_\ep(x,x_0)a$ belongs to $U(X,E)$.
Hence $Sf\in U_*(X,F)$.
In particular, the sequence $(b_n) = (Sf(x_n))$ is bounded in $F$.

\medskip
\noindent{Claim}. There exists $m\in \N$ so that 
\[ \limsup_n\|\Theta(x_n,sb_n) - \Theta(x_n,tb_n)\| \leq 1 \text{ if $s,t\in [0,1]$, $|s-t| \leq \frac{1}{m}$}.\]

First suppose that the claim holds.
Then
\begin{align*}
 \limsup_n\|\Theta(x_n,b_n)\| & =  \limsup_n\|\Theta(x_n,0)-\Theta(x_n,b_n)\|
 \\&\leq \sum^m_{k=1}\limsup\|\Theta(x_n,\frac{k-1}{m}\,b_n) - \Theta(x_n,\frac{k}{m}\,b_n)\| \leq m.
\end{align*}
However, $\Xi(x_n,d_\ep(x_n,x_0)a) = Sf(x_n) = b_n$ for all $n$.
Hence $\Theta(x_n,b_n) = d_\ep(x_n,x_0)a$ for all $n$.
In particular, $(\Theta(x_n,b_n))$ cannot be bounded, contradicting the preceding inequality.

To complete the proof of the lemma, let us verify the claim.
If the claim fails, for each $m\in \N$, one can find $s_m,t_m\in [0,1]$, $|s_m-t_m|\leq \frac{1}{m}$, so that
\[ \limsup_n\|\Theta(x_n,s_mb_n) - \Theta(x_n,t_mb_n)\| > 1.\]
We may assume that $(s_m), (t_m)$ both converge to some $t_0\in [0,1]$.
Without loss of generality, 
\[  \limsup_n\|\Theta(x_n,s_mb_n) - \Theta(x_n,t_0b_n)\| > \frac{1}{2} \text{ for all $m$.}\]
Choose $n_1 < n_2 < \cdots$ so that 
\begin{equation}\label{e6.4}\|\Theta(x_{n_m},s_mb_{n_m})  - \Theta(x_{n_m},t_0b_{n_m})\| > \frac{1}{2} \text{ for all $m$.}\end{equation}
Clearly, $(x_n)$, and hence $(x_{n_m})$, is a separated sequence by choice. Since $(t_0b_{n_m})$ is   bounded, there exist  $g_1\in U_*(X,F)$ so that 
$g_1(x_{n_m}) = t_0b_{n_m}$ for all $m$.
If there exists $\delta > 0$ and an infinite set $M$ so that $B(x_{n_m},\delta) = \{x_{n_m}\}$ for all $m\in M$, then each $\{x_{n_m}\}$ is a $\delta$-set in $X$, contradicting Lemma \ref{l6.11}.
Therefore, there is a sequence $(x'_m)$ in $X$ so that $0 < d(x_{n_m},x'_m)\to 0$.
Note that 
$\|s_mb_{n_m} - t_0b_{n_m}\| \to 0$.  
Hence there exists $h \in U_*(X,F)$ so that $h(x_{n_m}) = s_mb_{n_m} - t_0b_{n_m}$ and $h(x'_m) = 0$ for all sufficiently large $m$.
Set $g_2 = g_1 +h$.
Since $g_1(x'_m) = g_2(x'_m)$, $S^{-1}g_1(x'_m) = S^{-1}g_2(x'_m)$ for all sufficiently large $m$.  Thus
\begin{align*}
\|\Theta(x_{n_m},&\, t_0b_{n_m}) - \Theta(x_{n_m},s_mb_{n_m})\| = \|S^{-1}g_1(x_{n_m}) -  S^{-1}g_2(x_{n_m})\|  \\
&\leq \|S^{-1}g_1(x_{n_m}) -  S^{-1}g_1(x'_m)\| + \|S^{-1}g_2(x'_{m}) -  S^{-1}g_2(x_{n_m})\|.
\end{align*}
As $S^{-1}g_1, S^{-1}g_2$ are uniformly continuous functions and $d(x_{n_m},x'_m) \to 0$, both terms on the right of the inequality tend to $0$.  So we have reached a contradiction with (\ref{e6.4}).
This completes the proof of the claim and hence of the lemma.
\end{proof}

Lemmas \ref{l6.11} and \ref{l6.12} prove Theorem \ref{t6.10}.
If $T:U(X,E)\to U_*(Y,F)$ is a biseparating map, then $X$ is Bourbaki bounded by Theorem \ref{t6.10}.
Hence  $U(X,E) = U_*(X,E)$ \cite{O'F}.
Thus the characterization  Theorem \ref{t6.7.2} applies.

%%%%%%%%%%%%%%%%%%%%%%%%%
\section{Spaces of Lipschitz functions}\label{s8}

We focus on biseparating maps on Lipschitz spaces in this section.  Again, we restrict consideration to complete metric spaces $X$, $Y$ and Banach spaces $E, F$.
In contrast to previous cases, we will see that there is no difference between spaces of bounded and unbounded Lipschitz functions.  In fact, it is even sufficient to consider bounded metric spaces $X$ and $Y$.
Indeed, if $(X,d)$ is a metric space, let  $X_1$ be the set $X$ with the bounded metric $d\wedge 1$.
Then clearly $\Lip_*(X,E) = \Lip(X_1,E)$.
To see that $\Lip(X,E)$ is equivalent to some  $\Lip(Z,E)$ for a bounded metric space $Z$ via a linear biseparating map, we employ essentially the same argument from  \cite[Proposition 5.2]{LT}, which has its roots in \cite{W}.
Fix a distinguished point $e$ in $X$ and define a function $\xi:X\to \R$ by $\xi(x) = d(x,e)\vee 1$.
Denote the Lipschitz constant of a function $f$ by $L(f)$. Let $d':X\times X\to \R$ be given by
\[ d'(p,q) = \sup_{\stackrel{f\in \Lip(X)}{L(f),|f(e)| \leq 1}}\bigl|\frac{f(p)}{\xi(p)} - \frac{f(q)}{\xi(q)}\bigr|,\]
where $\Lip(X)$ is the space of all real-valued Lipschitz functions on $X$.

\begin{lem}\cite[Proposition 5.1]{LT}\label{p6.1}
\begin{enumerate}
\item $d'$ is a  metric on $X$ that is bounded above by $4$.
\item 
\[ \frac{d(p,q)}{\xi(p)\vee \xi(q)}\leq d'(p,q) \leq  \frac{3d(p,q)}{\xi(p)\vee \xi(q)}\]
for all $p,q\in X$.
\item If $\xi(p) \leq \xi(q)$, then
\[ d'(p,q) \leq d'(p,q)\xi(p) \leq 3d(p,q).\]
\item If $X$ is complete with respect to the metric $d$, then it is complete with respect to the metric $d'$.
\end{enumerate}
\end{lem}

Let $Z$ be the set $X$ with the metric $d'$.

\begin{prop}\label{p6.2}
$f\in \Lip(X,E)$ if and only if $\frac{f}{\xi} \in \Lip(Z,E)$.
In particular, $T:\Lip(X,E)\to \Lip(Z,E)$, $Tf = \frac{f}{\xi}$, is a linear biseparating map.
\end{prop}

\begin{proof}
%Since $Z$ is a bounded metric space, $\Lip(X_1,E) = \Lip_*(X_1,E)$.  
The second assertion follows easily from the first.
Suppose that $f\in \Lip(X,E)$.  Set $c = L(f)\vee \|f(e)\| \vee 1$.
For any $x^* \in E^*$, $\|x^*\| \leq 1$,  $g = (x^*\circ f)/c\in \Lip(X)$ and $L(g), |g(e)| \leq 1$.
By definition of $d'$,
\[ cd'(p,q) \geq c\bigl|\frac{g(p)}{\xi(p)} - \frac{g(q)}{\xi(q)}\bigr| = \bigl|x^*\bigl(\frac{f}{\xi}(p) - \frac{f}{\xi}(q)\bigr)\bigr|.
\]
Taking supremum over $\|x^*\|\leq 1$ shows that $\frac{f}{\xi}\in \Lip(Z,E)$ with Lipschitz constant at most $c$.

Conversely, suppose that $g= \frac{f}{\xi}\in \Lip(Z,E)$.  Let $p,q$ be distinct points in $X$ so that $\xi(p) \leq \xi(q)$.  Denote the Lipschitz constant of $g$ with respect to the metric $d'$ by $L'(g)$.
Then
\[ \|g(q)\| \leq \|g(q) - g(e)\| + \|g(e)\| \leq L'(g)d'(q,e) + \|g(e)\| \leq 4L'(g) + \|g(e)\|\]
since $d' \leq 4$.
Hence
\begin{align*}
\|f(p) - f(q)\| & \leq \|g(p) - g(q)\|\xi(p) + \|g(q)\|(\xi(q) - \xi(p))\\
&\leq L'(g)d'(p,q)\xi(p) + (4L'(g)+ \|g(e)\|)d(p,q)\\
&\leq (7L'(g)+\|g(e)\|)d(p,q) \text{ by Lemma \ref{p6.1}(3)}.
\end{align*}
Thus $f\in \Lip(X,E)$.
\end{proof}

\subsection{$\vp$ is a Lipschitz homeomorphism} 

In view of the above, throughout the rest of  this section, $X$ and $Y$ will be assumed to be bounded complete metric spaces.  Let $T:\Lip(X,E)\to \Lip(Y,F)$ be a biseparating map so that $T0 =0$.
Once again, we have  a representation (Proposition \ref{p4.2})
\begin{equation}\label{e7.1}Tf(y) = \Phi(y,f(\vp^{-1}(y))) \text{ for all $y\in Y$ and all $f\in \Lip(X,E)$},\end{equation}
where $\vp:X\to Y$ is a homeomorphism and $\Phi:Y\times E\to F$ 
is a function such that $\Phi(y,\cdot):E\to F$ is a bijection for all $y\in Y$.
Denote the inverse of $\Phi(y,\cdot)$ by $\Psi(x,\cdot)$, where $\vp(x) = y$.

\begin{prop}\label{p7.3}
Suppose that $(x_n), (x'_n)$ are sequences  in $X$.  Let $y_n = \vp(x_n), y_n' = \vp(x_n')$ for all $n$.
If $(y_n)$ is a separated sequence, then there exists $C<\infty$ so that $d(y_n,y_n') \leq Cd(x_n,x_n')$ for all $n$.
\end{prop}

\begin{proof}
Assume that the proposition fails.  There are sequences as in the statement of the proposition so that $d(y_n,y_n')/d(x_n,x_n') \to \infty$.  Since $Y$ is bounded, $d(x_n,x_n') \to 0$.
If $(y_n')$ has a convergent subsequence, then $(x_n')$, and hence $(x_n)$ has a convergent subsequence, which in turn implies that $(y_n)$ has  a convergent subsequence, contrary to the choice of $(y_n)$.
Thus, by taking a subsequence if necessary, we may assume that both $(y_n)$ and $(y_n')$ are separated sequences.  Fix a nonzero vector $a\in E$ and let $g = T(1\otimes a)$.

\medskip

\noindent\underline{Case 1}.  $d(y_n,y_n') \not\to 0$.

In this case, by taking a subsequence, we may assume that $(y_n)\cup (y_n')$ is a separated set.
Since $g\in \Lip(Y,F)$, $(g(y_n))$ is a bounded sequence in $F$.
Hence there exists $h\in \Lip(Y,F)$ so that $h(y_n) = g(y_n)$ and $h(y_n') = 0$ for all large $n$.
Then $T^{-1}h(x_n) = a$ and $T^{-1}(x_n') = 0$ for all large $n$.
Since $T^{-1}h$ is Lipschitz and $d(x_n,x_n')\to 0$, we have a contradiction.

\medskip

\noindent\underline{Case 2}.  $d(y_n,y_n')\to 0$.

%Let $g$ be the same function as in Case 1. 
If $(\frac{\|g(y_n)\|}{d(y_n,y_n')})$ is bounded, then there is a function $h\in \Lip(Y,F)$ so that 
$h(y_n) = g(y_n)$ and $h(y_n') = 0$ for all large $n$.
Thus $T^{-1}h(x_n) = a$ and $T^{-1}h(x_n') = 0$.  This is impossible since $T^{-1}h$ is Lipschitz and $d(x_n,x_n') \to 0$.

From the unboundedness of $(\frac{\|g(y_n)\|}{d(y_n,y_n')})$, we may assume without loss of generality that for each  $n$, there exists  $s_n$ so that $d(y_n,y'_n)\leq s_n< 2d(y_n,y_n')$ and that $k_n = \frac{\|g(y_n)\|}{s_n}\in \N$.
Since $\Phi(y_n,a) = g(y_n)$, $\Psi(x_n,g(y_n)) = a$.  Thus
\[ \sum^{k_n}_{j=1}[\Psi(x_n,\frac{jg(y_n)}{k_n}) - \Psi(x_n,\frac{(j-1)g(y_n)}{k_n})] = \Psi(x_n,g(y_n)) = a.
\]
Hence there exists $j_0\in \{1,\dots, k_n\}$ so that 
\[ \|\Psi(x_n,\frac{j_0g(y_n)}{k_n}) - \Psi(x_n,\frac{(j_0-1)g(y_n)}{k_n})\| \geq \frac{\|a\|}{k_n}.\]
Therefore, there exists $i_0\in \{j_0-1,j_0\}$ so that 
\begin{equation}\label{e7.2.0} \|\Psi(x_n, \frac{i_0g(y_n)}{k_n}) - \Psi(x'_n, \frac{j_0g(y_n')}{k_n})\| \geq \frac{\|a\|}{2k_n}.\end{equation}
Now
\begin{align}\label{e7.2}
\|\frac{i_0g(y_n)}{k_n} &- \frac{j_0g(y_n')}{k_n}\| \leq \frac{\|g(y_n)\|}{k_n}  + \frac{j_0}{k_n}\|g(y_n)-g(y_n')\| \\ \notag
&\leq s_n + \frac{j_0L(g)}{k_n}d(y_n,y_n') \leq (2+L(g))d(y_n,y_n'),
\end{align}
where $L(g)$ is the Lipschitz constant of $g$.
Since  $(y_n)$ is separated and $(\frac{i_0g(y_n)}{k_n})$ is a bounded sequence in $F$, there exists $h_1\in \Lip(Y,F)$ so that $h_1(y_n) = \frac{i_0g(y_n)}{k_n}$ for all $n$.
Let $L(h_1)$ be the Lipschitz constant of $h_1$.  By (\ref{e7.2}),
\begin{align*} \|\frac{j_0g(y_n)}{k_n} - h_1(y_n')\|& \leq \|\frac{j_0g(y_n)}{k_n} - h_1(y_n)\| + L(h_1)d(y_n,y_n')
\\&\leq (2+L(g)+L(h_1))d(y_n,y_n').
\end{align*}
  Therefore, 
one can construct a function $h_2\in \Lip(Y,F)$ so that 
\[ h_2(y_n) = 0 \text{ and } h_2(y_n') = \frac{j_0g(y_n')}{k_n}- h_1(y_n') \text{ for all large $n$}.\]
Let $f= T^{-1}(h_1+h_2)$.
Then $Tf(y_n) = h_1(y_n)$ and hence $f(x_n) =\Psi(x_n,  \frac{i_0g(y_n)}{k_n})$ for all large $n$.
Similarly, $f(x_n') = \Psi(x_n, \frac{j_0g(y_n)}{k_n})$ for all large $n$.
Note that $g$ is a bounded function.  Set $\|g\|_\infty = \sup_{y\in Y}\|g(y)\|$. By (\ref{e7.2.0}, 
\[ \|f(x_n) -f(x_n')\| \geq \frac{\|a\|}{2k_n} \geq \frac{\|a\|s_n}{2\|g\|_\infty} \geq \frac{\|a\|}{2\|g\|_\infty}d(y_n,y_n')\] for all large $n$.
Since $f$ is Lipschitz, it follows that $d(y_n,y_n')/d(x_n,x'_n)\not\to \infty$.
This contradiction completes the proof of the proposition.
\end{proof}

If $(x_0,e_0)\in X\times E$, $C<\infty$ and $r >0$, let
\begin{equation}\label{e7.4} \Delta(x_0,e_0,C,r) = \{(x,e)\in X\times E:  d(x,x_0) \leq r, \|e-e_0\| \leq Cd(x,x_0)\}.\end{equation}
Then set $\Delta(x_0,e_0,C) = \bigcup_{r > 0} \Delta(x_0,e_0,C,r)$.
It is not surprising that  understanding the map $T$ depends on analyzing the sets $\Delta(x_0,e_0,C,r)$ and $\Delta(x_0,e_0,C)$.  For a very special instance, see \cite[Section 7.2]{AZ}.
Define sets in $Y\times F$ in a similar manner.
Let $M:X\times E\to Y\times F$ be the function
\[M(x,u) = (\vp(x),\Phi(\vp(x),u)).\]
Then $M$ is a bijection.  Moreover, if $f \in \Lip(X,E)$, then $M(x_0, f(x_0)) = (\vp(x_0), Tf(\vp(x_0)))$.

Suppose that $\vp:X\to Y$ is not Lipschitz.  There are sequences $(x_n), (x'_n)$   in $X$, $x_n\neq x_n'$ for all $n$, so that taking  $y_n = \vp(x_n), y_n' = \vp(x_n')$, we have  $d(y_n,y_n')/d(x_n,x_n') \to \infty$.
Since $Y$ is bounded, $d(x_n,x_n') \to 0$.  By Proposition \ref{p7.3}, $(y_n)$ cannot be a separated sequence.
Taking a subsequence if necessary, we may assume that $(y_n)$ converges to some $y_0$.
Then $(x_n)$ converges to $\vp^{-1}(y_0) = x_0$. The same must hold for $(x'_n)$.  Therefore, $(y_n')$ also converges to $y_0$.
With further subsequences and relabeling the primed and unprimed terms if necessary, we may assume that $d(y_n,y_0) \leq d(y_n',y_0)$ for all $n$.
With this assumption, $y_n'\neq y_0$ for all $n$.  For otherwise $y_n = y_0 = y_n'$, which implies that $x_n = x'_n$, contrary to their choice.  Hence we may further assume that 
\[ 2d(y'_{n+1},y_0) <d(y_n',y_0) \text{ and } 2d(y'_{n+1},y_0)  < d(y_n,y_0) \text{ if $y_n\neq y_0$}.\]
%$y_m \notin B(y_n',\frac{d(y_n',y_0)}{3})$ if $m < n$.

\begin{prop}\label{p7.5}
Let $v_0\in F$ and let $u_n\in E$ be determined by $M(x_n,u_n) = (y_n,v_0)$.  There exists $m=m(v_0)\in \N$ so that for any $n\geq m$, if   $(y_n',v_n') \in \Delta(y_n,v_0,1)$, then 
\[ (x_n',u_n') = M^{-1}(y_n',v_n')\in \Delta(x_n,u_n,m).\] 
\end{prop}

\begin{proof}
Otherwise,  there are $n_m \geq m$, $(y_{n_m}',v_{n_m})\in \Delta(y_{n_m},v_0,1)$ so that 
$(x_{n_m}',u_{n_m}') = M^{-1}(y_{n_m}',v_{n_m}')\notin \Delta(x_{n_m},u_{n_m},m)$ for all $m
$.
%By Proposition \ref{p7.4}, $(n_m)$ cannot be bounded in $\N$.  Thus w
We may assume $n_m\uparrow \infty$.
For simplicity, relabel $n_m$ as $m$.
Thus
\[\vp(x_{m}) = y_{m}, \vp(x_{m}') = y_{m}', \Phi(y_{m}, u_{m}) = v_0, \Phi(y_{m}',u_m') = v_m'\]
and 
\[ \|v_m'-v_0\| \leq d(y_m,y_m'), \|u_{m}'-u_{m}\| > m d(x_{m}',x_{m}) \text{ for all $m$.}\]
Let $d_m = d(y_m',y_m)$.  
Define a function $g:Y\to F$ by
\[ g(y) = \begin{cases}
               v_0 + (1- \frac{4d(y,y_m')}{d_m})(v_m'-v_0) &\text{if $m\in \N$, $d(y,y_m') < \frac{d_m}{4}$}\\
               v_0 &\text{otherwise}.
               \end{cases}\]
From the disjointness of the balls $B(y_m',\frac{d_m}{4})$ and the inequality $\|v_m'-v_0\| \leq d_m$ for all $m$,
we see that $g\in \Lip(Y,F)$.

Next, we claim that $y_n \notin B(y_m',\frac{d_m}{4})$ for any $m,n\in \N$.
Indeed, this is obvious if $m =n$.
Note that \[ d_m \leq d( y_m',y_0) + d(y_0,y_m) \leq 2d(y_m',y_0).\]
If $m < n$, then 
\[ d(y_n,y_0) \leq d(y_n',y_0) < \frac{d(y_m',y_0)}{2}.
\]
Hence
\[ d(y_n,y_m') \geq d(y_m',y_0) - d(y_n,y_0)> \frac{d(y_m',y_0)}{2} \geq \frac{d_m}{4}.
\]
%Thus $y_n \notin B(y_m',\frac{d_m}{4})$.
On the other hand, if $m > n$ and $y_n\neq y_0$, then
\[2d(y_m',y_0) \leq 2d(y_{n+1}',y_0) < d(y_n,y_0).
\]
Hence
\[ d(y_n,y_m') \geq d(y_n,y_0) - d(y_m',y_0) \geq d(y_m',y_0) \geq \frac{d_m}{2}.
\]
Finally, if $y_n = y_0$, then $d(y_n,y_m') = d(y_0,y_m')  \geq \frac{d_m}{2}$.
Thus $y_n \notin B(y_m',\frac{d_m}{4})$ in all cases.

Obviously, $g(y_m') = v_m'= \Phi(y_m',u_m')$. 
From the fact that $y_n \notin B(y_m',\frac{d_m}{4})$ for all $m,n$,
 $g(y_m) = v_0 = \Phi(y_m,u_m)$ for all $n$.
Therefore, $T^{-1}g(x_m) = u_m$ and $T^{-1}g(x'_m) = u_m'$ for all $m$.
But $\|u_{m}'-u_{m}\| > m d(x_{m}',x_{m})$, which contradicts the fact that $T^{-1}g$ is Lipschitz.
\end{proof}

We are now ready to prove the main result regarding the homeomorphism $\vp$.

\begin{thm}\label{t7.5}
Let $T:\Lip(X,E)\to \Lip(Y,F)$ be a biseparating map, where $X, Y$ are complete bounded metric spaces and $E, F$ are Banach spaces.  
In the notation of (\ref{e7.1}), $\vp:X\to Y$ is a Lipschitz homeomorphism.
\end{thm}

\begin{proof}
If $\vp$ is not a Lipschitz function, then we obtain sequences $(x_n), (x_n')$, $(y_n)$ and $(y_n')$ as in the discussion before Proposition \ref{p7.5}.
For each $v\in F$, determine $m = m(v) \in \N$ by Proposition \ref{p7.5}.
Set $F_k = \{v\in F: m(v) \leq k\}$ for each $k\in \N$.
Then $F = \bigcup_{k=1}^\infty\ol{F_k}$.
By the Baire Category Theorem, there are an open ball $O$ in $F$ and $k_0\in \N$ so that $O\subseteq \ol{F_{k_0}}$.
Pick distinct points $a,b\in O\cap F_{k_0}$.  Since $d_n = d(y_n,y_n') \to 0$, we may assume without loss of generality that $\|a-b\| > d_n$ for all $n$. For each $n$, choose $k_n\in \N$ so that 
\[ \frac{k_nd_n}{2} \leq \|a-b\| < {k_nd_n}.\]
%where $d_n = d(y_n,y_n')$. 
Note  that $a + \frac{j}{k_n}(b-a) \in O$, $0\leq j\leq k_n$. By making  small perturbations, one can find  $w_{nj}\in O\cap F_{k_0}$, $0\leq j\leq k_n$, so that $w_{n0} = a$, $w_{nk_n} = b$ and $\|w_{nj} - w_{n,j-1}\|$ is sufficiently close to $\frac{\|a-b\|}{k_n}$ so as to make it $< d_n$.
Now
\begin{align*}
\|\sum^{k_n}_{j=1}[\Psi(x_n&,w_{nj}) - \Psi(x_n,w_{n,j-1})]\| = \|\Psi(x_n,v_1) -\Psi(x_n,v_0)\|
\\ & = \|T^{-1}(1\otimes v_1)(x_n) - T^{-1}(1\otimes v_0)(x_n)\|\\
& \to \|T^{-1}(1\otimes v_1)(x_0) - T^{-1}(1\otimes v_0)(x_0)\|\\
& = \|\Psi(x_0,b) -\Psi(x_0,a)\| = c.
\end{align*}
Since $\Psi(x_0,\cdot)$ is a bijection, $c > 0$.
For all $n$, there exists $1\leq j_n\leq k_n$ so that 
\[ \|\Psi(x_n,w_{n,j_n}) - \Psi(x_n,w_{n,j_n-1})\| > \frac{c}{2k_n}.\]
Now choose $i_n\in \{j_n-1,j_n\}$ so that, setting 
\[ u_n = \Psi(x_n,w_{n,j_n}) \text{ and } u_n' =  \Psi(x_n',w_{n,i_n}),\]
we have $\|u_n-u_n'\| > \frac{c}{4k_n}$.
Note that 
\[ M^{-1}(y_n,w_{n,j_n}) = (\vp^{-1}(y_n),\Psi(\vp^{-1}(y_n), w_{n,j_n})) %= (x_n,\Psi(x_n, w_{n,j_n})) 
=
(x_n,u_n).
\]
Similarly, $M^{-1}(y_n',w_{n,i_n}) =(x_n',u_n')$.
By choice,
\[ \|w_{n,i_n} - w_{n,j_n}\| \leq \|w_{n,j_n-1} - w_{n,j_n}\| < d_n = d(y_n',y_n).\]
Hence $(y_n', w_{n,i_n}) \in \Delta(y_n,w_{n,j_n},1)$.
Since $w_{n,j_n} \in F_{k_0}$, $m=m(w_{n,j_n}) \leq k_0$.
By definition of $m(w_{n,j_n})$, this implies that for all $n \geq k_0$,
\[ (x_n',u'_n) \in \Delta(x_n,u_n,m)\subseteq \Delta(x_n,u_n,k_0),
\]
which in turns yields that 
$\|u_n'-u_n\| \leq k_0d(x_n',x_n).$
Therefore,
\[ \frac{c}{2\|a-b\|}\cdot d_n \leq  \frac{c}{4k_n} < \|u_n - u_n'\| \leq  k_0d(x_n',x_n).\]
Since this holds for all sufficiently large $n$, and $d_n = d(y_n,y_n')$, it contradicts the assumption that $d(y_n,y_n')/d(x_n,x'_n) \to \infty$.

This completes the proof that $\vp$ is a Lipschitz function.  By symmetry, so is $\vp^{-1}$.  Hence $\vp$ is a Lipschitz homeomorphism.
\end{proof}

\subsection{Section problem for Lipschitz functions}

Let  $X$ be a complete bounded metric space and let $E,F$ be Banach spaces.  Consider a given function $\Xi:X\times E\to F$. Define $M:X\times E \to X\times F$ by $M(x,e) = (x,\Xi(x,e))$. Recall the sets $\Delta(x,e,C,r)$ and $\Delta(x,e,C)$ in $X\times E$ as given by (\ref{e7.4}).  Similar definitions apply in $X\times F$.  Theorem \ref{t7.7} solves the section problem for spaces of Lipschitz functions.  For a very special case, refer to \cite[Theorem 7.1]{AZ}.

\begin{lem}\label{l7.6}
Suppose that $x_0\in X$, $u_0\in E$ and $C<\infty$.
Let $(x^n_1), (x^n_2)$ in $X$, $(u^n_1), (u^n_2)$ in $E$  be sequences so that $x^n_1 \neq x^n_2$ for all $n$,
\[ \|u^n_1-u^n_2\|\leq Cd(x^n_1,x^n_2),\
\|u^n_i-u_0\| \leq Cd(x^n_i,x_0) \text{ $i =1,2$,  $n\in \N$, and}\]
$\lim_n d(x^n_i,x_0)=0$, $i =1,2$.
Then there exists $f\in \Lip(X,E)$ so that $f(x^n_i) = u^n_i$, $i =1,2$, for infinitely many $n$.
\end{lem}

\begin{proof}
Set $r^n_i = d(x^n_i,x_0)$.  There is no loss of generality in assuming that $r^n_2 \geq r^n_1$ for all $n$. After taking subsequences, we may divide the proof into the following cases.

\medskip

\noindent \underline{Case 1}.  $x^n_1 = x_0$,  i.e., $r^n_1 = 0$ for all $n$.

\medskip

Note that in this case $u^n_1 = u_0$ for all $n$.  Since $x^n_2 \neq x^n_1$ and $r^n_2=d(x^n_2,x_0) \to 0$, we may further assume that $r^{n+1}_2 < \frac{1}{3}r^n_2$ for all $n$, which implies that the balls $B(x^n_2,\frac{r_n}{2})$ are pairwise disjoint.
Define $f:X\to E$ by
\[ f(x) = \begin{cases}
u_0 + (1 - \frac{2d(x,x^n_2)}{r_n})(u^n_2-u_0) &\text{if $d(x,x^n_2) < \frac{r_n}{2}$}\\
u_0 &\text{otherwise}.
\end{cases}\]
Then it can be checked that $f\in \Lip(X,E)$, $f(x^n_2)  = u^n_2$, $f(x^n_1) = f(x_0) = u_0=u^n_1$ for all $n$.

\medskip

\noindent \underline{Case 2}.  $r^n_1 >0$ and there exists $c > 0$ so that $d(x^n_1,x^n_2) \geq cr^n_2$ for all $n$.

\medskip

We may of course assume that $0 < c < 1$.  The assumptions imply that the balls $B(x^n_1, \frac{cr^n_1}{2})$ and $B(x^n_2, \frac{cr^n_2}{2})$ are disjoint.
Since $r^n_2\to 0$, we may further assume that $B(x^{n+1}_1, \frac{cr^{n+1}_1}{2})\cup B(x^{n+1}_2, \frac{cr^{n+1}_2}{2})$ is disjoint from $\bigcup^n_{k=1}[B(x^{k}_1, \frac{cr^{k}_1}{2})\cup B(x^{k}_2, \frac{cr^{k}_2}{2})]$ for any $n$.
As a result, the sets $B(x^n_1, \frac{cr^n_1}{2})$, $B(x^n_2, \frac{cr^n_2}{2})$, $n\in \N$, are all mutually disjoint.
Define $f:X\to E$ by
\[ f(x) = \begin{cases}
u_0 + (1 - \frac{2d(x,x^n_i)}{cr^n_i})(u^n_i-u_0) &\text{if $d(x,x^n_i) < \frac{cr^n_i}{2}$, $n \in \N$, $i=1,2$,}\\
u_0 &\text{otherwise}.
\end{cases}\]
Then it can be checked that $f\in \Lip(X,E)$, $f(x^n_i)  = u^n_i$, $i=1, 2$, $n\in \N$.

\medskip

\noindent \underline{Case 3}.  $r^n_1 >0$ for all $n$ and  $d(x^n_1,x^n_2)/r^n_2 \to 0$.

\medskip

As in Case 2, we may assume that the sets $B(x^n_2, \frac{r^n_2}{2}), n\in \N$, are dsijoint.
In this instance, we may further assume that $B(x^n_1,d(x^n_1,x^n_2)) \subseteq B(x^n_2, \frac{r^n_2}{2})$ for all $n$.
Define $g:X\to E$ by
\[ g(x) = \begin{cases}
 (1 - \frac{2d(x,x^n_2)}{r^n_2})(u^n_2-u_0) &\text{if $d(x,x^n_2) < \frac{r^n_2}{2}$}\\
0 &\text{otherwise}.
\end{cases}\]
Since $\|u^n_2-u_0\| \leq Cd(x^n_2,x_0)= Cr^n_2$ for all $n$, $g\in \Lip(X,E)$ and has Lipschitz constant at most $2C$.
Clearly, $g(x^n_2) = u^n_2-u_0$ for all $n$.
Now let 
$h:X\to E$ be given by
\[ h(x) = \begin{cases}
 (1 - \frac{d(x^n_1,x)}{d(x^n_1,x^n_2)})(u^n_1-u_0-g(x^n_1)) &\text{if $d(x^n_1,x) < d(x^n_1,x^n_2)$}\\
0 &\text{otherwise}.
\end{cases}\]
Note that
\[\|u^n_1-u_0-g(x^n_1)\| \leq \|u^n_1-u^n_2\| + \|g(x^n_2)-g(x^n_1)\| \leq 3Cd(x^n_1,x^n_2).
\]
Taking into account the disjointness of the sets $B(x^n_1, d(x^n_1,x^n_2))$, it follows that $h\in \Lip(X,E)$.
Furthermore, 
\[ 
(g+h)(x^n_1) = u^n_1-u_0 \text{ and } (g+h)(x^n_2) = g(x^n_2) = u^n_2-u_0.\]
Finally, the function $f(x) = g(x) +h(x) +u_0$ is the one we seek.
\end{proof}

\begin{thm}\label{t7.7}
Let $\Xi:X\times E\to F$ be a given function.  Define $Sf(x) = \Xi(x,f(x))$ for any function $f:X\to E$.
Suppose that $Sf$ belongs to $\Lip(X,F)$ for all $f\in \Lip(X,E)$.  Then
%Let $(x_n)_{n\in I}$ be a finite  or an infinite separated sequence in $X$. 
\begin{enumerate}
\item If $(x_n)$ is a  separated sequence in $X$, and $B$ is a bounded set in $E$, then there is a finite set $N \subseteq \N$ so that $\bigcup_{n\notin N}\Xi(x_n,B)$ is bounded.
\item Suppose that $x_0\in X$, $u_0\in E$ and $C<\infty$.  There exist $r >0$ and $D<\infty$ so that 
\[ \|\Xi(x_1,u_1) - \Xi(x_2,u_2)\| \leq Dd(x_1,x_2) 
\]
whenever $\|u_1-u_2\| \leq Cd(x_1,x_2)$, $\|u_i-u_0\| \leq Cd(x_i,x_0)$ and  $d(x_i,x_0) \leq r$, $i=1,2$.
\item Let $(x_n)$ be a separated sequence in $X$ and $(u_n)$ be a bounded sequence in $E$.
For any $C<\infty$, there exist $r>0$ and $D<\infty$ so that 
\[ \|\Xi(x_n',u'_n) - \Xi(x_n,u_n)\| \leq Dd(x_n',x_n) \text{ for all $n$}\]
whenever  $\|u_n'-u_n\| \leq Cd(x_n',x_n)$  and $d(x_n',x_n) \leq r$ for all $n$.
\end{enumerate}
Conversely, suppose that conditions (1), (2) and (3)  hold. Then $Sf\in \Lip(X,F)$ for any $f\in \Lip(X,E)$.
\end{thm}

\begin{proof}
Suppose that $Sf\in \Lip(X,F)$ for any $f\in \Lip(X,E)$.  
Let $(x_n)$ be a separated sequence in $X$ and $B$ be a bounded set in $E$.  If $\bigcup_{n\notin N}\Xi(x_n,B)$ is unbounded for any finite set $N\subseteq \N$, there exists a  sequence $(u_n)\subseteq B$ so that $(\Xi(x_n,u_n))$ is unbounded. Since $(x_n)$ is separated and $(u_n)$ is bounded, there is a Lipschitz function $f:X\to E$ so that $f(x_n) = u_n$ for all $n$.
Then $Sf\in \Lip(X,F)$, $(\Xi(x_n,u_n)) = (Sf(x_n))$ is bounded in $F$, a contradiction.  This proves condition (1).

Suppose that condition (2) fails.  Then there are $(x^n_1), (x^n_2)$ in $X$, $(u^n_1), (u^n_2)$ in $E$  so that 
$\|u^n_1-u^n_2\| \leq Cd(x^n_1,x^n_2)$, 
$\|u^n_i-u_0\| \leq Cd(x^n_i,x_0)$ and $d(x^n_i,x_0) \leq \frac{1}{n}$, $i =1,2$, $n\in \N$, but $\|\Xi(x^n_1,u^n_1) - \Xi(x^n_2,u^n_2)\| > nd(x^n_1,x^n_2)$.
In particular, the last inequality implies that $x^n_1\neq x^n_2$ for all $n$.
Apply Lemma \ref{l7.6} to find a function $f\in \Lip(X,E)$ so that $f(x^n_i) = u^n_i$ for infinitely many $n$.
Let $L$ be the Lipschitz constant of $Sf$.
Then 
\[\|\Xi(x^n_1,u^n_1) - \Xi(x^n_2,u^n_2)\| = \|Sf(x^n_1) - Sf(x^n_2)\| 
\leq Ld(x^n_1,x^n_2)\]
for all $n$, contrary to  their choices.

Let $(x_n)$ be a separated sequence in $X$ and let $(u_n)$ be a bounded sequence in $E$.  Assume that condition (3) fails for a constant $C$.
For each $k\in \N$, there exist $n_k\in \N$, $x_k\in X $ and $u_k'\in E$ so that $\|u_k'- u_{n_k}\| \leq Cd(x_k',x_{n_k})$ and $d(x_k',x_{n_k}) < \frac{1}{k}$ but
\begin{equation}\label{e7.4.1}\|\Xi(x_k',u_k') - \Xi(x_{n_k},u_{n_k})\| > kd(x_k',x_{n_k}).\end{equation}
If $(n_k)$ has a constant subsequence, then, say, $x_{n_k} = x_0$ and $u_{n_k} = u_0$ for infinitely many $k$.
In this case, we have a contradiction to condition (2), which has been shown above.
Otherwise, we may assume that $n_k \uparrow \infty$.
Since $(x_{n_k})$ is separated and $(u_{n_k})$ is bounded, there exists $g\in \Lip(X,E)$ so that $g(x_{n_k}) = u_{n_k}$ for all $k$.
Let $L$ be the Lipschitz constant of $g$.  We have
\[ \|u_k'- g(x_k')\| \leq \|u_k'-u_{n_k}\| + \| u_{n_k} - g(x_k')\| \leq (C+L)d(x_{n_k},x_k').\]
As $(x_{n_k})$ is separated and $d(x_k',x_{n_k})\to 0$, we can find $h\in \Lip(X,E)$ so that $h(x_{n_k}) = 0$ and $h(x_k') = u_k'-g(x_k')$ for all large $k$.  Let $f = g+h\in \Lip(X,E)$.
Then $Sf\in \Lip(X,E)$ and 
\[ Sf(x_{n_k}) = \Xi(x_{n_k},u_{n_k}),\ Sf(x_k') = \Xi(x_k',u_k').\]
Thus (\ref{e7.4.1}) leads to a contradiction.

Conversely, suppose that conditions (1) - (3) hold.
Let $f\in \Lip(X,E)$ with Lipschitz constant $C$.
First, let us show that $Sf$ is a bounded function.
If not, there is a sequence $(z_n)\in X$ so that $\|Sf(z_n)\| \to \infty$.
By condition (1), $(z_n)$ cannot have a separated subsequence.
Hence we may assume that $(z_n)$ converges to some $z_0\in X$.
Then  $\|f(z_n) - f(z_0)\| \leq Cd(z_n,z_0)$ and $d(z_n,z_0) \to 0$.
Applying condition (2) with $(x_0,u_0) = (x_1,u_1) = (z_0,f(z_0))$ and $(x_2,u_2) = (z_n,f(z_n))$, we obtain $D<\infty$ so that 
\[ \|\Xi(z_0,f(z_0)) - \Xi(z_n,f(z_n))\| \leq Dd(z_0,z_n) \text{ for all sufficiently large $n$}.\]
Hence $(Sf(z_n)) = (\Xi(z_n,f(z_n))$ is surely bounded, contrary to its choice.

Now suppose that  $Sf\notin \Lip(X,F)$. There are sequences $(x_n)$, $(x_n')$ in $X$ so that 
\begin{equation}\label{e7.5} \|\Xi(x_n,u_n) - \Xi(x_n',u_n')\| = \|Sf(x_n) - Sf(x_n')\|> nd(x_n,x_n') \text{ for all $n$},
\end{equation}
where $u_n = f(x_n)$ and $u_n' = f(x_n')$.
Since $Sf$ is a bounded function, we must have $d(x_n,x_n') \to 0$.
By using subsequences, we may assume that either $(x_n)$ converges to some $x_0$ or that $(x_n)$ is a separated sequence.
In the former case, since $d(x_n,x_0), d(x_n',x_0)\to 0$, $\|f(x_n) - f(x_n')\|\leq Cd(x_n,x_n')$,
$\|f(x_n)-f(x_0)\| \leq Cd(x_n,x_0)$ and $\|f(x_n')-f(x_0)\| \leq Cd(x'_n,x_0)$, it follows from condition (2) that there exists $D<\infty$ so that for all sufficiently large $n$,
\[ \|\Xi(x_n,u_n) - \Xi(x_n',u_n')\| = \|\Xi(x_n,f(x_n)) - \Xi(x_n',f(x_n'))\| \leq Dd(x_n,x_n'),\]
contrary to (\ref{e7.5}).
The proof is similar in case $(x_n)$ is a separated sequence, using condition (3) instead.
\end{proof}

The next theorem is easily deduced from Theorem \ref{t7.7}, keeping in mind that $\vp:X\to Y$ is a Lipschitz homeomorphism.

\begin{thm}\label{t7.8}
Let $X,Y$ be complete bounded metric spaces and let $E, F$ be Banach spaces.
Suppose that $T:\Lip(X,E) \to \Lip(Y,F)$ is a biseparating map.  
Then there are a Lipschitz homeomorphism $\vp:X\to Y$ and a function $\Phi:Y\times E\to F$ so that 
\begin{enumerate}
\item For each $y\in Y$, $\Phi(y,\cdot):E\to F$ is a bijection with inverse $\Psi(x,\cdot):F\to E$, where $\vp(x) =y$.
\item $Tf(y) = \Phi(y,f(\vp^{-1}(y)))$ and  $T^{-1}g(x) = \Psi(x,g(\vp(x)))$ for all $f\in \Lip(X,E)$, $g\in \Lip(Y,F)$ and $x\in X$, $y\in Y$.
\item Let $(x_n)$ be a separated sequence in $X$. For any bounded sets $B$ in $E$ and $B'\in F$, there is a finite set $N \subseteq \N$ so that $\bigcup_{n\notin N}\Phi(\vp(x_n),B)$  and $\bigcup_{n\notin N}\Psi(x_n,B')$ are bounded.

\item Suppose that $x_0\in X$, $u_0\in E$, $v_0\in F$ and $C<\infty$.  There exist $r >0$ and $D<\infty$ so that 
\[
 \|\Phi(\vp(x_1),u_1) - \Phi(\vp(x_2),u_2)\|,  \|\Psi(x_1,v_1) - \Psi(x_2,v_2)\|\leq Dd(x_1,x_2) 
\]
whenever $\|u_1-u_2\|, \|v_1-v_2\| \leq Cd(x_1,x_2)$, $\|u_i-u_0\|, \|v_i-v_0\| \leq Cd(x_i,x_0)$ and  $d(x_i,x_0) \leq r$, $i=1,2$.
\item  Let $(x_n)$ be a separated sequence in $X$ and $(u_n), (v_n)$ be bounded sequences in $E$ and $F$ respectively.
For any $C<\infty$, there exist $r>0$ and $D<\infty$ so that 
\[ \|\Phi(\vp(x_n'),u'_n) - \Phi(\vp(x_n),u_n)\|,\ 
\|\Psi(x_n',v'_n) - \Psi(x_n,v_n)\| \leq Dd(x_n',x_n)\]
for all $n$,
whenever  $\|u_n'-u_n\|, \|v_n'-v_n\| \leq Cd(x_n',x_n)$ and  $d(x_n',x_n) \leq r$
for all $n$.
 \end{enumerate}
 Conversely, if $\vp$, $\Phi$ satisfy conditions (1)-(5) and $T$ is defined by (2), then $T$ is a biseparating map from $\Lip(X,E)$ onto $\Lip(Y,F)$.
\end{thm}

\subsection{A property of Lipschitz sections}

Let $X$ be a bounded metric space and let $E$ and $F$ be Banach spaces.
Theorem \ref{t7.7} characterizes the ``section maps'' $\Xi:X\times E\to F$ so that $Sf(x) = \Xi(x,f(x))$ is Lipschitz whenever $f\in \Lip(X,E)$.
An example in \cite[p.~190]{AZ}, where $X= [0,1]$ with the H$\ddot{\text{o}}$lder metric $d(x,y)= |x-y|^\al$, $0<\al < 1$, and $E = F = \R$, shows that for a given $x\in X$, the function $\Xi(x,\cdot):E\to F$ need not be continuous.
Nevertheless, in this subsection, we will show that if $x$ is an accumulation point of $X$, then there is a dense open set $O$ in $E$ so that  $\Xi(x,\cdot)$ is continuous on $O$.
Let $\Xi:X\times E\to F$ be a ``Lipschitz section''.  Taking $(x_1,u_1) = (x,u)$ and $(x_2,u_2) = (x_0,u_0)$ in Theorem \ref{t7.7}(2) yields the next lemma.

%We begin with a lemma which is essentially the same as Proposition \ref{p7.4}.

\begin{lem}\label{l7.10}
Let $(x_0,u_0) \in X\times E$ and let $v_0 = \Xi(x_0,u_0)$.
For any $C<\infty$, there exists $n = n(x_0,u_0,C) \in\N$ so that if $(x,u)\in X\times E$,
$\|u-u_0\| \leq Cd(x,x_0)$ and $d(x,x_0) < \frac{1}{n}$, then 
\[ \|v-v_0\|\leq nd(x,x_0), \text{ where $v = \Xi(x,u)$}.\]
\end{lem}

%\begin{proof}
%Otherwise, there is a sequence $((x_n,u_n))_n\subseteq X\times E$ 
%so that 
%\[ \|u_n-u_0\| \leq Cd(x_n,x_0),\ d(x_n,x_0) \to 0 \text{ and } \|v_n-v_0\| > n d(x_n,x_0),\]
%where $v_n = \Xi(x_n,u_n)$.
%Obviously $x_n\neq x_0$ for all $n$.  In Lemma \ref{l7.6}, take $(x^n_1,u^n_1) = (x_n,u_n)$ and $(x^n_2,u^n_2) = (x_0,u_0)$.
%The lemma yields a function $f\in \Lip(X,E)$ so that $f(x_n) = u_n$ and $f(x_0) = u_0$ for infinitely many $n$.
%Let $L$ be the Lipschitz constant of $Sf$.
%Then 
%\begin{align*}  n d(x_n,x_0) &< \|v_n-v_0\| = \|\Xi(x_n,u_n) - \Xi(x_0,u_0)\| \\&= \|Sf(x_n) - Sf(x_0)\| \leq Ld(x_n,x_0)
%\end{align*}
%for infinitely many $n$, which is absurd.
%\end{proof}

\begin{thm}\label{t7.11}
Let $x_0$ be an accumulation point of $X$.
There is a dense open set $O$ in $E$ so that $\Xi(x_0,\cdot)$ is continuous on $O$.
\end{thm}

\begin{proof}
In the notation of Lemma \ref{l7.10}, for each $n\in \N$, let 
\[ A_n = \{u_0\in E: n(x_0,u_0,1) \leq n\}.\]
By the lemma,  $E= \bigcup_n \ol{A_n}$.
Since $E$ is a complete metric space, $O = \bigcup_n \inte\ol{A_n}$ is a dense open set in $E$.
To complete the proof of the theorem, let us show that $\Xi(x_0,\cdot)$ is continuous on $O$.
Clearly, it suffices to show that $\Xi(x_0,\cdot)$ is continuous on each $\inte\ol{A_n}$.
Fix $N\in \N$.  Suppose that $(u_n)$ is a sequence in $\inte\ol{A_N}$ converging to $u_0 \in \inte\ol{A_N}$.

\medskip

\noindent\underline{Claim}.  There is a sequence $(u_n')$ in $A_N$ so that 
\[ \|u_n'-u_n\| , \|\Xi(x_0,u_n') - \Xi(x_0,u_n)\| \to 0.\]

\medskip

Consider a given $n\in \N$.  Since $\Xi(x,u_n)$ is a Lipschitz function of $x$ and $x_0$ is an accumulation point, there exists $x\in X$ so that $0< Nd(x,x_0) < \frac{1}{n}$ and that $\|\Xi(x,u_n) -\Xi(x_0,u_n)\| < \frac{1}{n}$.
As $u_n\in \ol{A_N}$, there exists $u_n'\in A_N$ so that $\|u_n'-u_n\| \leq d(x,x_0) < \frac{1}{nN}$.
Note that  $n(x_0,u_n',1) \leq N$.  Hence the condition $\|u_n-u_n'\| \leq d(x,x_0)< \frac{1}{N}$ implies
\[ \|\Xi(x,u_n) - \Xi(x_0,u_n')\| \leq Nd(x,x_0)< \frac{1}{n}.\]
Therefore,
\begin{align*}
\|\Xi(x_0,u_n)- &\Xi(x_0,u_n')\| \\&\leq  \|\Xi(x_0,u_n)- \Xi(x,u_n)\| + \|\Xi(x,u_n)- \Xi(x_0,u_n')\|\\
& < \frac{2}{n}.
\end{align*}
This completes the proof of the claim.

\medskip

In view of the claim, in order to prove the continuity of $\Xi(x_0,\cdot)$ at $u_0$, it suffices to show that $
\Xi(x_0,u_n') \to \Xi(x_0,u_0)$.
Let $\ep > 0$ be given.  As before, one can choose $x'$ so that $0< d(x',x_0) <  \frac{1\wedge \ep}{N}$ and that 
$\|\Xi(x',u_0) - \Xi(x_0,u_0)\| < \ep$.
For all sufficiently large $n$, $\|u_n'-u_0\| < d(x',x_0)$.
Once again, $\|u_0-u_n'\| \leq d(x',x_0) < \frac{1}{N}$ implies
\[ \|\Xi(x',u_0) - \Xi(x_0,u_n')\| \leq Nd(x',x_0) < \ep.\]
Therefore,
\begin{align*}
\|\Xi(x_0,u_n')- &\Xi(x_0,u_0)\| \\&\leq  \|\Xi(x_0,u_n')- \Xi(x',u_0)\| + \|\Xi(x',u_0)- \Xi(x_0,u_0)\|
 < 2\ep
\end{align*}
for all sufficiently large $n$.
\end{proof}

\section{Comparisons}\label{s9}

We close with some results comparing different types of spaces under nonlinear biseparating maps.
Throughout this section, $X,Y$ will be complete metric spaces and $E$, $F$ will be Banach spaces.

\begin{prop}\label{p8.1}
Let $T:A(X,E)\to \Lip(Y,F)$ be a biseparating map, where $Y$ is bounded.  If  $A(X,E) = U(X,E)$, then $X$ is separated.  If  $A(X,E) = U_*(X,E)$, then  both $X$  and $Y$ are  separated.
\end{prop}

\begin{proof}
Normalize $T$ by taking $T0 = 0$. Suppose that $A(X,E)$ is either $U(X,E)$ or $U_*(X,E)$. First assume, if possible, that there is a convergent sequence $(x_n)$ in $X$ consisting of distinct points.
Let $x_0$ be its limit, which we may assume to be distinct from all $x_n$'s.
Set $y_n = {\vp}(x_n)$, $n\in \N\cup\{0\}$, and $r_n = d(y_n,y_0)$, $n\in \N$.
Since $r_n \to 0$, without loss of generality, we may further assume that $r_{n+1} < \frac{r_n}{3}$ for all $n\in \N$.
Fix a nonzero vector $b\in F$.  For each $m\in\N$,  define $g_m: Y\to F$ by 
\[ g_m(y) = \begin{cases}
\bigl(1- \frac{2d(y,y_n)}{r_n}\bigr)mr_nb &\text{if $d(y,y_n) < \frac{r_n}{2}$, $n\in \N$,}\\
0 &otherwise.
\end{cases}\]
Then  $g_m\in \Lip(Y,F)$, ${g_m}(y_n) = mr_nb$ for all $n\in \N$ and ${g_m}(y_0) =0$.
By Proposition \ref{p4.2}, ${T^{-1}g_m}(x_0) =0$.
By continuity of $g_m$, there is an increasing sequence $(n_m)$ so that $T^{-1}g_m(x_{n_m}) \to 0$.
Thus, there is a function $f\in U_*(X,E)\subseteq A(X,E)$ so  that $f(x_{n_m}) = T^{-1}g_m(x_{n_m})$ for all $m\in \N$ and $f(x_0) =0$.
By Proposition \ref{p4.2}, 
\[ {Tf}(y_{n_m}) = {g_m}(y_{n_m}) = mr_{n_m}b\text{ and } {Tf}(y_0) = 0.\]
However, $Tf$ is Lipschitz on $Y$.
We have reached a contradiction since 
\[ \|{Tf}(y_{n_m}) - {Tf}(y_0)\| = mr_{n_m}b = md(y_{n_m},y_0)b.\]
This shows that $X$ does not contain any nontrivial convergent sequence.

If $X$ is not separated, there are points $x_n,x_n'\in X$ so that $0<d(x_n,x'_n)\to 0$. 
Let $y_n = {\vp}(x_n)$ and $y_n' = {\vp}(x_n')$.
Since ${\vp}$ is uniformly continuous by Proposition \ref{p6.3.0}, $d(y_n,y_n')\to 0$.
If $(y_n)$ has a  subsequence that converges in ${Y}$, then $(x_n)$ has a convergent subsequence. By the previous paragraph,  $(x_n)$ has a constant subsequence, which in turn implies that $(x'_n)$ has a nontrivial convergent subsequence, contrary to the last paragraph.  Thus, we may replace $(y_n)$ by a subsequence if necessary to assume that it is separated. 
As $(y_n)$ is separated and $d(y_n,y_n') \to 0$,
it is possible to choose ${g_m}\in \Lip(Y,F)$ so that ${g_m}(y_n)  = md(y_n,y_n')b$ and $g_m(y_n') =0$ for all $n$.
As before, we can find an increasing sequence $(n_m)$ so that $(T^{-1}g_m)(x_{n_m})\to 0$.
Then we can construct $f\in U_*(X,E)$ so that $f(x_{n_m}) = (T^{-1}g_m)(x_{n_m})$ and $f(x_{n_m}') = 0$ for all $m$.
By Proposition \ref{p4.2}, 
\[{Tf}(y_{n_m}) = {g_m}(y_{n_m}) = md(y_{n_m},y'_{n_m}) \text{ and } {Tf}(y_{n_m}') = 0.
\]
Once again, this contradicts with the fact that $Tf$ is Lipschitz.

Now if $A(X,E) = U_*(X,E)$, we show that $Y$ is also separated.
 By Theorem \ref{t3.5}, there is a homeomorphism ${\vp}:X\to {Y}$.
In particular, ${Y}$ must be discrete.  
If $Y$ is not separated, there are sequences $(y_n), (y_n')$ in $Y$ so that $0 < d(y_n,y_n')\to 0$.
Since $(y_n)$ cannot have a convergent subsequence in ${Y}$, we may assume that it is a separated sequence.
By taking a further subsequence, we may assume that $(y_n)\cup (y_n')$ consists of distinct points.
Let $x_n = {\vp}^{-1}(y_n)$ and $x'_n = {\vp}^{-1}(y_n')$. Then $(x_n)\cup (x_n')$ consists of distinct points.
Fix a nonzero vector $b\in F$ and let $a_n = T^{-1}(1\otimes b)(x_n)$. Then $(a_n)$ is a bounded sequence in  $E$.
Since $X$ is separated, there is a function $f\in U_*(X,E)$ so that $f(x_n) = a_n$ and $f(x_n') = 0$ for all $n$.
By Proposition \ref{p4.2}, $Tf(y_n) = b$ and $Tf(y_n') = 0$ for all $n$.
This is impossible since $Tf$ is uniformly continuous.
\end{proof}

\begin{thm}\label{t6.6}
Assume that $Y$ is bounded.
A map $T:U(X,E)\to \Lip(Y,F)$ is biseparating if and only if $X$ and $Y$ are finite sets of the same cardinality and there are a bijection $\psi:Y\to X$ and bijections $\Phi(y,\cdot):E\to F$ for each $y\in Y$ so that $Tf(y) = \Phi(y,f(\psi(y)))$ for all $f\in U(X,E)$ and all $y\in Y$.
\end{thm}

\begin{proof}
Assume that $T$ is biseparating.  Represent $T$ as in  Proposition  \ref{p4.2}.  Similarly, $T^{-1}$ has a representation ${T^{-1}g}(x) = \Psi(x,{g}({\vp}(x)))$.
By Proposition  \ref{p8.1}, $X$ is separated.  
Suppose that  $y\in Y$. Let  $\psi(y) = x\in {X}$.
If $b\in F$, then $\Psi(x,b) = T^{-1}(1\otimes b)(x) \in E$ and 
\[ b= T(T^{-1}(1\otimes b))(x) = \Phi(y,T^{-1}(1\otimes b)(x)) = \Phi(y,\Psi(x,b)).\]
Take an arbitrary function  $g:Y\to F$.  
Define $f:X\to E$ by $f(x) = \Psi(x,g(\vp(x)))$.
Since $X$ is separated, $f\in U(X,E)$.
Therefore, $Tf\in \Lip(Y,F)$.
By the above, for any $y\in Y$, $Tf(y) = g(y)$.
This shows that any function $g:Y\to F$ belongs to $\Lip(Y,F)$.  Clearly, this implies that $Y$ must be a finite set.
The remaining statements of the theorem follows easily from the representations of $T$ and $T^{-1}$.
The converse is clear.
\end{proof}

\begin{thm}\label{t6.7}
Assume that $Y$ is bounded.
A map $T:U_*(X,E)\to \Lip(Y,F)$ is biseparating if and only if 
\begin{enumerate}
\item $X$ and $Y$ are separated metric spaces.
\item There are a bijection $\psi:Y\to X$ and bijections $\Phi(y,\cdot):E\to F$, $y\in Y$, so that  
\begin{enumerate}
\item $Tf(y) = \Phi(y,f(\psi(y)))$ for all $f\in U_*(X,E)$ and all $y\in Y$.
\item For any bounded sets $B_1$ in $E$ and $B_2$ in $F$, there are a finite set $Y_0$ in $Y$ and a bounded set $B_3$ in $E$ so that 
\[ \bigcup_{y\notin Y_0}\Phi(y,B_1) \text{ is bounded in $F$ and } B_2 \subseteq \bigcap_{y\notin Y_0}\Phi(y,B_3).\] 
\end{enumerate}
\end{enumerate}
\end{thm}

\begin{proof}
Assume that $T:U_*(X,E)\to \Lip(Y,F)$ is biseparating.   By Proposition \ref{p6.4}, $X$ and $Y$ are both separated.  
Therefore, for either direction of the theorem, $X$ and $Y$ are separated.  
In this case, since $Y$ is assumed to be bounded, $U_*(X,E) = C_*(X,E)$ and $\Lip(Y,F) = C_*(Y,F)$.
Thus the problem reduces to proving that if $X$ and $Y$ are separated, then $T:C_*(X,E)\to C_*(Y,F)$ is biseparating if and only if condition (2) of the theorem holds.
Biseparating maps $T:C_*(X,E)\to C_*(Y,F)$ have been characterized in Theorem \ref{t5.5}.  
Note that presently, as $X$ and $Y$ are separated, a map $\vp:X\to Y$ is a homeomorphism if and only if it is a bijection.  Also, condition (2) of Theorem \ref{t5.5} is vacuous.
Thus, it remains to show that when $X$ and $Y$ are separated, condition (3) of Theorem \ref{t5.5} is equivalent to 
condition 2(b) above.

Let $B_1$ and $B_2$ be bounded sets in $E$ and $F$ respectively.  Since $X$ is separated, condition (3) of Theorem \ref{t5.5}is equivalent to the fact that every $(x_n) \in \prod_nZ_n(B_1,B_2)$ has a constant subsequence.
Note that  $Z_m(B_1,B_2) \supseteq Z_n(B_1,B_2)$ if $m \leq n$.
Therefore, said condition is satisfied if and only if there exist $n_0$ such that $Z_{n_0}(B_1,B_2)$ is finite.
If $Z_{n_0}(B_1,B_2)$ is finite, let $Y_0 = \vp(Z_{n_0}(B_1,B_2))$.
Then $Y_0$ is  a finite set and $y = \vp(x) \notin Y_0$ implies
Now 
\[
x\notin Z_{n_0}(B_1,B_2) \iff \Phi(y,B_1) \subseteq n_0B_F \text{ and } B_2\subseteq \Phi(y, n_0B_E).\]
Hence condition  2(b) is satisfied with $B_3 = n_0B_E$.
Conversely, if condition 2(b) holds.  Let $n_0$ be such that 
\[ \bigcup_{y\notin Y_0}\Phi(y,B_1) \subseteq n_0B_F \text{ and } B_2 \subseteq \bigcap_{y\notin Y_0}\Phi(y,n_0B_E).\] 
Then $\vp(x) \not\in Y_0$ implies $x\notin Z_{n_0}(B_1,B_2)$.  Therefore, $Z_{n_0}(B_1,B_2)$ is finite.
\end{proof}


\begin{thebibliography}{100}

\bibitem{A}
J.~Araujo, {\em Realcompactness and spaces of vector-valued continuous functions}, Fund.~Math.\ {\bf 172} (2002), 27-40.


\bibitem{A2} J.~Araujo, {\em Realcompactness and Banach-Stone theorems}, Bull.~Belg.~Math.~Soc.~Simon Stevin {\bf 11}(2004), 247-258.


\bibitem{A3} J.~Araujo, {\em Linear biseparating maps between spaces of vector-valued differentiable
functions and automatic continuity}, Adv.~Math.\ {\bf 187}(2004), 488-520.


\bibitem{A4}
 J.~Araujo, {\em The noncompact Banach-Stone theorem}, J.~Operator Theory {\bf55:2}(2006), 285-294.




\bibitem{ABN} J.~Araujo, E.~Beckenstein and L.~Narici, {\em Biseparating maps and realcompactifications}, J.~Math.~Anal.~Appl.\ {\bf 192}(1995), 258-265.


\bibitem{AD}  J.~Araujo and L.~Dubarbie, {\em Biseparating maps between Lipschitz function spaces},
 J.~Math.~Anal.~Appl.\ {\bf 357}(2009), 191-200.


\bibitem{AJ} J.~Araujo and K.~Jarosz, {\em Separating maps on spaces of continuous functions}, 
Function spaces (Edwardsville, IL, 1998), Contemp.~Math.\ {\bf 232}, 33-37, Amer.~Math.~
Soc., Providence, RI, 1999.


\bibitem{AJ2} J.~Araujo and K.~Jarosz, {\em Biseparating maps between operator algebras}, J.~Math.~Anal.~Appl.\ {\bf 282}(2003),  48-55.


\bibitem{AZ}
J.~Appell and P.~Zabrejko, Nonlinear superposition operators, Cambridge Tracts in Mathematics 95, 1990.


\bibitem{At} M.~Atsuji, {\em Uniform continuity of continuous functions of metric spaces}, Pacific J.~Math.\ {\bf 8}(1958), 11–16.




\bibitem{BG} G.~Beer and M. I.~Garrido, {\em Bornologies and locally Lipschitz functions}, Bull.~Aust.~Math.~Soc. {\bf 90}(2014), 257-263.



\bibitem{B} N.~Bourbaki, Elements of mathematics, general topology, Part 1, Hermann, Paris, 1966.

\bibitem{BBH} K.~Boulabiar, G.~Buskes and M.~Henriksen, {\em A generalization of a theorem on biseparating maps}, J.~Math.~Anal.~Appl.\ {\bf 280}(2003), 334-349.

\bibitem{GaJ} M. I.~Garrido and J.~A.~Jaramillo, {\em A Banach-Stone theorem for uniformly continuous functions}, Monatsch.~Math.\ {\bf 131}(2000),189-192.

\bibitem{GM} M. I.~Garrido and A.S.~Mero\~{n}o, {\em New types of completeness in metric spaces}, Ann.~Acad.~Sci.~Fennicae {\bf 39}(2014), 733-758.

\bibitem{GJW} H.-W.~Gau, J.-S.~Jeang and N.-C.~Wong, {\em Biseparating linear maps between continuous
vector-valued function spaces}, J.~Aust.~Math.~Soc. {\bf 74}(2003), no. 1, 101-109.


\bibitem {GK} {I.\ Gelfand and A.\ Kolmogorov}, On rings of continuous
functions on topological spaces, Dokl.\ Akad.\ Nauk.\ SSSR \textbf{22} (1939),
11 -- 15.

\bibitem{GJ}
L.~Gillman and M.~Jerison, Rings of continuous functions, Graduate Texts in
Mathematics, No. 43. Springer-Verlag, New York-Heidelberg, 1976.


\bibitem{H} J.~Hejcman, {\em Boundedness in uniform spaces and topological groups},  Czechoslovak Math.~J.\ {\bf 9}(1959), 544–563.

\bibitem{HBN} S.~Hernandez, E.~Beckenstein and L.~Narici, {\em Banach-Stone theorems and separating maps}, Manuscripta Math.\ {\bf 86}(1995), 409-416.

\bibitem{J} K.~Jarosz, {\em Automatic continuity of separating linear isomorphisms}, Bull.~Canad.~Math.~Soc.\ {\bf 33}(1990), 139-144.

\bibitem{JL} J.-S.~Jeang and Y.-F.~Lin, {\em Characterizations of disjointness preserving operators
on vector-valued function spaces}, Proc.~Amer.~Math.~Soc.\ {\bf 136}(2008), no. 3, 947-954.

\bibitem{JW} J.-S.~Jeang and N.~C. Wong, {\em Weighted composition operators on $C_0(X)$'s}, J.~Math.~Anal.~Appl.\ {\bf 201}(1996), 981-993.

\bibitem{J-V} A.~Jimenez-Vargas, {\em Linear bijections preserving the Holder seminorm}, Proc.~Amer.~Math.~Soc.\ {\bf 135}(2007), 2539-2547.


\bibitem{J-V2} A.~Jimenez-Vargas, {\em Disjointness preserving operators between little Lipschitz algebras},
J.~Math.~Anal.~Appl.\ {\bf  337}(2008), 984-993.


%\bibitem{J-V V-V} A.~Jimenez-Vargas and M.~Villegas-Vallecillos, {\em Order isomorphisms of little
%Lipschitz algebras}, Houston J. Math., ???

\bibitem {J-V W} A.~Jimenez-Vargas and Y.-S.~Wang, {\em Linear biseparating maps between vector-valued little Lipschitz function spaces}, Acta Math.~Sinica {\bf 26}(2010), 1005-1018.


\bibitem {K} {I.\ Kaplansky}, {\em Lattices of continuous functions},
Bull.\ Amer.\ Math.\ Soc. \textbf{53} (1947), 617 -- 623.

\bibitem{KM} H.~K$\ddot{\text{o}}$nig and V.~Milman, {\em Operator functional equations in analysis}, in Asymptotic geometric analysis, M.~Ludwig et al.~(eds.), Fields Institute Communications 68(2013), 189-209.

\bibitem{L} D.~H.~Leung, {\em Biseparating maps on generalized Lipschitz spaces}, Studia Math.\ {\bf 196}(2010), 23-40.

\bibitem{LT} D.~H.~Leung and W.-K.~Tang,
{\em Nonlinear order isomorphisms on function spaces}, Dissertationes Math.\ {\bf 517}.



\bibitem{O'F} A.~G.~O'Farrell, {\em When uniformly-continuous implies bounded}, Irish Math.\ Soc.\ Bull.\ {\bf 53} (2004), 53-56.


\bibitem {W} N.~Weaver, Lipschitz algebras, World Scientific, Singapore, 1999.


\end{thebibliography}
\end{document}